\documentclass[notitlepage,11pt,reqno]{amsart}  
\usepackage[foot]{amsaddr}
\usepackage[numbers,sort]{natbib}
\usepackage{amssymb,nicefrac,bm,upgreek,mathtools,verbatim}
\usepackage[final]{hyperref}
\usepackage[mathscr]{eucal}
\usepackage{dsfont} 
\usepackage{graphicx}
\usepackage{yfonts}
\usepackage[margin=10pt,font=small,labelfont=bf]{caption}
				
\usepackage[margin=1in]{geometry} 
\usepackage{babel}
\allowdisplaybreaks  
\usepackage{color}  
\definecolor{dmagenta}{rgb}{.4,.1,.5}       
\definecolor{dblue}{rgb}{.0,.0,.5}     
\definecolor{mblue}{rgb}{.0,.0,.8}     
\definecolor{ddblue}{rgb}{.0,.0,.4}            
\definecolor{dred}{rgb}{.6,.0,.0}   
\definecolor{dgreen}{rgb}{.0,.5,.0}  
\definecolor{Eeom}{rgb}{.0,.0,.5}

\newcommand{\frK}{\mathfrak{K}}
\newcommand{\srZ}{\mathscr{Z}}

\usepackage[normalem]{ulem}
\newcommand{\commentout}[1]{}

\newtheorem{lemma}{Lemma}[section]
\newtheorem{theorem}{Theorem}[section]
\newtheorem{proposition}{Proposition}[section]
\newtheorem{corollary}{Corollary}[section]

\theoremstyle{definition}
\newtheorem{definition}{Definition}[section]

\theoremstyle{remark}
\newtheorem{remark}{Remark}[section]
\numberwithin{equation}{section}
\newcommand{\eps}{\epsilon}
\hypersetup{
  colorlinks=true,
  citecolor=dblue,
  linkcolor=dblue,
  frenchlinks=false,
  pdfborder={0 0 0},
  naturalnames=false, 
  hypertexnames=false,
  breaklinks}
\usepackage[capitalize,nameinlink]{cleveref}

\crefname{section}{Section}{Sections}
\crefname{subsection}{Section}{Sections}
\crefname{condition}{Condition}{Conditions}
\crefname{hypothesis}{Hypothesis}{Conditions}
\crefname{assumption}{Assumption}{Assumptions} 
\crefname{lemma}{Lemma}{Lemmas} 
  
\Crefname{figure}{Figure}{Figures}

\crefformat{equation}{\textup{#2(#1)#3}}
\crefrangeformat{equation}{\textup{#3(#1)#4--#5(#2)#6}}
\crefmultiformat{equation}{\textup{#2(#1)#3}}{ and \textup{#2(#1)#3}}
{, \textup{#2(#1)#3}}{, and \textup{#2(#1)#3}}
\crefrangemultiformat{equation}{\textup{#3(#1)#4--#5(#2)#6}}%
{ and \textup{#3(#1)#4--#5(#2)#6}}{, \textup{#3(#1)#4--#5(#2)#6}}%
{, and \textup{#3(#1)#4--#5(#2)#6}}

\Crefformat{equation}{#2Equation~\textup{(#1)}#3}
\Crefrangeformat{equation}{Equations~\textup{#3(#1)#4--#5(#2)#6}}
\Crefmultiformat{equation}{Equations~\textup{#2(#1)#3}}{ and \textup{#2(#1)#3}}
{, \textup{#2(#1)#3}}{, and \textup{#2(#1)#3}}
\Crefrangemultiformat{equation}{Equations~\textup{#3(#1)#4--#5(#2)#6}}%
{ and \textup{#3(#1)#4--#5(#2)#6}}{, \textup{#3(#1)#4--#5(#2)#6}}%
{, and \textup{#3(#1)#4--#5(#2)#6}}

\crefdefaultlabelformat{#2\textup{#1}#3}
\newcommand{\upu}{\Upupsilon}
\newcommand{\fre}{\mathfrak{e}}
\newcommand{\rQ}{\ring Q}

\newcommand{\cE}{{\mathcal{E}}}  

\newcommand{\cG}{{\mathcal{G}}}  
\newcommand{\Lg}{\mathcal{L}}    

\newcommand{\cP}{{\mathcal{P}}}  
\newcommand{\cX}{{\mathcal{X}}}  

\newcommand{\veps}{\varepsilon}

\newcommand{\frC}{\mathfrak{C}}
\newcommand{\frD}{\mathfrak{D}}
\newcommand{\frL}{\mathfrak{L}}
\newcommand{\diff}{\mathfrak{d}}

\newcommand{\beql}[1]{\begin{equation}\label{#1}}

\newcommand{\beq}{\begin{displaymath}}
\newcommand{\eeqno}{\end{displaymath}}
\newcommand{\eeq}{\end{equation}}

\newcommand{\E}{\mathbb{E}}
\newcommand{\PP}{\mathbb{P}}

\newcommand{\RR}{\mathds{R}}
\newcommand{\NN}{\mathds{N}}
\newcommand{\ZZ}{\mathds{Z}}

\newcommand{\Rd}{\mathds{R}^{d}}

\newcommand{\Act}{{\mathbb{U}}}
\newcommand{\Uadm}{\mathfrak{U}}
\newcommand{\Wadm}{\mathfrak{W}}
\newcommand{\Usm}{\mathfrak{U}_{\mathrm{SM}}}
\newcommand{\Wsm}{\mathfrak{W}_{\mathrm{SM}}}
\newcommand{\Ussm}{\mathfrak{U}_{\mathrm{SSM}}}

\newcommand{\Ind}{\mathds{1}}   
\newcommand{\Cc}{\mathcal{C}}   

\newcommand{\transp}{^{\mathsf{T}}}

\DeclareMathOperator*{\diag}{diag}

\newcommand{\grad}{\nabla}

\newcommand{\calP}{\mathcal{P}}

\newcommand{\cC}{\mathcal{C}}
\newcommand{\cV}{\mathcal{V}}

\newcommand{\cF}{\mathcal{F}}

\newcommand{\cZ}{\mathcal{Z}}
\newcommand{\calY}{\mathcal{Y}}

\makeatletter
\DeclareRobustCommand\widecheck[1]{{\mathpalette\@widecheck{#1}}}
\def\@widecheck#1#2{%
    \setbox\z@\hbox{\m@th$#1#2$}%
    \setbox\tw@\hbox{\m@th$#1%
       \widehat{%
          \vrule\@width\z@\@height\ht\z@
          \vrule\@height\z@\@width\wd\z@}$}%
    \dp\tw@-\ht\z@
    \@tempdima\ht\z@ \advance\@tempdima2\ht\tw@ \divide\@tempdima\thr@@
    \setbox\tw@\hbox{%
       \raise\@tempdima\hbox{\scalebox{1}[-1]{\lower\@tempdima\box
\tw@}}}%
    {\ooalign{\box\tw@ \cr \box\z@}}}
\makeatother

\usepackage{accents}
\newlength{\dhatheight}

\newcommand{\cA}{\mathcal{A}}
\newcommand{\bU}{\mathbb{U}}

\let\oldtocsection=\tocsection
\let\oldtocsubsection=\tocsubsection
\let\oldtocsubsubsection=\tocsubsubsection
\renewcommand{\tocsection}[2]{\hspace{0em}\oldtocsection{#1}{#2}}
\renewcommand{\tocsubsection}[2]{\hspace{1em}\oldtocsubsection{#1}{#2}}
\renewcommand{\tocsubsubsection}[2]{\hspace{2em}\oldtocsubsubsection{#1}{#2}}

\newcommand{\ttl}{\Large Ergodic Risk Sensitive Control of Markovian Multiclass    \\[5pt]    Many-Server Queues with Abandonment}

\newcommand{\ttls}{Ergodic Risk Sensitive Control of  Multiclass $M/M/n+M$ Queues}

\begin{document}

\title[\ttls]{\ttl}

\author{Sumith Reddy Anugu}
\author{Guodong Pang}
\address{$^\ddag$Department of Computational Applied Mathematics and Operations Research,
George R. Brown School of Engineering,
Rice University,
Houston, TX 77005}
\email{anugu.reddy, gdpang@rice.edu}

\begin{abstract} 

We study the optimal scheduling problem for a Markovian multiclass queueing network with abandonment in the Halfin--Whitt regime, under the long run average (ergodic) risk sensitive cost criterion.  
The objective is to prove asymptotic optimality for the optimal control arising from the corresponding
ergodic risk sensitive control  (ERSC) problem for the limiting diffusion. In particular, we show that the optimal ERSC value 
associated with the diffusion-scaled queueing process converges to that of the limiting diffusion in the asymptotic regime. 
The challenge that ERSC poses is that one cannot express the ERSC cost as an expectation over the mean empirical measure associated with the queueing process, unlike in the usual case of  a long run average (ergodic) cost. 
We develop a novel approach by exploiting the variational representations of the limiting diffusion and the Poisson-driven queueing dynamics, which both involve certain auxiliary controls. 
The ERSC costs for both the diffusion-scaled queueing process and the limiting diffusion can be represented as the integrals of an extended running cost over a mean empirical measure associated with the corresponding extended processes using these auxiliary controls. 
For the lower bound proof, we exploit the connections of the ERSC problem for the limiting diffusion with a two-person zero-sum stochastic differential game. We also make use of the mean empirical measures associated with the extended limiting diffusion and diffusion-scaled processes with the auxiliary controls. 
One major technical challenge in both lower and upper bound proofs, is to establish the tightness of the aforementioned mean empirical measures for the extended processes.
 We identify nearly optimal controls appropriately in both cases so that the  existing ergodicity properties of the limiting diffusion and diffusion-scaled queueing processes can be used. 
\end{abstract}


\date{\today} 

\maketitle

\allowdisplaybreaks

\section{Introduction} \label{sec-intro}

We study the dynamic scheduling problem for a Markovian multiclass queueing network with abandonment (in particular, the `V' network) in the Halfin--Whitt regime, where the objective is to minimize a long-run average (ergodic) risk sensitive cost.  
Specifically, if $ \hat Q^n$  is the diffusion-scaled queue-length process for different classes of customers (with scaling parameter $n$), the ergodic risk sensitive control (ERSC) problem is to minimize the cost 
\begin{align*}
\limsup_{T\to\infty}\frac{1}{T}\log \E\Big[e^{\int_0^T \widetilde r (\hat Q^n_t)dt}\Big]
\end{align*}
for a running cost function $\widetilde r$.  The optimization is over all work-conserving scheduling policies, allowing preemption, that is, the allocation of service capacity to different classes at each time (no server will be idling if there is a job waiting in queue). 
 In the Halfin--Whitt regime, when the arrival rates, service rates and abandonment rates (all class-dependent) and the number of servers are scaled properly (with growing number of servers and arrival rates), the queueing dynamics can be approximated by a limiting diffusion (see  \cite{AMR04}). 
 The goal of this paper is to show the asymptotic optimality of the solution to the ERSC problem for the limiting diffusion, in particular, the optimal value for the ERSC problem for diffusion-scaled processes converges to that of the ERSC problem for the limiting diffusion.

The optimal scheduling problem for the `V' network model has been studied under the infinite-horizon discounted cost in \cite{AMR04} and ergodic cost in \cite{ABP15}. These cost criteria do not take into account risk sensitivity in the decision making, in particular, the exponential cost criterion concerns all the moments   of the cumulative cost process (which is the integral of the running cost), and particularly stresses the penalty when the queue lengths are large.  As far as we know, our work is the first to study infinite-horizon ergodic risk sensitive control problems for stochastic networks. Due to the ergodic and risk sensitive nature of our problem, the techniques used in this paper differ drastically from those in \cite{AMR04, ABP15}.

ERSC problems for Markov processes have been extensively studied, see, e.g.,  recent surveys \cite{biswas2022survey,bauerle2023markov}. 
The works on diffusions in  \cite{FM95,runolfsson90,runolfsson94,biswas2009small,biswas-eigenvalue,biswas2010risk,ari2018strict} under various conditions on the stability of controlled diffusions and the running cost function are most relevant. 
 In~\cite{AB18}, the ERSC problem for diffusions is studied under the near-monotone cost condition in addition to assuming that dynamics is recurrent, without any condition on the stability (like positive recurrence or exponential ergodicity) of the dynamics. 
 In \cite{ari2018strict}, it is further studied under certain stronger versions of exponential ergodicity and certain growth conditions on the running cost function. 
 For our limiting diffusion,  uniform exponential ergodicity  under any stationary Markov control  is proved in \cite{AHP21} (see Proposition~\ref{prop-lyap-with-abandonment}) and the running cost function is chosen to satisfy appropriate growth condition.  Therefore, by \cite[Theorem 4.1]{ari2018strict}, we obtain the existence and characterization for the optimal ERSC problem of our limiting diffusion (see Theorem~\ref{thm-diffusion}). 
 
The techniques we use to prove asymptotic optimality differ from those in~\cite{ABP15,AP2019}, where the classical ergodic control (CEC) problems for the `V' network and multiclass multi-pool network were studied, respectively.  
It is worth stressing that this difference is a necessity, rather than simply a technicality. 
 Under positive recurrence of the diffusion-scaled queueing process, it is clear that the CEC cost can be written as an integral of running cost function with respect to the mean empirical (occupation) measure of queueing process and the scheduling control policy
 (say $\mu_{n}$ with $n$ denoting the scaling parameter) corresponding to the diffusion-scaled queueing process. Therefore, proving that $\{\mu_{n}\}_{n\in \NN}$ is tight, can ensure that along a subsequence, the integral of the running cost function with respect to $\mu_{n}$ converges to the integral of the running cost function with respect to $\mu^*$ (with $\mu^*$ being the ergodic occupation measure associated with the limiting diffusion). 
In contrast, even under the positive recurrence (or even exponential ergodicity) of the diffusion-scaled queueing  process, the ERSC cost cannot be expressed as the integral with respect to the mean empirical measure $\mu_n$. Therefore, simply proving tightness of $\{\mu_{n}\}_{n\in \NN}$ and using the weak convergence of $\{\mu_{n}\}_{n\in \NN}$ (along a subsequence  to a limit) do not ensure the convergence of the sequence of ERSC cost for diffusion-scaled process to that of the limiting diffusion. 

To overcome this difficulty, we make extensive use of certain variational representations (see Theorems~\ref{thm-var-rep-BM-gen} and~\ref{thm-var-rep-poisson-gen}) of exponential functionals of Brownian motion and Poisson process. We refer interested readers to \cite{budhiraja2019analysis} which presents many relevant techniques and results using these representations.  Roughly, these representations take the following form in the  one-dimensional case: Fix $T>0$ and, let  $W$ and $N$ be one-dimensional standard Brownian motion and   Poisson process with rate $\lambda$, respectively. Then, for appropriate non-negative functions $G_1$ and $ G_2$,
\begin{align}\label{var-rep-W-intro}
	\frac{1}{T}\log \E[e^{TG_1( W)}]= \sup_{w} E\bigg[ G_1\bigg( W_{\cdot}+\int_0^{\cdot}w_sds\bigg) - \frac{1}{2T}\int_0^T\|w_t\|^2dt\bigg]
		\end{align}
		\begin{align}
		\label{var-rep-N-intro}
\frac{1}{T}\log \E\Big[ e^{T G_2(N)}\Big]= \sup_{\phi} \E\bigg[ G_2( N^\phi_{\cdot})-\frac{\lambda}{T}\int_0^T \varkappa (\phi_t)dt\bigg],
	\end{align}
	where the suprema in the above displays are over appropriate classes of functions (see Theorems~\ref{thm-var-rep-BM-gen} and~\ref{thm-var-rep-poisson-gen} for their precise definitions). Also, $N^\phi$ is an inhomogeneous Poisson process with the instantaneous rate given by $\phi$ and $\varkappa (r)= r\ln r -r+1$.

We first provide variational representations for the Brownian-driven limiting diffusion $X$ and the Poisson-driven diffusion-scaled queueing processes $\hat Q^n$. This is done by considering $T^{-1}\int_0^T \widetilde r\big((e\cdot X_t) U_t\big)dt$ and $T^{-1}\int_0^T \widetilde r\big(\hat Q^n_t\big)dt$ as functionals of the underlying Brownian motion $W$ (which is now multi-dimensional) and $N$ (which is now a vector of independent Poisson processes), respectively. Then, using multi-dimensional versions of~\eqref{var-rep-W-intro} and~\eqref{var-rep-N-intro}, we  show that  the ERSC cost associated with the limiting diffusion and the diffusion scaled queueing process as the maximization problem (over certain auxiliary controls which are, respectively,  multi-dimensional analogs of $w$ and $\phi$ above; see Corollaries~\ref{cor-var-rep-risk-B} and~\ref{cor-var-rep-unbounded-poisson}). However, these maximization problems are in terms of an extended running cost which is the difference of the original running cost and an appropriate relative entropy term, and in terms of processes (referred to as the ``extended processes"; see equations \eqref{X-control} and~\eqref{eq-prelimit-controlled}) that have the same evolution equations as $X$ and $\hat Q^n$, except that these equations are driven by $W_{\cdot}+\int_0^\cdot w_sds$ and $ N^\phi$, respectively, instead of just $W$ and $N$.

The advantage is that these representations are linear in the original running cost, and more importantly, the ERSC cost functions (associated with both the limiting diffusion, for a given admissible control and the diffusion scaled queueing process, for a given scheduling control policy) can then be regarded as  CEC problems for respective extended processes under their respective auxiliary controls.  As a consequence, one can rewrite the ERSC cost functions as
 the integrals of an extended running cost over the mean empirical 
 measures associated with the corresponding extended processes using these auxiliary controls. Therefore, we can treat these ERSC problems as the corresponding CEC problems for the extended processes with the auxiliary controls by exploiting some of the existing techniques and results in the CEC theory (together with additional techniques).

We note that  variational formulations have been used for risk sensitive control problems in the literature.  
 In the case of a finite-horizon risk sensitive control problem, such a variational formulation was derived using the theory of large deviations in \cite{whittle91}. One of the first works studying the variational formulation of ERSC problem for diffusions is in \cite{FM95,runolfsson90,runolfsson94}, which was then followed by \cite{biswas2010risk,AB20,ABBK20,ari2018strict}. In \cite{runolfsson94}, the variational formulation was derived under restricted conditions of Markov controls that are continuous in their arguments 
 and  inf-compactness of the running cost function, whereas, in \cite{runolfsson90} the case of a linear-quadratic control  problem is studied.  
However, all of these are in the context of diffusions under certain restrictive conditions, and not amenable to prove asymptotic optimality for our model. 

We now explain the methodology of the proof of asymptotic optimality.  
One critical component in both lower and upper bound proofs concerns the ergodicity properties of the extended processes with the auxiliary controls arising from the variational representations. 
This cannot be deduced in a straightforward manner from the existing ergodicity properties of the limiting diffusion and the diffusion-scaled queueing processes. In the proof of the lower bound, 
we choose a nearly optimal scheduling policy and a careful choice of auxiliary control for the diffusion-scaled processes, under which the positive recurrence property of the extended diffusion-scaled process can be established. In the proof of the upper bound, we choose a nearly optimal control for the limiting diffusion control problem, and then construct a sequence of scheduling policies together with the associated nearly optimal auxiliary controls. Then in order to prove tightness of the mean empirical measures associated with the extended diffusion-scaled processes, we prove tightness of a carefully chosen function of these nearly optimal controls. 

For the proof of the lower bound, we make use of the connection of ERSC problems for diffusions with stochastic differential games. Such connections are recently studied in~\cite{ABBK20} (under near-monotonicity for running cost functions) and in~\cite{AB20} (under blanket/uniform stability for controlled diffusions). For earlier results, we refer to the references in~\cite{ABBK20,AB20}.  
 For our purpose, we first show that the optimal value of ERSC problem for diffusions is equal to the value (where supremum and infimum operations are interchangeable) of a TP-ZS SDG, where the optimizing criterion is a long time average of a certain extended running cost.  (See Theorem~\ref{thm-sup-inf}. The closest work in terms of similarity of conditions on the dynamics and running cost is~\cite[Theorem 2.13] {AB20}, which is however given in terms of a Collatz--Wielandt formula.) 
Moreover, we show that there are compactly supported Markov strategies for the maximizing player that are nearly optimal. This uses a variant of spatial truncation technique (which was originally first introduced in~\cite[Section 4.1]{ABP15} in the context of  CEC problems).  
In this technique, we show that the aforementioned TP-ZS SDG is a limit of a family of TP-ZS SDGs where the allowed maximizing strategies are compactly supported. Using these results and choosing appropriately nearly optimal controls, on large enough compact sets, we make the aforementioned careful choice of auxiliary controls. This also helps us to prove the required positive recurrence of extended processes from the positive recurrence of the original process. This is because their respective infinitesimal generators coincide outside a large compact set. 
From this positive recurrence and the aforementioned interchangeability of supremum and infimum operations, we bound  the optimal value of ERSC problem for the limiting diffusion  from above by  the integral of the aforementioned running cost  with respect to  any ergodic occupation measure corresponding to the nearly optimal compactly supported Markov strategies for the maximizing player. 
This leads to the lower bound for asymptotic optimality.

For the proof of the upper bound, as a consequence of the variational representation for the diffusion-scaled queueing process, the ERSC cost can be written in terms of the  integral of a certain extended running cost function over the mean empirical (occupation) measure 
of the extended process with auxiliary controls. 
This extended running cost function is a difference of the original running cost function and an extra term that plays the role of a certain Radon-Nikodym derivative. 
Recall that the upper bound proof for the CEC problem of the same multiclass queueing model also employs such a strategy while the mean empirical measure only concerns the diffusion-scaled controlled queueing process itself (see Section 5.2 in \cite{ABP15}), for which tightness of these mean empirical measures as a result of uniform stability of the diffusion-scaled queueing process plays a crucial role. 
For the ERSC problem, despite the advantage of the representation using the mean empirical (occupation) measure mentioned above, we face the additional challenge to establish its tightness property, in particular, concerning the auxiliary controls. One difficulty comes from the fact that the nearly optimal auxiliary controls take values in a non-compact space and are parametrized by both a finite time and the scaling parameter $n$. 
We introduce a suitable topology in Section~\ref{sec-top} that is appropriately weak to establish compactness of the set of auxiliary controls (see Lemma~\ref{lem-comp},~\ref{lem-compact} and Corollary~\ref{cor-control-compact}). To be more elaborate, under this topology, the extra term in the extended cost behaves like an inf-compact function. In particular, bounding this term implies compactness under this topology. 
Using a Lyapunov function, we show in Lemma~\ref{lem-tightness-empirical} that the family of the extended processes  associated with any compact set of auxiliary controls (under this topology) is stable uniformly in $n$.   
Moreover, to prove the upper bound, we use appropriate truncations on both terms in the extended cost function in order to invoke compactness arguments and then take corresponding limits with these truncations in the appropriate topologies. Because of the exponential functional in the ERSC objective, the techniques used in taking the truncation to the limit resemble closely the techniques in the theory of large deviations \cite{dembo2009large} (see Lemmas~\ref{lem-tail-est} and~\ref{lem-fin}, and Corollary~\ref{cor-trunc-limit-L}).    In contrast to this, the CEC problem does not require such a truncation of running cost function as all the relevant controls are usually compact space valued.  
These results give us the desired tightness of the mean empirical measures and connections between the ERSC objective function and that associated with the variational representation, and therefore complete the proof of the upper bound.

Prior to this work, 
 risk sensitive control problems of stochastic networks have been studied to a limited extent 
 where they are related to deterministic differential games (in the conventional heavy traffic regime) under two kinds of scalings:
 
\noindent {\it Large deviation scaling:} Under this scaling, Atar et. al in \cite{atar2003escape}  have studied the risk sensitive escape time criterion for a multiclass multiserver Markovian queueing network;  in  \cite{AGS2013},  Atar et al. have studied the finite-horizon risk sensitive control problem for a Markovian multiclass parallel server model
and in \cite{atar2013risksensitive},  Atar et al. have shown that a certain priority policy  is asymptotically optimal for a risk sensitive control problem on the finite horizon, for a multiclass M/M/1 queueing model where the running cost is a weighted total queue length.  A similar problem is investigated by Atar and Mendelson \cite{atar2016nonmarkovian} in the case of a non-Markovian multiclass single server  queue with renewal arrivals and i.i.d service times (with general distribution). 
In these works, it is shown that the finite-horizon  risk sensitive control problem (the problem of risk-sensitive escape criterion in \cite{atar2003escape}, respectively)  can be represented (asymptotically) as a two-person zero-sum  deterministic differential game (TP-ZS DDG). 
The pay-off criterion in the game is expressed in terms of the  finite time integral of the original running cost  and the underlying large deviation rate functions associated with class-dependent  inter-arrival and  service times.

\noindent{\it  Moderate deviation scaling:}   Under this scaling, Atar and Saha \cite{AS2017} have studied the optimality of the generalized $c\mu$ rule for a non-Markovian multiclass single-server  network  under the finite-horizon risk-sensitive cost criterion.  In \cite{AC2016}, Atar and Cohen have studied a differential game arising from a finite-horizon risk sensitive control problem for a non-Markovian  multiclass single-server queueing  model. Atar and Cohen \cite{atar2017asymptotically}, Atar and Biswas \cite{atar2014control}, and Biswas \cite{biswas2014risk}  have studied a problem similar to that in \cite{AC2016} for the non-Markovian network model with multiple servers,  where the  optimal control problem is  also related to a TP-ZS DDG. 
The finite-horizon  risk sensitive control problem of various stochastic networks under this scaling  is shown to be asymptotically represented by a TP-ZS DDG. However, the main difference is that the associated pay-off criterion is the difference of the  finite-time integral of the original cost function and the moderate deviation rate function for the underlying class-dependent inter-arrival and service times. Since the moderate deviation rate function depends only on the variances of the underlying  general inter-arrival and service times, 
this is in contrast to the large deviation scaling, where the rate function depends on their entire distributions
 in the non-Markovian setting. 

In comparison to the works mentioned above, the results in this paper differ drastically. In particular,  we show that the ERSC problem for the  multiclass `V' network model approaches a TP-ZS SDG  instead of a TP-ZS DDG.  
 We also mention that a discounted cost version of the aforementioned TP-ZS SDG (arising from  minimizing the discounted  running cost under model uncertainty) is investigated in the context of multiclass M/M/1 queues in  \cite{cohen2019brownian} for the Brownian control problem (a similar study in the context of single-stage queue can be found in \cite{jain2010optimality}), whereas, the problems of asymptotic optimality and asymptotically optimal controls are investigated in \cite{cohen2021asymptotic,cohen2019asymptotic}. The model uncertainty in these works is accounted for \emph{via.} an addition of a relative entropy term to the cost function that penalizes the discrepancy in the model specification.

\subsection{Organization of the paper} In the rest of this section, we introduce the notation used in the paper. Section~\ref{sec-main} presents the model description, ERSC problem formulation, and the main result. Section \ref{sec-queue} describes the network model in detail, and presents the ERSC formulation for the diffusion-scaled processes. Section \ref{sec-ERSC-diffusion} presents the ERSC problem for the limiting diffusion, and the characterization of the solution to the ERSC problem. 
Section \ref{sec-AO} gives the main result on asymptotic optimality, and provides an overview of the main ideas for its proof.  Sections \ref{sec-sub-var-BM} and \ref{sec-sub-var-PP} present the variational formulations for the limiting diffusion and the Poisson-driven diffusion-scaled queueing dynamics, respectively. Some important properties associated with these variational formulations are proved in these sections and in Sections \ref{sec-top} and \ref{sec-aux-diffusion} in order to be used in the proofs for asymptotic optimality. 
Sections~\ref{sec-low-bound} and~\ref{sec-upp-bound}  prove the lower and upper bounds, respectively, for asymptotic optimality. Finally, we collect some auxiliary results and their proofs in the Appendix.

\subsection{Notation} \label{sec-notation}We use $(\Omega, \cF, \PP)$ to denote the underlying abstract probability space with $\E$ as the associated expectation. $\E_x$ denotes the expectation when the underlying process starts at $x$. The standard Euclidean norm in $\RR^d$ is denoted by $\| \cdot \|$, $ x\cdot y $  denotes the inner
product of $x,y\in \RR^d$, and $x\transp$ denotes the transpose of $x \in \RR^d$. The set of nonnegative real numbers (integers) is denoted by $\RR_+$ ($\ZZ_+$), $\NN$ stands for the set of natural numbers, and $\Ind_{\{\cdot\in A\}}$ denotes the indicator function corresponding to set $A$. The minimum (maximum) of two real numbers $a$ and $b$ is denoted by $a \wedge b$ ($a \vee b$), respectively, and
$a^{\pm} \doteq  (\pm a) \vee 0$. The closure, boundary, and complement of a set $A \subset \RR^d$ are denoted by $\bar  A$, $\partial A$ and $A^c$, respectively.

The term domain in $\RR^d$ refers to a nonempty, connected open subset of $\RR^d$. For a domain $D \subset \RR^d$, the space $\cC^k (D)$ ($\cC^\infty (D)$, respectively), $k \geq 0$, refers to the class of all real-valued
functions on $D$ whose partial derivatives up to order $k$ (any order, respectively) exist and are continuous. 
The space $L^p (D)$, $p \in [1, \infty)$, stands for
the Banach space of (equivalence classes of) measurable functions $f$ satisfying $\int_D |f (x)|^p dx < \infty$,
and $L^\infty (D)$ is the Banach space of functions that are essentially bounded in $D$. The standard
Sobolev space of functions on $D$ whose generalized derivatives up to order $k$ are in $L^p (D)$, equipped
with its natural norm, is denoted by $W^{k,p} (D)$, $k \geq 0, p \geq 1$. In general, if $\cX$ is a space of real-valued
functions on a set $Q$, $\cX_\text{loc}$ consists of all functions $f$ such that $f \phi \in \cX$  for every $\phi $ that is compactly supported smooth function on $Q$.  Here, $f\phi$ is simply the scalar multiplication of the functions $f$ and $\phi$.

For $k \in  \NN$, we let $\frD^k \doteq  D(\RR_+ , \RR^k )$ ($\frC^k\doteq C(\RR_+,\RR^k)$, respectively) denote the space of $\RR^k$--valued c\`adl\`ag functions equipped with Skorohod topology (continuous functions equipped with locally uniform topology, respectively) on $\RR_+$ ($\frD^1$ and $\frC^1$ are simply written as $\frD$ and $\frC$, respectively).  Whenever the domain is $[0,T]$, we write $\frD^k_T$ or $\frC^k_T$ ($\frD_T$ or $\frC_T$, for $k=1$). For a Polish space $\cX$, $\cP(\cX)$ is the set of Borel probability measures on $\cX$ equipped with the topology of weak convergence.  We say that a function $V:\cX\rightarrow \RR^+$ is inf-compact, if the set $\{x\in \cX: V(x)\leq l\}$ is either compact in $\cX$ or empty, for every $l\geq 0$. For a function $f:\RR^d\rightarrow \RR$, 
\begin{equation}\label{eq-diff-f}
 \diff f(x;y)\doteq f(x+y)-f(y).
 \end{equation}
The identity function on the real line by $\fre$.

\section{Model and Results} \label{sec-main}

\subsection{ERSC for Multiclass $M/M/n+M$ queues}\label{sec-queue}
We study a multiclass Markovian queueing model with $d$ classes of jobs/customers and one pool of $n$ parallel servers. Each class has an independent Poisson arrival process of rate $\lambda^n_i>0$, $i=1,\dots, d$. 
The service times for class $i$ jobs are i.i.d. exponential with rate $\mu^n_i>0$. Jobs of each class form their own queue and are served in the first-come first-served (FCFS) discipline. Jobs can abandon while waiting in the queue, and class $i$ jobs have i.i.d. exponential patience times with rate $\gamma^n_i>0$.
 Let $r_i^n = \lambda^n_i/\mu_i^n$ be the mean offered load of class $i$, and then the traffic intensity is given by $\rho^n= n^{-1} \sum_{i=1}^d r^n_i$.  
We assume that the system is operating in the Halfin--Whitt regime, in which the parameters are assumed to satisfy the following conditions: as $n\to \infty$,
\begin{align} \label{eqn-HW-parameter}
\frac{\lambda^n_i}{n} \to \lambda_i>0, \quad \frac{\lambda^n_i - n \lambda_i}{\sqrt{n}} \to \hat\lambda_i\in \RR,  \nonumber\\
 \mu^n_i \to \mu_i>0,  \quad \sqrt{n}(\mu^n_i-\mu_i) \to \hat\mu_i, \quad \gamma^n_i \to \gamma_i > 0,  
\end{align}
and
\begin{align} \label{eqn-HT-condition}
\tilde \rho^n\doteq \sqrt{n} (1-\rho^n) \to \hat{\rho}= \sum_{i=1}^d \frac{\rho_i \hat\mu_i - \hat\lambda_i}{\mu_i} \in \RR\,, 
\end{align}
where  $ \rho_i = \frac{\lambda_i}{\mu_i}<1$ satisfies $\sum_{i=1}^d \rho_i=1$. It is clear that $n^{-1} r_i^n\to \rho_i$ for each $i$. In addition, for each $i=1,\dots,d$, as $n\to \infty$,
\[
\ell^n_i = \frac{\lambda^n_i - n \lambda_i}{\sqrt{n}} - \rho_i \sqrt{n}(\mu^n_i-\mu_i) \to \ell_i = \frac{\hat\lambda_i - \rho_i \hat\mu_i}{\mu_i}\,. 
\]
Denote $\ell = (\ell_1,\dots, \ell_d)\transp$.  

 Let $\{X^n_{i,t}\}_{t\ge 0}$ be the number of class $i$ jobs in the system, $ \{Q^n_{i,t}\}_{t\ge 0}$ be the number of class $i$ jobs in the queues and $\{Z^n_{i,t}\}_{t\ge 0}$ be the number of class $i$ jobs in service at each time.  Write  $X^n = (X^n_{1}, \dots, X^n_{d})\transp$ as the $d$-dimensional processes, and similarly, for $Q^n, Z^n$ and so on. 
The processes $\{Z^n_{i,t}\}_{t\ge 0}$ also represent the server allocation at each time, and hence they are regarded as ``scheduling control policies" (SCPs).   We will only consider work-conserving policies that are  non-anticipative (see Definition~\ref{def-adm} below) and allow preemption (that is, service of a customer can be interrupted to serve some other customer of a different class and resumed at a later time). These policies satisfy the following condition:
\begin{align} \label{eqn-WC-condition}
e\cdot Z^n_t = (e\cdot X^n_t) \wedge n, \quad t \ge 0.
\end{align}
The action set $\mathbb{A}^n(x)$ is given by
\[
\mathbb{A}^n(x) = \{z \in \ZZ^d_+: z \le x \,\, \text{and}\,\, e\cdot z = (e\cdot x) \wedge n\}
\]
and the  balance equation is given by
\[
X^n_{i,t} = Q^n_{i,t} + Z^n_{i,t}, \quad t \ge 0, \quad i=1,\dots,d\,.
\]
From hereon, we assume that $X^n_{i.0}$, $i=1,\ldots,d$, are  deterministic.

We now introduce the notion of admissibility of an SCP. Let $A^n_i, S^n_i, R^n_i$, $i=1,\dots,d$, be mutually independent standard Poisson processes. Also, let 
\begin{align*}
\mathscr{F}^n_t &\doteq \sigma \Big\{ \widetilde A^n_i(s),\widetilde S^n_i(s), \widetilde R^n_i(s): s\leq t\Big\}\vee\mathcal{N}\,,\\
\mathscr{G}^n_t &\doteq \sigma \Big\{ \delta\widetilde A^n_i(t,s),\delta \widetilde S^n_i(t,s), \widetilde R^n_i(t,s): s\geq 0\Big\}\,,
\end{align*}
where $\mathcal N$ is the collection of all $\PP$--null sets, and
\begin{align*}
\widetilde A^n_i(t)&\doteq A^n_i(\lambda^n_i t) & \delta \widetilde A^n_i(t,s)\,&\doteq \widetilde A^n_i(t+s)-\widetilde A^n_i(t)\,,\\
\widetilde S^n_i(t)&\doteq  S^n_i\big(\mu^n_i\int_0^t Z^n_i(r)dr\big)\,,& \delta\widetilde S^n_i(t,s)&\doteq S^n_i\big(\mu^n_i\int_0^t Z^n_i(r)dr+ \mu^n_i s\big)-\widetilde S^n_i(t)\,,\\
\widetilde R^n_i(t)&\doteq  R^n_i\big(\gamma^n_i\int_0^t Q^n_i(r)dr\big)\,,& \delta\widetilde R^n_i(t,s)&\doteq R^n_i\big(\gamma^n_i\int_0^t Q^n_i(r)dr+ \gamma^n_i s\big)-\widetilde R^n_i(t)\,.
\end{align*} 
\begin{definition}\label{def-adm} An SCP $Z^n$ is said to be admissible if for every $t\geq 0$,
\begin{enumerate}
\item $Z^n(t)\in \mathbb{A}^n\big(X^n(t)\big)$ a.s. ;
\item $Z^n(t)$ is $\mathscr{F}^n_t$--measurable;
\item (Non-anticipativity) $\mathscr{F}^n_t$ is independent of  $\mathscr{G}^n_t$;
\item For every $i=1,\ldots,d$, the processes $\delta \widetilde S^n_i(t,\cdot)$ and $\delta \widetilde R^n_i(t,\cdot)$ agree in law with $S^n_i(\mu^n_i\cdot)$ and $R^n_i(\gamma^n_i\cdot)$, respectively.
\end{enumerate}
\end{definition}
Let $\Act^n$ be the set of all admissible control policies $Z^n \in \mathbb{A}^n(x)$ for any given $x\in \ZZ^d_+$.  
Under a work-conserving control policy $Z^n$ satisfying \eqref{eqn-WC-condition}, 
the process $X^n$ can then be described by the following equation:
\begin{equation}\label{eq-rep-Xn}
X^n_{i,t} = X^n_{i,0} + A^n_i(\lambda^n_i t) - S^n_i\left(\mu^n_i \int_0^t Z^n_{i,s} ds\right) - R^n_i \left(\gamma^n_i \int_0^t Q^n_{i,s}ds \right)\,.
\end{equation}
Let $\hat{X}^n$, $\hat{Q}^n$ and $\hat{Z}^n $  be the diffusion-scaled processes defined by 
\[
\hat{X}^n_{i,t} = \frac{1}{\sqrt{n}} (X^n_{i,t} -n \rho_i  )\,, \quad \hat{Q}^n_{i,t} = \frac{1}{\sqrt{n}} Q^n_{i,t}\,, \quad 
\hat{Z}^n_{i,t} = \frac{1}{\sqrt{n}} (Z^n_{i,t} -n \rho_i  )\,. 
\]
We now define a new process $U^n$ as follows: for $t\geq 0$, 
\begin{align}\label{def-control-n}
U^n_t\doteq \begin{cases}\frac{\hat X^n_t-\hat Z^n_t}{(e\cdot \hat X^n_t)^+}, & \text{whenever $e\cdot  X^n_t>0$},\\
\, e_d, & \text{ otherwise.}
\end{cases}
\end{align} 
Here, $e_d=(0,\ldots, 1)\in \RR^d.$ Observe that 
 $U^n \in\Act \doteq  \{u \in\RR^d_+: e \cdot u=1\} $. Also, the choice of $e_d$ is made only for convenience and to ensure that $e\cdot U^n=1$; it can be replaced by any $v\in \RR^d_+$ that satisfies $e\cdot v=1$. With slight abuse to notation, we denote the set of all $U^n$ that correspond to admissible SCPs also by $\Act^n$.
Evidently, for the diffusion-scaled processes, in terms of  $U^n$, we have
\begin{equation}\label{eqn-QZ-hat}
\hat{Q}^n_t = (e\cdot \hat{X}^n_t)^+ U^n_t, \quad \hat{Z}^n_t = \hat{X}^n_t -  (e\cdot \hat{X}^n_t)^+ U^n_t\,. 
\end{equation}
It is clear that at any given time $t$ and for a fixed realization of $\hat X^n_t$, the values of important random variables of the $n$th system such as $\hat Q^n_t$ and $\hat Z^n_t$  will depend on the choice of $U^n_t$. 
For this reason, we interpret $U^n$ as the control process.

Given the initial state $\hat{X}^n_0$ and a work-conserving control policy $\hat Z^n \in \Act^n$ (equivalently, $U^n\in \Act^n$, which we use instead), 
the ergodic risk sensitive cost function for the diffusion-scaled state process $\hat{X}^n$ is given by 
\begin{equation} \label{eqn-RS-cost-n}
J^n(\hat{X}^n_0, U^n) \doteq  \limsup_{T\to\infty} \frac{1}{T} \log \E\left[ e^{\int_0^T   \kappa\cdot [(e\cdot \hat{X}^n_t)^+U^n_t]  dt}\right],
\end{equation}
for some constant $\kappa\in \RR^d_+$.  Here the objective function 
penalizes the queueing process in the diffusion scaling (represented by the diffusion-scaled queueing process $\hat{Q}^n$). Recall \eqref{eqn-QZ-hat}.
For notational convenience, we write $r(\hat{X}^n_t, U^n_t)=\kappa\cdot[(e\cdot \hat{X}^n_t)^+U_t]$, or $r(x,u)= \kappa\cdot[(e\cdot x)^+u]$. (See also Remark \ref{rem-runningcost} for the choice of this running cost function.) Whenever there is no confusion in the scheduling policy, we write $r(\hat{X}^n_t, U^n_t)= \widetilde r(\hat Q^n_t) \doteq \kappa\cdot \hat Q^n_t$. 

The associated cost minimization problem is
\begin{equation} \label{eq-ERSC-problem}
\hat{\Lambda}^n(\hat{X}^n_0) \doteq  \inf_{U^n\in \Act^n}  J^n(\hat{X}^n_0, U^n)\,.
\end{equation}

We refer to $\hat{\Lambda}^n(\hat{X}^n_0) $ as the diffusion-scaled ergodic risk-sensitive value given the initial state $\hat{X}^n_0$.  Before we proceed further, we show that $\hat{\Lambda}^n(\hat{X}^n_0)$ is finite. Define 
 \begin{align}\label{eq-hatx} \hat x^n(x)\doteq \Big( \frac{x_1-n\rho_1}{\sqrt{n}}, \frac{x_2-n\rho_2}{\sqrt{n}},\ldots,\frac{x_d-n\rho_d}{\sqrt{n}}\Big). \end{align}
 Whenever there is no confusion, we simply write $\hat x^n.$  
  The generator of the diffusion-scaled queueing process $\hat{X}^n$ under a constant SCP $u\in \bU^n$ is given by
\begin{equation} \label{eq-Lg-hatXn}
\Lg^{n,u}f(\hat x^n)\doteq  \sum_{i=1}^d\Big(\lambda^n_i \diff f(x;e_i) + \big(\mu^n_i z_i +\gamma^n_i q_i(x,z)\big) \diff f(x,-e_i)\Big)
 \end{equation}
with $q_i(x,z)=x_i-z_i$ and the notation $\diff f$ defined in \eqref{eq-diff-f}. The following exponential ergodicity result helps us in showing that the ERSC problem in \eqref{eq-ERSC-problem} is well-defined.  The result is proved in \cite[Theorem 3.4]{AHP21}, where it is stated with only the two terms $\hat C_0, \hat C_1\|\hat x^n\|$ on the right hand side in \eqref{eq-lyap-foster-prelimit}. A careful tracking of the $n$ dependence in the proof therein gives us the expression below with the additional term $\frac{\hat C_2}{\sqrt n}$ for each $n$. We omit the derivation of this for brevity.

\begin{proposition}\label{prop-lyap-with-abandonment-prelimit} For every $\varrho>0$,  there exist an inf-compact $\cC^2$ function  $\cV_\varrho^n:\RR^d\rightarrow \RR_+$, and  positive constants $\hat C_i$, for $i=1,2,3$ such that 
	\begin{align}\label{eq-lyap-foster-prelimit}
	\Lg^{n,u} \cV_\varrho^n(\hat x^n)\leq \varrho\Big(\hat C_0- \hat C_1\|\hat x^n\| + \frac{\hat C_2}{\sqrt n}\|\hat x^n\|\Big)\cV^n_\varrho(\hat x^n), \text{ for every $n$}\,.
	\end{align}
\end{proposition}

An immediate consequence of the above proposition is that $\hat \Lambda^n (\hat X^n_0)<\infty$. In other words, there is at least one $U^n\in \Act^n$ such that $J^n(\hat X^n_0,U^n)<\infty.$ Moreover, we have the following result.

\begin{lemma}\label{lem-markov-finite-cost}  For $n\geq \big(\frac{2\hat C_2}{\hat C_1}\big)^2$ and  for every $U^n\in \Act^n$, $
	J^n(\hat X^n_0,U^n)<\infty$.
\end{lemma}
\begin{proof}
	Recall that $r(x,u)= \kappa \cdot [(e\cdot x)^+u]$. 
	Using Proposition~\ref{prop-lyap-with-abandonment-prelimit} with $\varrho=\frac{2\kappa}{\hat C_1}$, $\Lg^{n,u} \cV_\varrho(\hat x^n)\leq \varrho (\hat C_0 -l(\hat x^n))\cV_\varrho^n(\hat x^n)$ and applying It{\^o}'s formula to $e^{\int_0^t l(\hat X^n_s)ds}\cV_\varrho^n(\hat X^n_t)$ and using the fact that  $\inf_{x\in \RR^d}\cV_\varrho^n(x)> 0$,  we have $$ J^n_l\doteq \limsup_{T\to\infty}\frac{1}{T}\log \E\Big[e^{\int_0^T l(\hat X^n_t)dt}\Big]\leq \hat C_0.$$ 
	
	For any $U^n\in \Act^n$, using the inf-compactness of $l_r$, we have
	\begin{align*}
	e^{\int_0^t r(\hat X^n_t,U^n_t)dt}\leq e^{\int_0^t l(\hat X^n_t)dt + \hat C_3t},
	\end{align*}
	for some large enough $\hat C_3>0$.
 	 This implies that $\hat \Lambda^n(\hat X^n_0)\leq J^n(\hat X^n_0,U^n) \leq J^n_l+\hat C_3\leq \hat C_0+\hat C_3$. This proves the result.
\end{proof}
The following corollary of Proposition~\ref{prop-lyap-with-abandonment-prelimit} is required in the proof of Lemma~\ref{lem-AO-upper} (in particular, in the proof of Lemma~\ref{lem-tail-est}).
\begin{corollary}\label{cor-finite-cost-prelimit}
	For any $\hat C_3>0$ and under any admissible control $U^n\in \Act^n$,
	 $$ \limsup_{n\to\infty}\limsup_{T\to \infty}\frac{1}{T}\log\E\Big[e^{\hat C_3\int_0^T\|\hat X^n_t\|dt }\Big]<\infty.$$
\end{corollary}
\begin{proof}
	 From the inf-compactness of  $\hat C_3\|x\|$ and  following the arguments in the proof of Lemma~\ref{lem-markov-finite-cost}, we get the result. 
\end{proof}

\subsection{ERSC for the limiting diffusion}  \label{sec-ERSC-diffusion}
It is shown in \cite{AMR04} that under the conditions in \eqref{eqn-HW-parameter}--\eqref{eqn-HT-condition} and under the work-conserving control policies, if
 \begin{align}\label{eq-init-conv}n^{-\frac{1}{2}}\ZZ^d\ni\hat{X}^n_0\to X_0\in \RR^d\end{align}as $n\to \infty$,  $\hat{X}^n\Rightarrow{X}$ in $\frD^d$ as $n\to\infty$, where ${X}$ is a $d$-dimensional controlled diffusion given as a solution to 
\begin{equation} \label{eqn-hatX}
{X}_t = {X}_0 + \int_0^t b({X}_s, U_s)ds + \Sigma W_t,
\end{equation}
where 
\begin{equation} \label{eqn-b}
b(x,u) = \ell - R(x- (e\cdot x)^+u) - (e\cdot x)^+ \Gamma u
\end{equation}
with 
\[
R=\diag(\mu_1, \dots, \mu_d), \quad \Gamma =\diag(\gamma_1,\dots,\gamma_d), \quad \Sigma\Sigma\transp=\diag(2\lambda_1,\dots,2\lambda_d). 
\]
In the rest of the paper, we assume that~\eqref{eq-init-conv} holds with $X_0=x$.
The process $U$ (referred to as control) is assumed to take values in $\Act =  \{u \in\RR^d_+: e \cdot u=1\}$. 
\begin{definition} \label{def-admissible} 
A $\bU$--valued process $U$ is called admissible if it satisfies the following: if $U_t=U_t(\omega)$ is jointly measurable in $(t,\omega)\in \RR^+\times \Omega$ and  for every $0\leq s< t$, $W_t-W_s$ is independent of the completed filtration (with respect to $(\cF,\PP)$) generated by $\{X_0,U_r, W_r: r\leq s\} $. The set of all such controls is denoted by $\Uadm$.
\end{definition}

Let $\Usm$ denote the set of stationary Markov controls. In order to study the convergence of stationary Markov controls or existence of optimal stationary Markov controls, it is useful to consider a weaker notion of a stationary Markov control, \emph{viz.,} relaxed control - the control is defined in the sense of distribution. To be more precise, a control $v$ is said to be a relaxed Markov control if $v=v(\cdot)$ is a Borel measurable map from $\RR^d$ to $\cP(\bU)$. In this case, we write $v(du|x)$ to distinguish the relaxed Markov control $v$ from other Markov controls which are referred to as precise Markov controls. Clearly, the set of relaxed Markov controls contains $\Usm$. But with slight abuse of notation, we represent the set of relaxed Markov controls also by $\Usm.$  
Under $v\in\Usm$, the controlled diffusion ${X}$ in \eqref{eqn-hatX} has a unique solution \cite[Theorem 2.2.4]{ari2011}.

We denote the generator $\Lg^{u}:\Cc^{2}(\RR^{d})\mapsto\Cc(\RR^{d})$ of the controlled diffusion ${X}$ as 
\begin{align*}
\Lg^{u} f(x) & =  \sum_{i=1}^d b_i(x,u) \frac{\partial}{\partial x_i} f(x) + \sum_{i=1}^d  \lambda_i  \frac{\partial^2}{\partial x_i^2} f(x) \,. 
\end{align*}
We write the generator $\Lg^{u}$ as $\Lg^{v} $ under $v \in\Usm$.
$\Lg^{v} $ is the generator of a strongly-continuous
semigroup on $\Cc_{b}(\RR^{d})$, which is strong Feller.
We denote by $\Ussm$ the subset of $\Usm$ that consists
of \emph{stable stationary Markov controls} (\emph{i.e.,} stationary Markov controls
under which $X$ is positive recurrent)
and by $\mu_v$ the invariant probability measure of the process
under the control $v\in\Ussm$.
For our model, $\Ussm = \Usm$ since the controlled diffusion ${X}$ is uniformly exponentially ergodic under all stationary Markov controls (see Proposition~\ref{prop-lyap-with-abandonment}).

For the limiting diffusion ${X}$ in \eqref{eqn-hatX} with $X_0=x$ and control $U\in \Uadm$, the ergodic risk sensitive cost function is given by 
\begin{equation} \label{eqn-RS-cost-hat}
J(x, U) \doteq  \limsup_{T\to\infty} \frac{1}{T} \log \E\left[ e^{\int_0^T r({X}_t, U_t )dt }\right].
\end{equation}
The associated cost minimization problem  is
\begin{equation}
{\Lambda}({x}) \doteq  \inf_{U\in \Uadm}  J({x}, U)\, \quad \text{ and } \quad \Lambda\doteq \inf_{x\in \RR^d} \Lambda(x).
\end{equation}
This is the optimal value for the ERSC problem for limiting diffusion given the initial state ${x}$. 
In addition, let 
\begin{equation}
{\Lambda}_{\text{SM}}({x}) \doteq \inf_{v\in \Usm}  J({x}, v(\cdot))\,,
\end{equation}
be the the optimal value over all stationary Markov controls given the initial state $x$. 
We will see that the optimal value is independent of the initial value. Denote
\[
{\Lambda}_{\text{SM}}  = \inf_{x\in \RR^d}{\Lambda}_{\text{SM}}(x) \quad \text{ and } \quad {\Lambda}  = \inf_{x\in \RR^d}{\Lambda}(x) \,. 
\]

For notational convenience, let 
\[
{\Lambda}_v(x) \doteq  J(x, v(\cdot)) \quad \text{for} \quad \, v\in \Usm. 
\]

We state the uniform stability result that is proved in~\cite{AHP21} as it is vital to invoke the existing results on ERSC problems for diffusions (see \cite{ari2018strict,AB18}). To that end, we have the following result \cite[Theorem 2.1]{AHP21} for the limiting diffusion. 

\begin{proposition}\label{prop-lyap-with-abandonment}
There exist a $\cC^2(\RR^d)$ inf-compact function $\cV$  with $\cV\geq 1$ and constants $C_0,C_1>0$ such that for every $u\in \bU$, 
	\begin{align}\label{eq-lyap-with-abandonment} \Lg^u \cV(x)\leq \big(C_0-C_1\|x\|^2\big)\cV(x) \quad\text{ for $x\in \RR^d$.}\end{align}
\end{proposition}
 The above proposition ensures that the ERSC cost is finite for every admissible control - in particular, $\Lambda<\infty$ which we now prove in the following corollary. 

\begin{corollary}\label{cor-well-defined-limit}For every $U\in \Uadm$, $J(x,U)<\infty$. In particular, $\Lambda<\infty$.
\end{corollary}
\begin{proof}Recall that $r(x,u)= \kappa \cdot [(e\cdot x)^+u]$. The function  $l_r(x)\doteq l(x)-\max_{u\in \mathbb{U}}r(x,u)$ is inf-compact with $l(x)\doteq C_1\|x\|^2$.
	
	Using Proposition~\ref{prop-lyap-with-abandonment}, $\Lg^{u} \cV( x)\leq (C_0 -l(x))\cV(x)$ and applying It{\^o}'s formula to $e^{\int_0^t (l(X_s)-C_0)ds}\cV(X_t)$ up to the stopping time $T\wedge \tau_R$ (where $\tau_R$ is the first exit time of the $R$-radius ball around origin)  gives us 
	\begin{align*}
	&\E\Big[e^{\int_0^{T\wedge \tau_R} (l(X_t)-C_0)dt} \cV(X_{T\wedge\tau_R})\Big]\\
	&= \cV(x)+ \E\Big[\int_0^{T\wedge \tau_R} e^{\int_0^{t} (l(X_s)-C_0)ds}\Big( \Lg^{U_t} \cV(X_t) + \big(l(X_t) -C_0\big) \cV(X_t)\Big) dt\Big]\\
		&\leq \cV(x)\,.
	\end{align*}
	To get the  inequality, we use the Lyapunov drift inequality which is~\eqref{eq-lyap-with-abandonment}. From here, using the fact that $\cV\geq 1$ and then the  Fatou's Lemma, we immediately get
	\begin{align*}
	\E\Big[e^{\int_0^{T} (l(X_t)-C_0)dt} \Big]\leq \liminf_{R\to\infty} \E\Big[e^{\int_0^{T\wedge \tau_R} (l(X_t)-C_0)dt} \Big]\leq \cV(x)\,.
	\end{align*}
	This implies 
	\begin{align*}
J_l\doteq 	\limsup_{T\to\infty}\frac{1}{T} \log \E\Big[e^{\int_0^{T} l(X_t)dt} \Big]\leq C_0\,.
	\end{align*}
	Using the  above display and the inf-compactness of $l_r$, we have
	\begin{align*}
	e^{\int_0^t r(X_t,U_t)dt}\leq e^{\int_0^t l(X_t)dt +  C_3t},
	\end{align*}
	for some large enough $C_3>0$.
 	 This implies that $ J(x,U) \leq J_l+C_3\leq C_0+C_3$. This proves the corollary.
\end{proof}

	To keep the notation concise, whenever $v\in \Usm$, we write $r(x,v(x))$ as $r^v(x)$. Whenever there is no confusion, we also write $r(x,u)$ as $r^u(x)$ for $u\in \bU.$
The following theorem gives the well-posedness and characterization of optimal stationary Markov controls, which follows directly from  \cite[Theorem 4.1]{ari2018strict}.
\begin{theorem} \label{thm-diffusion}  There exists a pair $(V, \mathsf{\Lambda})\in \cC(\RR^d)\times \RR_+ $ such that $\inf_{x\in \RR^d} V(x)>0$,  satisfying 
\begin{equation} \label{eqn-HJB}
		\min_{u \in \Act} \bigl[\Lg^{u} V(x) + r^u(x)\,V(x)\bigr] \;=\; { \mathsf{\Lambda}}\,V(x)
		\qquad\forall\,x\in\Rd\,.
		\end{equation}
Moreover, 
	\begin{enumerate}
		\item[(i)] ${ \mathsf{\Lambda}}= \Lambda_{\text{SM}}=\Lambda$ and the function $V$ is unique up to a multiplicative constant.
		\item[(ii)] Any $v\in\Usm$ that satisfies 
		\begin{equation}\label{eqn-optimality1}
		\Lg^v V(x) + r^v(x)\,V(x)\;=\;
		\min_{u\in\Act}\; \bigl[\Lg^{u} V(x) + r^u(x)\,V(x)\bigr]
		\quad \text{for a.e.\ }x\in\Rd
		\end{equation}
		is stable, and is optimal in the class $\Usm$, i.e., $\Lambda_v(y)= \mathsf{\Lambda}$
		for all $y\in\Rd$.
		\item [(iii)] Every optimal stationary Markov control satisfies~\eqref{eqn-optimality1}. 
	\end{enumerate}
\end{theorem} 
\begin{remark}  \label{rem-runningcost} 
We have chosen $r(x,u)$ to be in the form $\kappa\cdot [(e\cdot x)^+u]$.  The reason for not choosing a more general $r(x,u)$ is as follows. The main contribution of this work is proving the asymptotic optimality, where we show that the ERSC problem for the diffusion-scaled queueing process can be approximated by that of the limiting diffusion as $n\to\infty$  (in an appropriate sense). To address this, we have to show that ERSC problem for the diffusion-scaled process is well-defined for the chosen running cost function for large enough $n$. 

From Proposition~\ref{prop-lyap-with-abandonment-prelimit} and the arguments in the proof of Lemma~\ref{lem-markov-finite-cost}, it is clear that the ERSC problem for the diffusion-scaled queueing process is well-defined as long as \begin{align}\label{eq-inf-comp-cost-prelim}\big(\varrho\hat C_1-\frac{\varrho\hat C_2}{\sqrt n}\big)\|x\|-\max_{u\in \mathbb{U}^n(x)} r(x,u) \text{ is inf-compact.}\end{align}
On the other hand, to invoke \cite[Theorem 4.1]{ari2018strict} (and thereby proving that the ERSC problem for the limiting diffusion is well-defined), the conditions on $r(x,u)$ can be relaxed as long as \begin{align}\label{eq-inf-comp-cost-lim}C_1\|x\|^2-\max_{u\in \bU} r(x,u) \text{ is inf-compact.}\end{align} Here, $C_1$ is the constant from Proposition~\ref{prop-lyap-with-abandonment}.  We therefore have chosen $r(x,u)$ to be in the form $\kappa\cdot [(e\cdot x)^+u]$. 
\end{remark}

\begin{remark} \label{rem-choice-parameter}

We have focused on  the `V' network with abandonment for the following reason. 
For the ERSC problem of the limiting diffusion, the Foster-Lyapunov inequality in Proposition~\ref{prop-lyap-with-abandonment} is the crucial existing result (which can be found in \cite{AHP21}). In particular, we use the fact that for any $C_1>0$ from Proposition~\ref{prop-lyap-with-abandonment}, inf-compactness in~\eqref{eq-inf-comp-cost-lim} holds and show that the ERSC cost for any admissible control is finite (in Corollary~\ref{cor-well-defined-limit}). In addition to the finiteness of the ERSC cost, inf-compactness in~\eqref{eq-inf-comp-cost-lim} also  helps us in characterizing the  optimal stationary Markov controls (associated with the ERSC problem for the limiting diffusion) \emph{via.} the existing result from~\cite{ari2018strict} as stated in Theorem~\ref{thm-diffusion}. 
For the  `V' network  without abandonment, under the safety staffing condition requiring that $\tilde \rho$ in~\eqref{eqn-HT-condition} is strictly positive, the Foster-Lyapunov inequality was shown to take the following form:  
 there exist a $\cC^2(\RR^d)$ inf-compact function $\widetilde \cV$ such that $\widetilde \cV\geq 1$, a small $\veps>0$ and, positive constants $\widetilde C_0$ and $\widetilde C_1$ such that 
\begin{align}\label{eq-wo-abandon-limit}
\Lg^u\widetilde  \cV(x)\leq \widetilde C_0-\widetilde C_1\big( \frac{\tilde \rho}{2d} + \veps \|x^-\|\big)\widetilde \cV(x), \text{ for $x\in \RR^d$}\,.
\end{align}
Here, $x^-\doteq (x_1^-, x^-_2,\ldots, x^-_d)\,.$ 
Using this property, we are unable to establish if $J(x,U)<\infty$, for any $U\in \Uadm$. 
The methodology undertaken in this paper cannot be easily applied or extended to study the ERSC problems of the network models without abandonment, which we leave as future work.

\end{remark}

\subsection{Asymptotic optimality} \label{sec-AO} 
We are now in a position to state the main result of this paper. 
 \begin{theorem} \label{thm-AO} The following holds:
	\begin{align*}
	\lim_{n\to\infty} \hat{\Lambda}^n(\hat{X}^n_0) =  \Lambda. 
	\end{align*}	
\end{theorem}

The proof is given in the later sections.
 In the following, we give a brief overview of the proof which is divided into two parts: In the first part, we show the lower bound 
 \begin{align}\label{eq-overview-lb}  \liminf_{n\to\infty} \hat \Lambda^n(\hat X^n_0) \ge \Lambda
 \end{align} in Section~\ref{sec-low-bound} and in the second part, we show the upper bound \begin{align}\label{eq-overview-ub}  \limsup_{n\to\infty} \hat \Lambda^n(\hat X^n_0) \le \Lambda \end{align} in Section~\ref{sec-upp-bound}. 
We now illustrate the ideas to prove these two parts below.  
We do not explicitly give definitions of the processes and sets of controls here and refer the reader to the later sections. 
We begin by viewing the ERSC problem (for the limiting diffusion) from a different perspective, using 
the following variational representation: for $v\in \Usm$, 
\begin{align}\label{illus-BM-1}
\limsup_{T\to\infty} \frac{1}{T}\log \E\Big[e^{\int_0^T r^v(X_t)dt}\Big]=\limsup_{T\to\infty}\sup_{w\in \cA}\E\Bigg[  \frac{1}{T}\int_0^T\Big(  r^v(X^{*,v,w}_t)\ -\frac{1}{2}\|w_t\|^2\Big)dt\Bigg].
\end{align}
 Here, the set $\cA$ consists of an appropriate class of square integrable processes. This implies 
\begin{align}\label{illus-BM-2}
\Lambda= \inf_{v\in \Usm}\limsup_{T\to\infty}\sup_{w\in \cA}\E\Bigg[  \frac{1}{T}\int_0^T\Big(  r^v(X^{*,v,w}_t)\ -\frac{1}{2}\|w_t\|^2\Big)dt\Bigg].
\end{align}
Here, $w$ is a new auxiliary control,  and $X^{*,v,w}$ is an extended process (see equation \eqref{X-control}) associated with the limiting diffusion $X$ and that auxiliary control. When $w\equiv 0$, we have $X=X^{*,v,w}$.  In Section~\ref{sec-sub-var-BM}, we show that 
\begin{align}\label{illus-BM-3} \Lambda= \sup_{w\in \Wsm} \inf_{v\in \Usm} \limsup_{T\to\infty}\E\Bigg[  \frac{1}{T}\int_0^T\Big(  r^v(X^{*,v,w}_t)\ -\frac{1}{2}\| w_t\|^2\Big)dt\Bigg].\end{align}
Here, $\Wsm\subset \cA$ is an appropriate set of stationary Markov controls. Note the switch between $\sup_{w\in \Wsm}$ and $\inf_{v\in \Usm}$.

Next, we re-express the ERSC cost for diffusion-scaled queueing process in the following way. For a given SCP $Z^n$ (with $\hat Q^n_t$ depending on $Z^n$), we show the following representation to hold in Section~\ref{sec-sub-var-PP}: 
\begin{align}\label{illus-PP}
\limsup_{T\to\infty} \frac{1}{T}\log \E\Big[e^{\int_0^T \widetilde r(\hat Q^n_t)dt}\Big] =\limsup_{T\to\infty}&\sup_{{\uppsi \in \cE^n}}\E\Bigg[\frac{1}{T} \int_0^T \big(\widetilde r(\hat Q^{n,\uppsi}_t) -\lambda^n\varkappa(\phi_{t})-n\mu^n\varkappa(\psi_{t})-n\gamma^n\varkappa(\varphi_{t})\big)dt\Bigg]\,.
\end{align}
Here, $\uppsi=(\phi,\psi,\varphi)$ and $\hat Q^{n,\uppsi}$ are a new auxiliary control and an extended process of $\hat Q^{n}$ associated with that control, respectively, and $\varkappa(r)= r\ln r-r+1$ which is the relative entropy of the Poisson distribution (with expectation $r$) with respect to the standard Poisson distribution. 
Also, the set $\cE^n$ is the collection of all progressively measurable (with respect to the natural filtration of the underlying Poisson processes) positive real-valued functions $\uppsi$ such that 
$$ \int_0^T\big(\lambda^n \varkappa(\phi_{t})-n\mu^n\varkappa(\psi_{t})-n\gamma^n\varkappa(\varphi_{t})\big)dt <\infty, \text{ for every $T$}\,.$$

When $\uppsi\equiv ( e,e,e)$, we have $\hat Q^n\equiv \hat Q^{n,\uppsi}$. When $\uppsi = \uppsi^n$, 
we write $\hat Q^{n,\uppsi^n}$. The process  $\hat X^{n,\uppsi}$ is defined similarly with $\hat X^{n,\uppsi}\equiv\hat X^n$, whenever $\uppsi\equiv (e,e,e)$. See \eqref{eq-prelimit-controlled}. Recall that $e=(1,1,\ldots,1)\transp\in \RR^d$.

These variational representations are crucial in the proofs of the lower and upper bounds. We illustrate how they are used below. 

\noindent{\bf Sketch proof of the lower bound in \eqref{eq-overview-lb} (Lemma~\ref{lem-AO-lower}):}\\  For every $n$, we choose a nearly optimal SCP $Z^n$ for $\hat \Lambda^n(\hat X^n_0)$ and for such an SCP, we apply~\eqref{illus-PP}.  We then make a particular choice of $\uppsi^n=(\phi^n,\psi^n,\varphi^n)$ which is a priori sub-optimal (with respect to the supremum in~\eqref{illus-PP}). This choice (which is motivated from Theorem~\ref{thm-fclt-poisson}) is 
$$ \uppsi^n_t= \Bigg(e-\frac{w^*(\hat X^{n,\uppsi}_t)}{\sqrt n}, e-\frac{w^*(\hat X^{n,\uppsi}_t)}{\sqrt n},e\Bigg)$$
with $w^*$ being an appropriate nearly optimal (maximizing) $w$ in~\eqref{illus-BM-3}
such that 
\begin{align}\label{illus-BM-5}
\Lambda \leq \int_{\RR^d} \big( r^v(x)\ -\frac{1}{2}\| w^* (x)\|^2\big)\widetilde{\mu}^{*}_{v,w} (dx) +\delta,
\end{align}
for  the ergodic occupation measure $\widetilde{\mu}^{*}_{v,w}$ of $X^{*,v,w^*}$, for any $v\in \Usm$ and for small $\delta>0$. 
The main hurdle in proving the above display is to prove positive recurrence of $X^{*,v,w^*}$ for the ergodic occupation measure to be well-defined. We achieve this by showing that there are nearly optimal controls $w^*$ that vanish outside a large enough ball. This ensures that stability properties of $X$ can be borrowed by $X^{*,v,w^*}$ as their respective infinitesimal generators coincide outside a large ball.
Finally, we show that 
\begin{align*} \limsup_{T\to\infty}&\sup_{{\uppsi \in \cE^n}}\E\Bigg[\frac{1}{T} \int_0^T \big(\widetilde r(\hat Q^{n,\uppsi^n}_t) -\lambda^n\varkappa(\phi^n_{t})-n\mu^n\varkappa(\psi^n_{t})-n\gamma^n\varkappa(\varphi^n_{t})\big)dt\Bigg]
\end{align*}
	converges as $n\to\infty$ (or at least along a subsequence) to the right hand side of \eqref{illus-BM-5}. 
From~\eqref{illus-BM-5}, this is bounded from below by $\Lambda-\delta$. 
 This proves~\eqref{eq-overview-lb}.

\noindent{\bf Sketch proof of the upper bound in \eqref{eq-overview-ub} (Lemma~\ref{lem-AO-upper}):}\\ We begin by choosing a nearly optimal control (for the limiting diffusion) that is stationary, Markov and continuous (Lemma~\ref{lem-cont-control} guarantees that such a control exists). Using this control we explicitly construct an SCP (which is a priori sub-optimal for the ERSC problem for the diffusion-scaled queueing process) using the construction in \cite{ABP15}. For this constructed SCP,  we apply~\eqref{illus-PP} and choose $\uppsi^n=(\phi^n,\psi^n,\varphi^n)$ that is nearly optimal (with respect to the supremum in~\eqref{illus-PP}). This will then give us
\begin{align*}
\hat \Lambda(\hat X^n_0)\leq \limsup_{T\to\infty}\E\Bigg[\frac{1}{T} \int_0^T \big(\widetilde r(\hat Q^{n,\uppsi^n}_t) -\lambda^n\varkappa(\phi^n_{t})-n\mu^n\varkappa(\psi^n_{t})-n\gamma^n\varkappa(\varphi^n_{t})\big)dt\Bigg]+\delta,
\end{align*}
for small $\delta>0$. 
It then remains to identify the limit as $T\to\infty$ and let $n\to\infty$. To identify the  limit as $T\to\infty$, it is necessary to show that the family of the mean empirical measures of $\hat Q^{n,\uppsi}$ is tight. This is not at all obvious as this cannot be inferred directly from the stability of the process $\hat Q^n$. We achieve this by first showing the tightness of the mean empirical measures of the joint processes $\big(\hat Q^{n,\uppsi^n}, h^n(\uppsi^n)\big)$ in $T,n$ with $h^n(\uppsi^n)= \big(\sqrt{n}(e-\phi^n),\sqrt{n}(e-\psi^n),\sqrt{n}(e-\psi^n)\big)$. 
We introduce a suitable topology to prove  the tightness of the mean empirical measures of $h^n(\uppsi^n)$, and then use the tightness of $h^n(\uppsi^n)$ along with a Lyapunov function (motivated from \cite{AHP21}) to show that the family of the mean empirical measures of $\hat Q^{n,\uppsi^n}$ is tight (in fact, both in $n$ and $T$). 
Finally, from this tightness, we can show that the right hand of the above display is bounded from above by  $\Lambda+\delta$.
This proves~\eqref{eq-overview-ub}.

Before proceeding to prove these bounds, we introduce the variational representations and present some preliminary results in the next section.

\section{Variational formulations of ERSC problems}\label{sec-var}

\subsection{Variational formulation for the limiting diffusion}\label{sec-sub-var-BM}

We develop a variational formulation of ERSC problem for the limiting diffusion. Moreover, in Theorem~\ref{thm-sup-inf}, we also show that the optimal value of the ERSC problem for the limiting diffusion can be represented as the optimal value of a TP-ZS SDG with an extended running cost. This is the main result of this subsection.
The fundamental result we use is the following variational representation of exponential functionals of Brownian motion (\cite[Theorem 5.1]{boue1998}). 

\begin{theorem}\label{thm-var-rep-BM-gen} For $T>0$, suppose $G:\frC^d_T\rightarrow \RR$ is a non-negative Borel measurable function. Then the following holds: 
	\begin{align}\label{eq-var-rep}
	\frac{1}{T}\log \E[e^{TG(W)}]= \sup_{w\in\cA} \E\bigg[ G\bigg(W_{(\cdot)}+\int_0^{\cdot}w_sds\bigg) - \frac{1}{2T}\int_0^T\|w_s\|^2ds\bigg],
	\end{align}
	where $\cA$ is the set of all $\cG_t$--progressively measurable functions $w:\RR_+\rightarrow \RR^d$ such that 
\begin{align}\label{def-A} \frac{1}{T}\E\Big[\int_0^T\|w_s\|^2ds\Big]<\infty, \quad \text{ for every $T>0$}.\end{align}
 Here $\cG_t$ is the filtration generated by $\{W_s:0\leq s\leq t\}$ such that $\cG_0$ includes all the $\PP$--null sets. 
\end{theorem}

\begin{remark}\label{rem-var-rep-proof-BM} 
We give an idea of the proof (see \cite{boue1998} for the details). 
The starting point  is the following well-known entropy formula (\cite[Proposition 2.2]{budhiraja2019analysis}): For some measurable space $(\mathcal{X},\mathfrak{X})$ and  bounded measurable function $f:\mathcal{X}\rightarrow \RR$ and measure $\mu$ on $(\mathcal{X},\mathfrak{X})$, we have
\begin{align}\label{eq-entropy} \log\int_\mathcal{X} e^{f(x)}\mu(dx)= \sup_{\nu\in \calP(\mathcal{X})}\Big(\int_{\mathcal{X}} f(x)\nu(dx) - R(\nu||\mu)\Big). \end{align}
Here, $R(\nu||\mu)$ is the relative entropy defined as $\int_{\mathcal{X}} \log\frac{d\nu}{d\mu}(x) \nu(dx)$ whenever $\nu$ is absolutely continuous (with respect to $\mu$ with $\frac{d\nu}{d\mu}$ being the associated Radon-Nikodym derivative) and as $\infty$ otherwise.  We apply the above formula to the Wiener measure. Using Girsanov's theorem, we can construct a certain family of absolutely continuous measures, whose corresponding Radon-Nikodym derivatives can be easily computed. The crucial part is to show that one can replace the supremum over $\calP(\mathcal{X})$ in~\eqref{eq-entropy} with the supremum over the family of the aforementioned absolutely continuous measures given by Girsanov's theorem. This gives us~\eqref{eq-var-rep}.
\end{remark}
\begin{remark}
	The representation in Theorem~\ref{thm-var-rep-BM-gen} was first used in \cite{boue2001} in the context of ERSC where the authors studied the problem of maximizing the escape times of diffusion under the control (which is associated with a ERSC cost function). The authors also studied the small noise asymptotics of the maximizing control (see \cite[Lemma 4.5 and Theorem 4.6]{boue2001}).
\end{remark}

For $v\in \Usm$, it is easy to see that Theorem~\ref{thm-var-rep-BM-gen} can be applied to the case of ERSC cost for the limiting diffusion $X $ in \eqref{eqn-hatX}, by choosing $G(W)= \frac{1}{T}\int_0^T r(X_t,v(X_t))dt$, with $X_t$ being regarded as a functional of the driving Brownian motion $W$ (noting that we write $X$ to indicate the dependence on the control $v$ explicitly).
In the rest of this sub-section, we fix $X_0=x$.

For $w\in \cA$, define  ${X}^{*,v,w}$ as the solution to 
\begin{align}\label{X-control} d{X}^{*,v,w}_t=  b({X}^{*,v,w}_t, v(X^{*,v,w}_t))dt +\Sigma w_tdt + \Sigma dW_t\,, \quad  \text{ with }\,\,  {X}^{*,v,w}_0=x.\end{align}
Just as $X_0$, we also fix $X^{*,v,w}_0=x$, for the rest of this sub-section.
\begin{remark}
	The reader will notice that all the results in the rest of the section hold if one replaces $\Sigma$ with a general function $\Sigma:\RR^d\to\RR^{d\times d}$ which is Lipschitz continuous and such that there exists a $\sigma>0$ with the following property: $\Sigma(x)\Sigma(x)\transp \geq \sigma \|x\|^2$, $x\in \RR^d.$ 
\end{remark}

As an immediate corollary to Theorem~\ref{thm-var-rep-BM-gen}, we have the following result. 
\begin{corollary}\label{cor-var-rep-risk-B}For $v\in \Usm$,
	\begin{align}\label{eqn-erg-cont-var-rep}
	\Lambda_v= \limsup_{T\to\infty}\sup_{w\in\cA} \E\Bigg[ \frac{1}{T}\int_0^T \Big( r(X^{*,v,w}_t, v(X^{*,v,w}_t)) - \frac{1}{2}\|w_t\|^2\Big)dt\Bigg].
	\end{align}

\end{corollary} 
Note that in~\eqref{eqn-erg-cont-var-rep}, we are optimizing over $\cA$, which is a subset of all $\cG_t$--progressively measurable processes. In other words, $w_t= g(W_{[0,t]},t),$ for some Borel measurable $g: \frC^d\times \RR_+\rightarrow \RR^d.$ 
\begin{remark}
The analysis that follows is applicable even when $v$ is a relaxed Markov control. It is straightforward to show that the infimum of the ERSC cost over the set of all relaxed controls is greater than or equal to the optimal ERSC cost. We reserve the same notation $X^{*,v,w}$, even when $v$ is a relaxed Markov control. In this case, $b(x,v(x))$ and $r (x,v(x))$ will be replaced by $\int_\bU b(x,u)v(du|x)$ and $ \int_\bU r(x,u) dv(du|x)$, respectively. Also, $\Lg^v $ is defined with $\int_\bU b(x,u)v(du|x)$ instead of $b(x,v(x))$.  Therefore, we restrict ourselves with only precise Markov controls in this section. 
\end{remark}

We now  write $\Lambda_v$ as the optimal value of a  CEC problem for an extended diffusion over an auxiliary control.  To that end, let $\Wadm$ be the set of admissible $\RR^d$--valued controls (here we use admissibility as in  Definition~\ref{def-admissible} with $\bU$ replaced by $\RR^d$)  $ w$. Since every $ w\in\Wadm$ is of the form $w_t=  w(W_{[0,t]},t)$, we have $\Wadm\subset \cA$.  
 
Now, define the following ergodic control problem
\begin{align}\label{eqn-erg-cont-aug}
\widetilde \Lambda_v\doteq \sup_{ w\in \Wadm}\limsup_{T\to\infty}\E\Bigg[  \frac{1}{T}\int_0^T\Big( r^v(X^{*, v,w}_t)\ -\frac{1}{2}\| w_t\|^2\Big)dt\Bigg].
\end{align}

 We next show that $\widetilde \Lambda_v=\Lambda_v$, \emph{i.e.,} 
$\Lambda_v$ is equal to the optimal value of the CEC problem of $w$ in~\eqref{eqn-erg-cont-aug}. To begin with, we denote $\Wsm\subset \Wadm$ as the set of stationary Markov controls (including relaxed Markov controls) and then first study the following CEC problem: For $l>0$, let $\Wsm(l)$ be the set of all stationary Markov controls (which is a subset of $\Wadm$) that are of the form $w=w(\cdot )$ on $B_l$ and $w=0$ on $B_l^c$,   and satisfy $\sup_{x\in \RR^d}\|w(x)\|\leq l$. Also, let $\chi_l:\RR^d\rightarrow \RR$  be a continuous function that satisfies  	$ \chi_l(x)=0$, whenever $x\in B_{ l}^c$ and $\chi_{ l}(x)=1$, whenever $x\in B_{\frac{ l}{2}}$.  Define 
\begin{align}\label{def-game-cost}
f_{l}(x,u,w)\doteq
r(x,u)\wedge l-\frac{1}{2}\|\chi_l(x) w\|^2  \quad \text{ and } \quad \Delta_l(x,w)\doteq \chi_l(x)
\Sigma w\,.
\end{align}
For $v\in \Usm$, we set $f_l^v(x,w)= f_l(x,v(x),w)$. Finally, define
$$ \Lambda_v(l)\doteq \sup_{w\in \Wsm(l)} \limsup_{T\to\infty} \frac{1}{T}\E\Big[\int_0^T \Big(r^v(X^{*,v,w}_t)\wedge l-\frac{1}{2}\|w(X^{*,v,w}_t)\|^2\Big)dt\Big]\,.$$
We then consider the limit as $l\to\infty$. There are multiple reasons for using such an approach:

\begin{enumerate}
	\item[(i)] For every $l>0$, the process $X^{*,v,w}$ is exponentially ergodic, as outside $B_l$, the generator of $X^{*,v,w}$ is the same as that of $X$ (under control $v\in \Usm$) which is exponentially ergodic.
	\item [(ii)] We can use a {variant of}  the method of spatial truncation (introduced in \cite[Section 4]{ABP15}) where a CEC problem can be shown to be approximated arbitrarily well, by a class of CEC problems with an appropriate truncated version of the original drifts and running costs.  
	\item [(iii)] This approach also tells us that we can find nearly optimal stationary Markov controls which are supported only on a large ball (this property is used in proof of Lemma~\ref{lem-AO-lower}).
\end{enumerate}

We state and prove a modified version of \cite[Theorem 4.1]{ABP15}.
\begin{proposition} For $v\in \Usm$
	and  every $l>0$, there exists a pair  $(\upu^v_l,\rho^v_l)\in W^{2,p}_{\text{loc}}(\RR^d)\times \RR$, for any $p>d$, with $\upu^v_l(0)=1$, satisfying the equation
	\begin{align}\label{eq-trunc}
	 \Lg^v \upu^v_l + \max_{w\in \RR^d}\big\{ f^v_l(x,w) + \Delta_l(x,w)\cdot\grad \upu^v_l\big\}= \rho^v_l.
	\end{align}
	Moreover, $\rho^v_l=\Lambda_v(l)$ and $\rho^v_l$ is non-decreasing in $l$.

\end{proposition}
\begin{proof}It is evident that for every $l>0$ and $w\in \Wsm(l)$, $X^{*,v,w}$ is exponentially ergodic and therefore satisfies the hypothesis of \cite[Theorem 4.1]{ABP15} and this proves the result.  
\end{proof}

The following lemma studies how $\{\upu^v_l\}_{l>0}$ and $\{\rho^v_l\}_{l>0}$ behave as $l\uparrow \infty$. From \cite[Lemma 2.4]{ari2018strict}, for  $v\in \Usm$, there is a unique pair $(\Psi^v,\Lambda^v)$ such that $\Psi^v>0$, $\Psi^v\in W^{2,d}_{\text{loc}}(\RR^d)$ and
\begin{align}\label{eq-groundstate}
\Lg^v\Psi^v(x)+r^v(x)\Psi^v(x)=\Lambda^v\Psi^v(x).
\end{align} 
\begin{lemma}\label{lem-limit-truc} For every $v\in \Usm$, 
	there exists a unique pair $(\Phi_v,\widetilde \Lambda_v)$ such that as $l\to\infty$, 
	$\rho^v_l\to \widetilde \Lambda_v$ and $\upu^v_l \to \Phi^v\doteq \log \Psi^v$ in $W^{2,p}_{\text{loc}}(\RR^d)$, for any $p>d$. Moreover, $\widetilde \Lambda_v=\Lambda_v$.
\end{lemma}

\begin{remark}
	The content of \cite[Theorem 4.1]{ABP15} also addresses the aforementioned convergence. The reason for giving a separate lemma is that in our case, as $l\to\infty$, the space $\Wsm$ (considered as an appropriate limit of $\Wsm(l)$ as $l\uparrow \infty$) is no longer compact and the topology of Markov controls (as done in \cite[Section 2.4]{ari2011}) may not be compact and/or metrizable. But this will not be an issue for us as will be evident from the proof. 
\end{remark}
\begin{remark}
	A result similar to this is showed in \cite[Section 2.4]{ari2018strict}, but under the following condition on the cost function $r$: $r(x,u)\to 0$ as $\|x\|\to\infty$. This is too restrictive for our setting of multiclass queueing networks. 
	
\end{remark}

\begin{proof}  We first prove that $\Lambda_v\geq \widetilde \Lambda_v$. To do this, we fix $\delta>0$ and choose $ \bar w^*\in \Wadm$ such that 
	$$\widetilde \Lambda_v\leq  \limsup_{T\to\infty}\E\Bigg[  \frac{1}{T}\int_0^T\Big( r^v(X^{*,v,\tilde w^*}_t)\ -\frac{1}{2}\|\tilde w^*_t\|^2\Big)dt\Bigg]+\delta.$$
	Since $\bar w^*\in \cA$, we have
	\begin{align*}
	\limsup_{T\to\infty}\frac{1}{T}\E\Bigg[ \int_0^T \Big( r^v(X^{*,v,\bar w^*}_t) - \frac{1}{2}\|\bar w^*_t\|^2\Big)dt\Bigg]&\leq 	\limsup_{T\to\infty}\frac{1}{T}\sup_{w\in\cA} \E\Bigg[ \int_0^T \Big( r^v(X^{*,v,w}_t) - \frac{1}{2}\|w_t\|^2\Big)dt\Bigg]\\
	\implies \widetilde \Lambda_v -\delta&\leq \Lambda_v \,.
	\end{align*}
		This consequently gives us $\Lambda_v\geq \widetilde \Lambda_v$.  

	From here, the finiteness of $\Lambda_v$ implies the finiteness of  $\widetilde \Lambda_v$.  Since $ \rho^v_l\leq \widetilde\Lambda_v$, $\{\rho^v_l\}_{l>0}$ is convergent along a subsequence (with, say, $\rho^{*,v}\leq \widetilde \Lambda_v$ as the limit point). Using the standard elliptic regularity theory (arguments similar to \cite[Theorem 3.5.2 and Lemma 3.5.3]{ari2011}),
	 we can then conclude that $\upu^v_l$ converges to some function $\upu^v$ in $W^{2,p}_{\text{loc}}(\RR^d)$ that satisfies
	\begin{align}\label{eq-trunc-limit}
	\Lg^v \upu^v + r^v(x)+ \max_{w\in \RR^d}\big\{\Sigma w.\grad \upu^v  - \frac{1}{2}\|w\|^2\big\}= \rho^{*,v}.
	\end{align}
	It is clear that~\eqref{eq-trunc-limit} can be rewritten as
	$$ 	\Lg^v \upu^v + r^v+ \frac{1}{2}\|\Sigma\transp\grad \upu^v\|^2= \rho^{*,v}.$$
	By making a substitution $\widetilde \upu^v= e^{\upu^v}$, we have
	$$ 	\Lg^v \widetilde \upu^v + r^v\widetilde\upu^v= \rho^{*,v}\widetilde \upu^v.$$
	Using \cite[Corollary 2.1]{ari2018strict}, we can conclude that $\rho^{*,v}\geq \Lambda_v$. This proves that $\Lambda_v=\widetilde \Lambda_v$ as we already know that $\widetilde \Lambda_v\geq \rho^{*,v}$ and that $\Lambda_v\geq \widetilde\Lambda_v$.
	Since \cite[Lemma 2.4]{ari2018strict} gives us a unique solution to the above equation, $\widetilde \upu^v=\Psi^v$. This completes the proof as $\upu^v=\Phi^v$.	
\end{proof}

We now proceed to prove the main result of this section. In addition to proving that $\Lambda$ is equal to the optimal value of a  TP-ZS SDG, we also show that such a game can be written as the limit of a family of TP-ZS SDGs where the strategies of the maximizing player are  only allowed to be compactly supported (see~\eqref{eq-lim-sup-inf} below). Before we state and prove this result, we state an important existing result in the TP-ZS SDG taken from   \cite[Section 4.2] {borkar1992stochastic},  where both the minimizing and the maximizing strategies take values in compact spaces, and  moreover, the running cost function is bounded.   To that end, we first define a family of TP-ZS SDGs. For $l>0$,  let 
\begin{align*}
J^{*,l}_{v,w}&\doteq \limsup_{T\to\infty} \frac{1}{T}\E\Big[ \int_0^T \big(r^v(X^{*,v,w}_t)\wedge l-\frac{1}{2}\|w(X^{*,v,w}_t)\|^2\big)dt\Big],\\
\overline \rho_l&\doteq \inf_{v\in \Usm} \sup_{w\in  \Wsm(l)} J^{*,l}_{v,w}\,,\\
\underline \rho_l&\doteq  \sup_{w\in  \Wsm(l)} \inf_{v\in \Usm}J^{*,l}_{v,w}\,. 
\end{align*}

We have suppressed the dependence of $J^{*,l}_{v,w}$ on $X^{*,v,w}_0=x$, as it is fixed for this entire sub-section.
Recall that $w\in\Wsm(l)$ is only compactly supported and hence, the process $X^{*,v,w}$ is exponentially ergodic from Proposition~\ref{prop-lyap-with-abandonment}. This allows us to invoke Theorems 4.5 and 4.6 of \cite{borkar1992stochastic} to give us the following.

\begin{theorem}\label{thm-2p-game}
There exists a unique pair  $(\upu_l,\rho_l)\in W^{2,p}_{\text{loc}}(\RR^d)\times \RR$, $2\leq p<\infty$ such that the following hold.  
\begin{enumerate}
\item [(i)] \begin{align}\label{eq-hjb-game}
&\min_{u\in \bU}\max_{w:\|w\|\leq l} \Big\{\Lg^u \upu_l(x) + f_l (x,u,w)+ \Delta_l(x,w)\cdot \nabla \upu_l(x)\Big\} \nonumber\\
&\qquad=\max_{w:\|w\|\leq l} \min_{u\in \bU} \Big\{\Lg^u \upu_l(x)+ f_l(x,u,w)+ \Delta_l(x,w)\cdot \nabla \upu_l(x)\Big\}=\rho_l\,, \text{ for } x\in \RR^d\,. 
\end{align}
\item[(ii)] $\rho _l=\overline \rho_l=\underline \rho_l$. 
\item [(iv)] $v^*\in \Usm $ satisfies $\sup_{w\in \Wsm(l)}J^{*,l}_{v^*,w}= \rho_l$ if and only if for a.e. $x\in \RR^d$,
\begin{align*} \min_{u\in \bU}\max_{w:\|w\|\leq l} \Big\{\Lg^u \upu_l (x)&+ f_l (x,u,w)+ \Delta_l(x,w)\cdot \nabla \upu_l(x)\Big\}\\
&= \max_{w:\|w\|\leq l} \Big\{\Lg^{v^*} \upu_l(x) + f^{v^*}_l (x,w)+ \Delta_l(x,w)\cdot \nabla \upu_l(x)\Big\}\,.\end{align*}
\item [(v)] $w^*\in \Wsm(l)$ satisfies $\inf_{v\in \Usm} J^{*,l}_{v,w^*}=\rho_l$ if and only if for a.e. $x\in \RR^d$,
\begin{align*} \min_{u\in \bU} \Big\{\Lg^u \upu_l(x)&+ f_l(x,u,w^*(x))+ \Delta_l(x,w^*(x))\cdot \nabla \upu_l(x)\Big\}\\
&=\max_{w:\|w\|\leq l} \min_{u\in \bU} \Big\{\Lg^u \upu_l(x)+ f_l(x,u,w)+ \Delta_l(x,w)\cdot \nabla \upu_l(x)\Big\}\,.\end{align*}
Moreover, $w^*$ is bounded and continuous on $\RR^d$.
\end{enumerate}
\end{theorem}
Using the above theorem, we have the following result. For $v\in \Usm$ and $w\in \Wsm$, define 
\begin{align*} J^*_{v,w}&\doteq \limsup_{T\to\infty} \frac{1}{T}\E\Big[ \int_0^T \big(r^v(X^{*,v,w}_t)-\frac{1}{2}\|w(X^{*,v,w}_t)\|^2\big)dt\Big]\,.
\end{align*} 
We again suppressed the dependence of $J^*_{v,w}$ on $X^{*,v,w}_0=x$ as it is fixed for this entire sub-section.

 \begin{theorem}\label{thm-sup-inf} The following equalities hold:
\begin{align}\label{eq-inf-sup}
\Lambda&= \inf_{v\in \Usm} \sup_{w\in \Wsm} J^*_{v,w}\\\label{eq-sup-inf}
&=  \sup_{w\in \Wsm}\inf_{v\in \Usm} J^*_{v,w}\\\label{eq-lim-sup-inf}
&=\lim_{l\to\infty}   \sup_{w\in \Wsm(l)}\inf_{v\in \Usm} J^*_{v,w}\,.
\end{align}
\end{theorem}

\begin{remark} The main content of this theorem is establishing the switch of supremum and infimum operations  in the context of TP-ZS SDGs, where the maximizing strategy takes values in non-compact spaces (the Euclidean space) and the associated running cost $\underline r(x,u,w)\doteq r(x,u)-\frac{1}{2}\|w\|^2$ is neither bounded from below nor above. To the best of our knowledge, the switch of supremum and infimum operations in~\eqref{eq-inf-sup} and~\eqref{eq-sup-inf} is novel under the aforementioned setup.
\end{remark}

\begin{proof} 

\eqref{eq-inf-sup} is an immediate consequence of Lemma~\ref{lem-limit-truc}. Since $\inf_{x}\sup_{y} f(x,y)\geq \sup_y\inf_x f(x,y)$, using~\eqref{eq-inf-sup}, we have $$  \Lambda\geq \sup_{w\in \Wsm}\inf_{v\in\Usm}J^*_{v,w}\,.$$
Therefore, to prove~\eqref{eq-sup-inf}, we show that for any $\delta>0$, there exists $w^*\in \Wsm$ such that 
\begin{equation}\label{eq-sup-inf-rep-1}
\Lambda\leq \inf_{v\in\Usm}J^*_{v,w^*}+\delta\,.
\end{equation} 
To do this, for now we assume that  the following holds: for every $\delta>0$, there exists $l_0>0$ such that whenever $l>l_0$, we have
\begin{align}\label{eq-claim} \Lambda_v\leq \sup_{w\in \Wsm(l)} J^{*,l}_{v,w}+\delta,\end{align}
 for every $v\in \Usm$. This claim is proved in Lemma~\ref{lem-unif-int} at the end of this section. Since $\rho_l= \inf_{v\in\Usm} \sup_{w\in \Wsm(l)} J^{*,l}_{v,w}$ from Theorem~\ref{thm-2p-game}, we choose $v^*_l\in \Usm$ such that $\sup_{w\in \Wsm(l)} J^{*,l}_{v^*_l,w}= \rho_l$ to obtain
 \begin{align*}
 \Lambda\leq \Lambda_{v^*_l}\leq \sup_{w\in \Wsm(l)} J^{*,l}_{v^*_l,w}= \rho_l+ \delta\,.
 \end{align*}
This proves that 
$\Lambda\leq \liminf_{l\to\infty} \rho_l$. Now observing that $J^{*,l}_{v,w}\leq J^*_{v,w}$, for $w\in \Wsm(l)$, we get
\begin{align}\label{eq-lim-sup-inf-2} \Lambda\leq \limsup_{l\to\infty} \sup_{w\in \Wsm(l) }\inf_{v\in \Usm} J^{*,l}_{v,w}\leq  \limsup_{l\to\infty} \sup_{w\in \Wsm(l) }\inf_{v\in \Usm} J^*_{v,w}\,.\end{align}
Hence,~\eqref{eq-sup-inf-rep-1} is proved.

Finally, to prove~\eqref{eq-lim-sup-inf}, we make an easy observation from~\eqref{eq-sup-inf} that for every $l>0$, we have
\begin{align*}
\Lambda \geq   \sup_{w\in \Wsm(l)}\inf_{v\in \Usm}J^*_{v,w}\,.
\end{align*}
Combining the above display with~\eqref{eq-lim-sup-inf-2}, we obtain~\eqref{eq-lim-sup-inf}. 
This proves the theorem.
\end{proof}
\begin{remark}\label{rem-markov}In the above analysis, observe that we choose the fixed $v$ to be a relaxed control and not a general admissible control. The reason for doing this is as follows: suppose we choose $v$ to be some non-Markov control. Then the running cost function for the maximization problem in~\eqref{eqn-erg-cont-aug} is a priori dependent on the entire past. This restricts us from using the principle of dynamic programming and the analysis above is not applicable. 
\end{remark}

Therefore, all that remains to show is that~\eqref{eq-claim} holds. In the following, we achieve this  in Lemma~\ref{lem-unif-int} (which is also used in the proof of lower bound) and also prove another simple yet important result (Lemma~\ref{lem-trivial-bound}) which is used in the proof of the upper bound. Both these results involve the notion of  the mean empirical (occupation) measure and the ergodic occupation measure of the extended diffusion process.  For any $\Uadm\ni u_t=u(t, W_{[0,t]})$, let  $X^{*,u,w}$ be a diffusion process as a solution to 
\begin{align}\label{X-control-adm} d{X}^{*,u,w}_t=  b({X}^{*,u,w}_t, u(t,W^w_{[0,t]}))dt +\Sigma w_tdt + \Sigma dW_t\,,  \text{ with } {X}^{*,u,w}_0=x.\end{align}
Here, $W^w_\cdot=W_\cdot+ \int_0^\cdot w_tdt.$ To keep expressions concise, we write $u^w_t=u(t,W^w_{[0,t]})$. In case that $u$ is a Markov control,  we know that $u$ depends on $W$ through the diffusion, that is, we have $u^w_t= u(X^{*,u,w}_t)$. In other words, for any $v\in \Usm$ and $w\in \cA$, $X^{*,v,w}$ is the solution to~\eqref{X-control}. 

The mean empirical (occupation) measure is defined as follows.  For a given $u\in \Uadm$ and $w\in \cA$,
\begin{align}\nonumber \mu^{*,T}_{u,w}(A\times B\times C)&\doteq \frac{1}{T}\int_0^T \Ind_{\{(X^{*,u,w}_t,u^w_t,w_t)\in A\times B\times C\}} dt\,,\\\nonumber
\mu^{*,T,1}_{u,w}(A\times B)&\doteq \frac{1}{T}\int_0^T \Ind_{\{(X^{*,u,w}_t,u^w_t)\in A\times B\}} dt\,,\\
\label{def-mu*} \mu^{*,T,2}_{u,w}(A\times C)&\doteq \frac{1}{T}\int_0^T \Ind_{\{(X^{*,u,w}_t,w_t)\in A\times C\}} dt\,, \\\nonumber
\mu^{*,T,3}_{u,w}(A)&\doteq \frac{1}{T}\int_0^T \Ind_{\{X^{*,u,w}_t \in A\}} dt\,,
\end{align}
for any Borel sets $A,C\subset \RR^d$ and $B\subset \bU$.   For $v\in \Usm$  and $w\in  \Wsm$, we write
$$ \widetilde \mu^{*,T}_{v,w} (A)\doteq \frac{1}{T}\int_0^T \Ind_{\{X^{*,v,w}_t\in A\}}dt\,, \text{ for any Borel set $A\subset \RR^d.$}$$ 
We represent their weak limits (if they exist and the associated subsequence $T_k$ is irrelevant) by $ \mu^{*}_{u,w}, \mu^{*,1}_{u,w}, \mu^{*,2}_{u,w}, \mu^{*,3}_{u,w},\widetilde\mu^{*}_{v,w}$, respectively.

Observe that using the above notation,  we can rewrite $\Lambda_v$ in \eqref{eqn-erg-cont-var-rep}  as
 $$ \Lambda_v= \limsup_{T\to\infty} \sup_{w\in \cA}\int_{\RR^d\times  \RR^d} \big(r^v(x)-\frac{1}{2} \|y\|^2\big) d\mu^{*,T,2}_{v,w}(x,y)\,.$$
 Using Theorem~\ref{thm-sup-inf}, we can further write $\Lambda_v$ as
 $$ \Lambda_v= \sup_{w\in \Wsm}\inf_{v\in \Usm}\int_{\RR^d} \big(r^v(x)-\frac{1}{2} \|w(x)\|^2\big) d\widetilde\mu_{v,w}^{*}(x).$$

Before we proceed to prove that~\eqref{eq-claim} holds, we state and prove a necessary tightness result concerning the mean empirical measures of the pair $(X^{*,v,w^*},w^*)$, whenever $v\in \Usm$ and $w^*\in \cA$ is a nearly optimal auxiliary control.
\begin{lemma}\label{lem-unif-int-0} For every $v\in \Usm$, let $w^*=w^*(\delta,T)\in \cA$ be such that 
$$\sup_{w\in \cA}\E\Bigg[ \frac{1}{T}\int_0^T \Big( r^v(X^{*,v,w}_t) - \frac{1}{2}\|w_t\|^2\Big)dt\Bigg]\leq  \E\Bigg[ \frac{1}{T}\int_0^T \Big( r^v(X^{*,v,w^*}_t) - \frac{1}{2}\|w^*_t\|^2\Big)dt\Bigg]+\delta, $$
for every $T>0$. Then, 
$$\limsup_{T\to\infty}\frac{1}{T} \E\Big[ \int_0^T \|X^{*,v,w^*}_t\|^2dt \Big] \leq M \quad \text{ and } \quad \limsup_{T\to\infty}\frac{1}{T}\E\Big[\int_0^T \|w^*_t\|^2 dt\Big]\leq M\,.$$
Here, $M>0$ is a constant that depends only on $\delta$. In particular, the family of mean empirical measures $\{\mu^{*,T,2}_{v,w^*}\}_{T>0}$ is tight.

\end{lemma}
\begin{proof} Fix $\delta>0$, $v\in \Usm$ and, choose $w^*\in \cA$ such that it satisfies the hypothesis of the lemma (and  set $X^*=X^{*,v,w^*}$  in the rest of the proof). In particular, using Corollary~\ref{cor-var-rep-risk-B} we have 
\begin{align*}
\Lambda_v\leq  \limsup_{T\to\infty}\E\Bigg[ \frac{1}{T}\int_0^T \Big( r^v(X^{*}_t) - \frac{1}{2}\|w^*_t\|^2\Big)dt\Bigg]+\frac{\delta}{2}\,.
\end{align*}
Due to the non-negativity of  $\Lambda_v$, this in turn implies that for large enough $T$ (which is what we restrict ourselves to in the rest of the proof),
\begin{align}\label{eq-1}
\frac{1}{2T}\E\Big[ \int_0^T \|w^*_t\|^2dt\Big]\leq \frac{1}{T} \E\Big[\int_0^T  r^v(X^{*}_t)dt\Big]-\Lambda_v+\delta\leq \frac{1}{T} \E\Big[\int_0^T  r^v(X^{*}_t)dt\Big]+\delta\,.
\end{align}
 Now let $\widehat \cV\doteq \log \cV$, where $\cV$ is the inf-compact function from Proposition~\ref{prop-lyap-with-abandonment}. A straightforward substitution of $\cV=e^{\widehat \cV}$ in~\eqref{eq-lyap-with-abandonment} shows that
\begin{align*}
\Lg^u \widehat \cV(x)+ \frac{1}{2}\|\Sigma \nabla \widehat \cV(x)\|^2\leq  C_0- C_1\|x\|^2, \text{ for every $(x,u)\in \RR^d\times \bU$}\,.
\end{align*}

Since the generator associated with the process $X^{*}$ is \[\widehat \Lg^{u,w} f\doteq \Lg^u f+ \Sigma w\cdot \nabla f,\] a simple application of Young's inequality: for $a,b\geq 0$, $ab\leq \frac{1}{2} a^2+ \frac{1}{2}b^2$ gives us 
$$\widehat \Lg^{u,w} \widehat \cV(x)\leq C_0-C_1\|x\|^2 + \frac{1}{2}\|w\|^2 \,. $$
From here,  an application of It\^o's formula (followed by Fatou's lemma) with $X^{*}_0=x$, $u=v(X^{*}_t)$ and $w= w^*_t$ gives us the following:
\begin{align*} 
\frac{C_1}{T} \E\Big[ \int_0^T \|X^{*}_t\|^2dt \Big] 
&\leq \frac{\widehat \cV(x)}{T} + C_0 +\frac{1}{2T}\E\Big[\int_0^T \|w^*_t\|^2 dt\Big] \\
& \leq  \frac{\widehat \cV(x)}{T} + C_0 +\frac{1}{T} \E\Big[\int_0^T  r^v(X^{*}_t)dt\Big]+\delta\,.
 \end{align*}
We use~\eqref{eq-1} to obtain the second inequality above. Since $r(x,u)= \kappa\cdot\big[(e\cdot x)^+ u\big]$, it is clear that $r(x,u)\leq \frac{C_1}{2}\|x\|^2+ \widetilde M $, for some large enough constant $\widetilde M>0$ (that depends only on $\kappa$ and $C_1$). From this and the above display, we get
\begin{align*} \frac{C_1}{2T} \E\Big[ \int_0^T \|X^{*}_t\|^2dt \Big]&\leq    \frac{\widehat \cV(x)}{T} + C_0+ \widetilde M +\delta\,.
 \end{align*}
  Finally, taking $T\to\infty$, we obtain
\begin{align*}\limsup_{T\to\infty} \frac{1}{T} \E\Big[ \int_0^T \|X^{*}_t\|^2dt \Big]&\leq \frac{2 (C_0+\widetilde M+\delta)}{C_1} \,.
 \end{align*}
 From the above display and~\eqref{eq-1}, setting $M= \max\{2C_1^{-1} (C_0+\widetilde M+\delta), C_0+2\widetilde M+2\delta \}$, we have the result.
\end{proof}
\begin{lemma}\label{lem-unif-int} For every $\delta>0$, there exist $l_0>0$ and $w^*\in \Wsm(l)$ for $l>l_0$ such that for any  $v\in\Usm$,
\begin{equation*}
 \Lambda_v\leq \int_{\RR^d} \big(r^v(x)\wedge l-\frac{1}{2}\|w^*(x)\|^2\big)d\widetilde\mu_{v,w^*}^{*}(x) +\delta.
 \end{equation*} 
 In particular,  we have 
\begin{align}\label{eqn-cor-var-rep-BM1} \Lambda\leq \Lambda_v\leq\sup_{w\in \Wsm(l)} J^{*,l}_{v,w}+\delta\,.\end{align}
\end{lemma}
\begin{proof} Fix $\delta>0$ and $v\in \Usm$. Let $w^*=w^*(\delta,T)\in \cA$ be as in the hypothesis of Lemma~\ref{lem-unif-int-0} with $\delta/4$.  We again set $X^*= X^{*,v,w^*}$. Using $w^*$, we define a new $\widetilde w^*=\widetilde w^*(\delta,T)\in \cA$ as follows: for $l_1>0$, let $$ \tau^*\doteq \inf\{t>0: \|w^*_t\|>l_1 \text { or } \|X^{*}_t\| >l_1\}\, \quad \text{ and } \quad \widetilde w^*_t \doteq w^*_t \Ind_{[0,\tau^*]}(t)\,.$$ 
It is trivial to see that whenever either $\|w^*_t\|>l_1$ or $\|X^{*}_t\|>l_1$, $\widetilde w^*_t=0$ and for $t\leq \tau^*$, $\widetilde w^*_t=w^*_t$. Setting $\widetilde X^*= X^{*,v,\widetilde w^*}$, it is clear that for $t\leq \tau^*$, $X^{*}_t$ and $\widetilde X^*_t$ are identical.  We now show that for large enough $l_1$, $\widetilde w^*$ also satisfies the hypothesis of Lemma~\ref{lem-unif-int-0} with $3\delta/4$. To that end, we have
\begin{align*}
	&\E\Big[ \frac{1}{T}\int_0^T  \Big( r^{v} (\widetilde X^{*}_t) -\frac{1}{2}\|\widetilde w^*_t\|^2\Big)d t\Big]-\E\Big[ \frac{1}{T}\int_0^T \Big(  r^{v} \big(X^{*}_t) -\frac{1}{2}\|w^*_t\|^2\Big)d t\Big]\\
	&\geq \E\Big[ \frac{1}{T}\int_0^T  r^{v} (\widetilde X^{*}_t)d t\Big]-\E\Big[ \frac{1}{T}\int_0^T   r^{v} (X^{*}_t)d t\Big]\\
	&\geq  \frac{1}{T}\int_0^T  \E\Big[ r^{v} (\widetilde X^{*}_t)\Ind_{[\tau^*,\infty)}(t) \Big]d t-  \frac{1}{T}\int_0^T\E\Big[   r^{v} (X^{*}_t)\Ind_{[\tau^*,\infty)}(t)\Big] d t\\
	&\doteq  J_1(T)-J_2(T)\,.
	\end{align*}
	In the above, to get the first inequality, we use the fact that $\|\widetilde w^*_t\|\leq \|w^*_t\|$, for $t\geq 0$ and to get the second inequality, we use the fact that  $X^{*}$  and $\widetilde X^{*}$ are identical for $t\in [0,\tau^*]$. 
		From Lemma~\ref{lem-unif-int-0}, we can choose $l_1$ large enough (uniformly in $T$) such that $|J_2(T)|\leq \delta/4$.  From the fact that the generator associated with $\widetilde X^*$ coincides with the generator of $X$ under $v\in \Usm$ outside the ball of radius $l_1$ around origin (which follows from the definition of $\widetilde w^*$), using Proposition~\ref{prop-lyap-with-abandonment}, we can  again ensure that for $l_1$ large enough (uniformly in $T$),  $|J_1(T)|<\delta/4$. To summarize, we have shown that 
	$$ \Lambda_v\leq \limsup_{T\to\infty}\E\bigg[ \frac{1}{T}\int_0^T  \Big(  r^{v} (\widetilde X^{*}_t) -\frac{1}{2}\|\widetilde w^*_t\|^2\Big)d t\bigg]+\frac{3\delta}{4}\,.$$
	Again from Lemma~\ref{lem-unif-int-0}, we can conclude that $r^v(\cdot)$ is uniformly integrable  with respect to the mean empirical measures  of $(\widetilde X^{*}_{[0,T]},\widetilde w^*_{[0,T]})$. This means that for $l_2>0$ large enough, we can ensure that  
	$$ \Lambda_v\leq \limsup_{T\to\infty}\E\bigg[ \frac{1}{T}\int_0^T  \Big(  r^{v} (\widetilde X^{*}_t)\wedge l_2 -\frac{1}{2}\|\widetilde w^*_t\|^2\Big)d t\bigg]+\delta\,,$$
	 and that there exists an ergodic occupation measure $$\widetilde \mu_{v,\widehat w}^*(d x, d w)= \widetilde  \eta^*_v(d x)\widehat w(d w|x)\in \calP(\RR^d\times \RR^d),$$ for some $\widehat w\in \Wsm(l_1)$ such that  the mean empirical measures  of $(\widetilde X^{*}_{[0,T]},\widetilde w^*_{[0,T]})$ converge weakly to $\widetilde \mu_{v,\widehat w}^*$, along a subsequence $T_k$. This means that 
	$$ \limsup_{T_k\to\infty} \E\bigg[ \frac{1}{T_k}\int_0^{T_k} \Big(  r^{v} (\widetilde X^{*}_t) \wedge l_2- \frac{1}{2}\|\widetilde w^*_t\|^2\Big)d t\bigg]\leq \int_{\RR^d\times \RR^d}\Big( r^{v}(x)\wedge l_2-\frac{1}{2}\|w\|^2\Big) d \widetilde \mu_{v,\widehat w}^*(x,w) \,.$$
	In the above, we use the lower semi-continuity of the function $\|w\|^2$. Since $\widehat w\in \Wsm(l_1)$, from the arbitrariness of $\delta$ and choosing $l=\max\{l_1,l_2\}$, the proof of the lemma is complete.
\end{proof}

\begin{remark} For every $l>0$, the problem of existence of optimal solution to  
$$ \sup_{w\in \Wsm(l)}\inf_{v\in \Usm}  \int_{\RR^d} \big(r^v(x)-\frac{1}{2}\|w^*(x)\|^2\big)d\widetilde \mu_{v,w}^{*}(x)$$
over all allowed ergodic occupation measures $\widetilde \mu^*_{v,w}$ can also be solved using the convex analytic approach introduced by Borkar and Ghosh in \cite{borkar1992stochastic} in the case of  TP-ZS SDG  with long-run average expected cost criterion. This approach was originally introduced in the context of CEC problem for Markov decision process by Borkar in \cite{borkar1988convex}. We also referred the reader to \cite[Chapter 3]{ari2011} for a detailed discussion in the context of diffusions.   
\end{remark}

We now give another immediate consequence of the above analysis, which will be used in the proof of Lemma~\ref{lem-AO-upper}.

\begin{lemma}\label{lem-trivial-bound}  For $L>0$, $v\in \Usm$ and $w\in \cA$,  let $\{\mu_{v,w}^{*,T,3}\}_{K\in\NN}$ be a  tight family of measures in $\calP(\RR^d)$ with $\mu_{v,w}^{*,3}$ being a weak limit point corresponding to a sequence $T_k$. Then,
\begin{equation} \label{eqn-lem-trivial-bound}
\Lambda_v\geq \int_{\RR^d}r^v(x)\wedge Ld\mu_{v,w}^{*,3}(x) -\limsup_{T\to\infty} \frac{1}{2T}\int_0^T \|w_t\|^2 dt\,.
\end{equation} 
\end{lemma}
\begin{proof} From~\eqref{eqn-erg-cont-var-rep} and the choice of $w$, for $L>0$, we have
\begin{align*}
\Lambda_v&\geq \E\bigg[ \frac{1}{T_k}\int_0^{T_k} \Big( r^v(X^{*,v, w}_t) - \frac{1}{2}\| w_t\|^2\Big)dt\bigg]\\
&\geq  \limsup_{k\to\infty}  \E\bigg[ \frac{1}{T_k}\int_0^{T_k} \Big( r^v(X^{*,v, w}_t)\wedge L - \frac{1}{2}\| w_t\|^2\Big)dt\bigg]\,.
\end{align*}
From the hypothesis on $\mu_{v,w}^{*,T,3}$ and the definition of $\mu_{v,w}^{*,3}$, we conclude that \eqref{eqn-lem-trivial-bound} holds. 
\end{proof}
To summarize, we have shown that the ERSC problem for the limiting diffusion is equivalent to a TP-ZS SDG with the long-run average expected cost criterion for the extended process  $X^{*,v,w}$  associated with the auxiliary control $w$.

\subsection{Variational formulation for Poisson-driven controlled queueing dynamics} \label{sec-sub-var-PP}
This section provides a variational formulation for the ERSC problem of the diffusion-scaled queueing processes 
  in the context of Markovian stochastic networks. To the best of knowledge of the authors, this formulation is novel. 
  The following variational representation of exponential functionals of Poisson process (see \cite[Theorem 3.23]{budhiraja2019analysis}) is crucial in what follows. We state this result for a $1$-dimensional Poisson process $\widetilde N$ with rate $\lambda>0$. The function $\varkappa$ defined as 
  \begin{equation} \label{eqn-varkappa}
  \varkappa(r)\doteq r\ln r -r+1
  \end{equation}
   plays an important role in the analysis. We give some important lemmas regarding the function $\varkappa(\cdot)$ in Appendix~\ref{sec-varkappa}. Let $\cF_t$ be the filtration generated by the $\widetilde N$ such that $\cF_0$ contains all null sets.
\begin{theorem}\label{thm-var-rep-poisson-gen}For $T>0$,
	suppose that $G: \frD_T\rightarrow \RR$ is a bounded Borel measurable function. Then the following holds:
	\begin{align}
\frac{1}{T}\log \E\Big[ e^{TG(\widetilde N)}\Big]= \sup_{\phi\in \widetilde \cE} \E\bigg[ G(\widetilde N^\phi)-\frac{\lambda}{T}\int_0^T \varkappa (\phi_s)ds\bigg].
	\end{align}
		Moreover, for every $\delta>0$, 
			\begin{align}\label{eq-var-rep-poisson-gen-bound}\frac{1}{T}\log \E\Big[ e^{TG(\widetilde N)}\Big]\leq\sup_{\phi \in \widetilde \cE_M} \E\bigg[  G(\widetilde N^\phi)-\frac{\lambda}{T}\int_0^T \varkappa (\phi_s)ds\bigg]+\delta.\end{align}
	Here, $\widetilde \cE$  is the set of all the $\phi$ which are progressively measurable (with respect to $\cF_t$) such that for every $T>0$, 
	$$\frac{\lambda}{T}\int_0^T \varkappa(\phi_{s})ds<\infty,$$
	$ \widetilde \cE_M$ is the set of all $\phi\in \widetilde \cE$ such that for $T>0$,$$\frac{\lambda}{T}\int_0^T \varkappa(\phi_{s})ds \leq M, \text{  with $M$ depending only on $\delta$ and $\|G\|_\infty$} $$
	and  $\widetilde N^\phi$ denotes  a ``controlled" Poisson process which is a solution to the martingale problem below: for $f\in \cC^2(\RR)$,
	\begin{align}\label{eq-tilde-N-defn} f(\widetilde N_t^\phi)-f(0)-  \lambda\int_0^t\phi_s\big[f(\widetilde N^\phi_s+1)-f(\widetilde N^\phi_s)\big]ds \end{align}
	is a martingale with respect to $\cF_t$.
\end{theorem}
\begin{remark} Comments similar to those in Remark~\ref{rem-var-rep-proof-BM} are applicable. In this case, instead of Girsanov's theorem for Brownian motion, Girsanov's theorem for Poisson process is used. See \cite[Theorem 2.1]{budhiraja2011} for the proof. 
\end{remark}
\begin{remark}
	The above theorem can be easily extended to the multi-variate independent Poisson processes. We do not state the multi-variate version in the general setting, however we state the version that is relevant to us in the context of the ERSC problem for the queueing network model (see Proposition~\ref{prop-var-rep-poisson}).
\end{remark} 
Our interest in using the Theorem~\ref{thm-var-rep-poisson-gen} lies in proving Theorem~\ref{thm-AO}.  To do this, we first understand and study how to apply Theorem~\ref{thm-var-rep-poisson-gen} in a simple case of  a diffusion-scaled Poisson process where $G(\widetilde N)$ is replaced by $G(\widetilde N^n)$ with 
$$\widetilde N^n_t\doteq \frac{\widetilde N_{nt}-\lambda nt}{\sqrt{n}}.$$ 
We revisit the weak convergence for Poisson process $\widetilde N_t$ with rate $\lambda$ 
 using the associated variational representation which is given in Theorem~\ref{thm-var-rep-poisson-gen}. 
Such a study is also helpful for the reader to understand the intuition behind the proof of Theorem~\ref{thm-AO}. 
By Theorem~\ref{thm-var-rep-poisson-gen}, we obtain 
\begin{align}\label{eqn-var-rep-scaled-poisson}
\frac{1}{T}\log \E\big[ e^{TG(\widetilde N^n)}\big]= \sup_{\phi \in \widetilde \cE} \E\Bigg[ G\Big(\frac{\widetilde N^{n,\phi}-\lambda nt}{\sqrt{n}}\Big)- \frac{\lambda n}{T}\int_0^T\varkappa(\phi_s)ds\Bigg],
\end{align}
where $\widetilde N^{n,\phi}$ is the unique solution to the martingale problem below: for $f\in \cC^2(\RR)$, 
\begin{align}\label{eq-scaled-cont-poisson}
f\Bigg(\frac{\widetilde N^{n,\phi}_t-nt}{\sqrt{n}}\Bigg)&-f(0)- \int_0^t n\lambda\phi_s\Bigg[f\Bigg(\frac{\widetilde N^{n,\phi}_s+1-ns}{\sqrt{n}}\Bigg)-f\Bigg(\frac{\widetilde N^{n,\phi}_s-ns}{\sqrt{n}}\Bigg)\Bigg]ds \nonumber \\
&- \int_0^t \sqrt{n}\lambda\grad f \Bigg(\frac{\widetilde N^{n,\phi}_s-ns}{\sqrt{n}}\Bigg)ds
\end{align}
is an $\cF_t$--martingale.  
We prove the weak convergence of $\widetilde N^n_t$  using the above variational formulation.

\begin{theorem}\label{thm-fclt-poisson} For $T>0$, let $G:\frD_T\rightarrow \RR$  be a bounded continuous function. Then 
	\begin{align}\label{eq-fclt} 
	\lim_{n\to\infty}\frac{1}{T} \log\E[e^{TG(\widetilde N^n)}]= \sup_{w\in \cA}\E\Big[ G\Big(W+\int_0^\cdot w_t dt\Big) -\frac{\lambda}{2T}\int_0^T|w_t|^2 dt\Big]= \frac{1}{T}\log \E[e^{TG(W)}].
	\end{align}
	Here, $W$ is a  one-dimensional Brownian motion and $\cA$ is as defined in \eqref{def-A}, but for $\RR$ instead of $\RR^d$.
\end{theorem}

	In the following, we discuss several key elements of the proof of the theorem while deferring the proofs to the appendix. 	
	To begin with,
	we write
		\begin{align}
\frac{1}{T}	\log \E\big[ e^{TG(\widetilde N^n)}\big]&= \sup_{\phi \in \widetilde \cE} \E\Bigg[ G\bigg(\frac{\widetilde N^{n,\phi}- \lambda n \int_0^\cdot \phi_tdt-\lambda n\fre(\cdot) + \lambda n\int_0^\cdot \phi_tdt}{\sqrt{n}}\bigg)- \frac{\lambda n}{T}\int_0^T\varkappa(\phi_t)dt\Bigg] \nonumber\\&= \sup_{\phi \in \widetilde \cE} \E\Bigg[ G\bigg(\frac{\widetilde N^{n,\phi}- \lambda n \int_0^\cdot \phi_tdt}{\sqrt n}-\lambda {\int_0^\cdot \sqrt n(1-\phi_t)dt}\bigg)- \frac{\lambda n}{T}\int_0^T\varkappa(\phi_t)dt\Bigg]\,, \label{eq-var-rep-poisson-fclt-scaled}
	\end{align}
	where $\fre(t)= t$ and $\widetilde \cE$ is as defined in the statement of Theorem~\ref{thm-var-rep-poisson-gen}.  
	 Clearly, it suffices to  show that 
	\begin {align}\nonumber
	\lim_{n\to\infty}\sup_{\phi \in \widetilde \cE} &\E\Bigg[ G\bigg(\frac{\widetilde N^{n,\phi}- \lambda n \int_0^\cdot \phi_tdt}{\sqrt n}-\lambda {\int_0^\cdot \sqrt n(1-\phi_t)dt}\bigg)- \frac{\lambda n}{T}\int_0^T\varkappa(\phi_t)dt\Bigg] \\\label{eq-opt-problems-equal}
	&= \sup_{w\in \cA}\E\Bigg[ G\Big(W+\int_0^\cdot w_s ds\Big) -\frac{\lambda}{2T}\int_0^T|w_t|^2 dt\Bigg].
	\end{align}
	Here, $\cA$ is as defined in~\eqref{def-A} for some one-dimensional Brownian motion $W$. We observe that
	$$ M^{n,\phi}\doteq \frac{\widetilde N^{n,\phi}- \lambda n \int_0^\cdot \phi_tdt}{\sqrt{n}}$$
	is a square integrable $\cF_t$--martingale.

	Comparing~\eqref{eq-var-rep-poisson-fclt-scaled} with the right hand side of~\eqref{eq-fclt}, it is evident that analyzing the behavior of $\{\int_0^\cdot \sqrt n (1-\phi^n_t) dt\}_{n\in\NN}$ is necessary. It will be shown using Lemma~\ref{lem-compact} that $\{\sqrt n (1-\phi^n)\}_{n\in \NN}$ converges to $w \in L^2([0,T],\RR)$ in an appropriate sense, along a subsequence. This in turn, ascertains the convergence of $\{\int_0^\cdot \sqrt n(1-\phi^n_t)dt\}_{n\in \NN}$ to $\int_0^\cdot w_tdt$, again in an appropriate sense, along the same subsequence. To see if the corresponding family $\{M^{n,\phi^n}\}_{n\in \NN}$ is tight in $\frD_T$, we use the tightness of $\{\sqrt{n}(1-\phi^n)\}_{n\in \NN}$ in $L^2([0,T],\RR)$ (in an appropriate sense) with $\phi^n$ being $\delta$--optimal control for every $n$. This in turn, implies that  $\phi^n$ converges to $1$ in $L^2([0,T],\RR)$. 
	We can then show that $\int_0^\cdot \phi^n_sds$ converges to $\fre(\cdot)$ in $\frC_T$. Using random time change lemma (\cite[Pg. 151]{billingsley1999}) and martingale central limit theorem, we can conclude that $M^{n,\phi^n}$ converges to a Brownian motion.
	
	To prove the upper bound, we choose  a family of $\delta$--optimal controls (denoted by $\phi^n$), while to prove the lower bound, we let
	 $\phi^n= 1-\frac{w^*}{\sqrt{n}}$, which is a priori suboptimal control corresponding to the left hand side of~\eqref{eq-opt-problems-equal}.

Inspired by this study on the diffusion-scaled Poisson process, 
in the proof of lower bound (Lemma~\ref{lem-AO-lower}), we choose $\uppsi^n=(\phi^n,\psi^n,\varphi^n)$ according to~\eqref{eq-choice-control} and in the proof of upper bound (Lemma~\ref{lem-AO-upper}), we choose $\uppsi^n= (\phi^n,\psi^n,\varphi^n) $ that is nearly optimal corresponding to~\eqref{eq-up-sup-1} and then analyze the sequences $\{\sqrt n (e-\phi^n)\}_{n\in \NN}$, $\{\sqrt{n}(e-\psi^n)\}_{n\in \NN}$ and $\{\sqrt{n}(e-\varphi^n)\}_{n\in \NN}$.  
\begin{remark} To avoid/clarify any confusion, we re-iterate that we have used $\widetilde \cE$ ($\cA$, respectively) to denote the set of controls in the Poisson case (Brownian case, respectively). We will use $\cE^n$ to denote the set of controls in the case of diffusion-scaled queueing processes.
\end{remark}

We now set up the notation to state the multi-variate version of Theorem~\ref{thm-var-rep-poisson-gen} in the context of the ERSC problem for the queueing network.

Define $$N^n=\Big(\{\widecheck A^n_i\}_{i=1}^d,\{\widecheck S^n_i\}_{i=1}^d,\{\widecheck R^n_i\}_{i=1}^d\Big)$$ as the $3d$--dimensional vector of independent Poisson processes with rates  $$\Big(\{\lambda^n_i\}_{i=1}^d,\{n\mu^n_i\}_{i=1}^d,\{n\gamma^n_i\}_{i=1}^d\Big).$$

The filtration of the process $N^n$ is denoted by $\bar \cG^n_t$, for $t\geq 0$ (such that $\bar \cG^n_0$ includes all $\PP$--null sets). In the following, using the processes $\{\widecheck A^n_i\}_{i=1}^d$, $\{\widecheck S^n_i\}_{i=1}^d$ and $\{\widecheck R^n_i\}_{i=1}^d$, we re-define the processes $\hat X^n$, $\hat Q^n$ and $\hat Z^n$. Since the re-defined processes have the same laws, we reserve the original notation to denote them. In the rest of the paper, we always consider the re-defined versions of these processes.
 In terms of $\{\widecheck A^n_i\}_{i=1}^d$, $\{\widecheck S^n_i\}_{i=1}^d$ and $\{\widecheck R^n_i\}_{i=1}^d$, $\hat X^n_t$ is the diffusion-scaled queueing process, with $  U^n_t$ being the corresponding control,  given by 
\begin{align*}
\hat{X}^n_{i,t} &= \hat{X}^n_{i,0} + \ell^n_i t - \mu^n_i \int_0^t \hat{Z}^n_{i,s} ds - \gamma^n_i \int_0^t \hat{Q}^n_{i,s} ds,\\
&\qquad\qquad + \frac{1}{\sqrt{n}}\Big( \widecheck A^n_i( t) - \lambda^n_i t \Big)- \frac{1}{\sqrt{n}}\bigg(\widecheck S^n_i\left( \int_0^t \frac{Z^n_{i,s}}{n} ds\right)  - \mu^n_i \int_0^t Z^n_{i,s} ds  \bigg) \\
&\qquad\qquad - \frac{1}{\sqrt{n}}\bigg(\widecheck R^n_i\left( \int_0^t \frac{Q^n_{i,s}}{n} ds\right)  - \gamma^n_i \int_0^t Q^n_{i,s}ds  \bigg) \,,
\end{align*} 
where, as given in \eqref{eqn-QZ-hat},
$$ \hat{Q}^n_t = (e\cdot \hat{X}^n_t)^+ U^n_t, \quad \hat{Z}^n_t = \hat{X}^n_t -  (e\cdot \hat{X}^n_t)^+ U^n_t\,.$$
Here, we assume that the control $U^n$ is admissible which implies that  $U^n_t= U^n(t,\widecheck A^n_{[0,t]}, \widecheck S^n_{[0,t]}, \widecheck R^n_{[0,t]})$.
We enforce this assumption in the rest of the section. 

For a triplet $$\uppsi\doteq \Big(\{\phi_i\}_{i=1}^d,\{\psi_i\}_{i=1}^d,\{\varphi_i\}_{i=1}^d\Big)$$ such that $\phi_i,\psi_i,\varphi_i$ are $\RR_+$--valued  functions, let $\hat X^{n,\uppsi}$ be the solution to the following equation:
\begin{align}\nonumber
\hat{X}^{n,\uppsi}_{i,t} &= \hat{X}^{n,\uppsi}_{i,0} + \ell^n_i t - \mu^n_i \int_0^t\hat{Z}^{n,\uppsi}_{i,s} ds - \gamma^n_i \int_0^t\hat{Q}^{n,\uppsi}_{i,s} ds\\\label{eq-prelimit-controlled}
&\qquad + \frac{1}{\sqrt{n}}\Big( \widecheck A^n_i\Big(\int_0^t\phi_{i,s}ds\Big) - \lambda^n_i t \Big) - \frac{1}{\sqrt{n}}\bigg(\widecheck S^n_i\left( \int_0^t\psi_{i,s}\frac{Z^{n,\uppsi}_{i,s}}{n} ds\right)  - \mu^n_i \int_0^t  {Z^{n,\uppsi}_{i,s}} ds  \bigg) \nonumber\\
&\qquad - \frac{1}{\sqrt{n}}\bigg(\widecheck R^n_i\left(\int_0^t\varphi_{i,s} \frac{Q^{n,\uppsi}_{i,s}}{n} ds\right)  - \gamma^n_i \int_0^t {Q^{n,\uppsi}_{i,s}}ds  \bigg)\,,
\end{align} 
where \[\hat X^{n,\uppsi}_0=\hat X^n_0\,,\quad \hat X^{n,\uppsi}_t= \hat Q^{n,\uppsi}_t+\hat Z^{n,\uppsi}_t\,, \quad \hat Z^{n,\uppsi}_t= \hat X^{n,\uppsi}_t-(e\cdot \hat X^{n,\uppsi}_t)^+U^{n,\uppsi}_t \,,\]
with $U^{n,\uppsi}= U^n(t,\widecheck A^{n,\phi}_{[0,t]}, \widecheck S^{n,\psi}_{[0,t]}, \widecheck R^{n,\varphi}_{[0,t]})$.  Here, $\widecheck A^{n,\phi}$, $\widecheck S^{n,\psi}$, $\widecheck R^{n,\varphi}$ are defined according to~\eqref{eq-scaled-cont-poisson}. 
\begin{remark}
	Observe that we choose $\{\widecheck S^n_i\}_{i=1}^d$ and $\{\widecheck R^n_i\}_{i=1}^d$ with rates $\{n\mu^n_i\}_{i=1}^d$ and $\{n\gamma^n_i\}_{i=1}^d$, respectively, instead of $\{\mu^n_i\}_{i=1}^d$ and $\{\gamma^n_i\}_{i=1}^d$.
	 This will be convenient in proving the appropriate stability in Lemma~\ref{lem-tightness-empirical}. 
\end{remark}

Define 
\begin{align*}
J^n_L(\hat X^n_0,U^n)\doteq \limsup_{T\to \infty}\frac{1}{T}\log \E\Big[e^{T\big(\frac{1}{T}\int_0^T r(\hat X^n_t, U^n_t)dt\big)\wedge LT}\Big],
\end{align*}  
\begin{equation} \label{eqn-frKn} 
\frK^n(\uppsi,T)\doteq  \frac{1}{T} \int_0^T \mathfrak{k}^n(\uppsi_s)ds\,, 
\end{equation} 
where, 
\begin{equation*} 
\mathfrak{k}^n(\uppsi_s)\doteq \sum_{i=1}^d \big(\lambda^n_i\varkappa(\phi_{i,s})+n\mu^n_i\varkappa(\psi_{i,s})+n\gamma^n_i\varkappa(\varphi_{i,s})\big) \,. 
\end{equation*} 

 Let $\cE^n$ be the set of all $\uppsi$ functions that are $\bar \cG^n_t$--progressively measurable 
 such that for every $T>0$, $\frK^n(\uppsi,T)<\infty$.  Also for every $M>0$, define $\cE^n_M$ as the set of all $\uppsi\in \cE^n$ such that $\frK^n(\uppsi,T)\leq M$, for every $T>0$. We are now in a position to state the crucial variational representation results. We again remark that these results are stated in the form that is convenient for us to work with.
 \begin{corollary}\label{cor-var-rep-unbounded-poisson}
 We have 
 $$\limsup_{T\to\infty} \frac{1}{T}\log \E[e^{\int_0^T r(\hat X^n_t,U^n_t)dt}]= \limsup_{T\to\infty}\sup_{\uppsi \in \cE^n} \E\Big[ \frac{1}{T} \int_0^Tr(\hat X^{n,\uppsi}_t, U^{n,\uppsi}_t)dt -\frK^n(\uppsi,T)\Big].$$
 \end{corollary}
\begin{proposition}\label{prop-var-rep-poisson}
 	 	\begin{align}\label{eqn-var-rep-poisson}
J^n_L(\hat X^n_0,U^n) &= \limsup_{T\to\infty}\sup_{\uppsi \in \cE^n} \E\Big[ \Big(\frac{1}{T} \int_0^Tr(\hat X^{n,\uppsi}_t, U^{n,\uppsi}_t)dt\Big)\wedge L -\frK^n(\uppsi,T)\Big].
 	\end{align}
 	Moreover, for every $\delta>0$, 
 	\begin{align}\label{eq-var-rep-poisson-bound}J^n_L(\hat X^n_0,U^n)\leq \limsup_{T\to\infty}\sup_{\uppsi \in \cE^n_M} \E\Big[ \Big(\frac{1}{T} \int_0^Tr(\hat X^{n,\uppsi}_t, U^{n,\uppsi}_t)dt\Big)\wedge L -\frK^n(\uppsi,T)\Big]+\delta.\end{align}
 	Here, $M$ only depends on $\delta$ and $L$.
 \end{proposition}
 
\begin{remark}
	The content of the second statement of the above proposition is that for every $\delta>0$, there are $2\delta$--optimal controls $\uppsi\in \cE^n_M$. This in particular asserts that the hypothesis of Lemma~\ref{lem-compact} is satisfied which will in turn help us to use Lemma~\ref{lem-compact} in proving Lemma~\ref{lem-tightness-empirical}. 
\end{remark}

We next define the mean empirical (occupation) measures for the extended diffusion-scaled processes 
$\hat X^{n,\uppsi}_t$ in \eqref{eq-prelimit-controlled} as follows. For any $U^n\in \bU^n$ and $\uppsi \in \cE^n$,   let 
\begin{align}\nonumber
\mu_{U,\uppsi}^{n,T} (A\times B\times C)&\doteq \frac{1}{T}\int_0^T \Ind_{\{(\hat X^{n,\uppsi}_t,U^{n,\uppsi}_t,\uppsi_t)\in A\times B\times C\}} dt,\\
\label{def-adm-emp-m}\mu^{n,T,1}_{U,\uppsi} (A\times B)&\doteq \frac{1}{T} \int_0^T \Ind_{\{(\hat X^{n,\uppsi}_t,U^{n,\uppsi}_t)\in A\times B\}} dt,\\\nonumber
 \mu^{n,T,2}_{U,\uppsi}(A\times C)&\doteq \frac{1}{T}\int_0^T \Ind_{\{(\hat X^{n,\uppsi}_t,\uppsi_t)\in A\times C\}} dt,\\\label{def-markov-emp-m}\mu^{n,T,3}_{U,\uppsi}(A)&\doteq \frac{1}{T}\int_0^T \Ind_{\{\hat X^{n,\uppsi}_t\in A\}} dt,\
\end{align}
for any Borel set $A\subset \RR^d$, $B\subset \bU$ and $C\subset \calY$ with $\calY\doteq \RR^d_+\times \RR^d_+\times \RR^d_+$. We will use $\mu^{n,T,3}_{U,\uppsi}$ with only $\uppsi$ being Markov in the lower bound proof while with only $U^n$ being Markov in the upper bound proof below. Thus we do not define a corresponding $\widetilde{\mu}^{n,T}_{U,\uppsi}$ when both $U^n$ and $\uppsi$ being Markov as  $\widetilde \mu_{v,w}^{*,T}$ defined in~\eqref{def-mu*}. For $n\in \NN$, we represent their weak limits in $T$ (if they exist and the associated subsequence $T_k$ is irrelevant) by $\mu^{n}_{U,\uppsi}$, $\mu^{n,1}_{U,\uppsi}$, $\mu^{n,2}_{U,\uppsi}$, $ \mu^{n,3}_{U,\uppsi}$, respectively.

 It is evident from Corollary~\ref{cor-var-rep-unbounded-poisson} that the term $\frac{1}{T} \int_0^Tr(\hat X^{n,\uppsi}_t, U^{n,\uppsi}_t)dt -\frK^n(\uppsi,T)$  can also be written as the integral over the mean empirical measure of $(\hat X^{n,\uppsi},U^{n,\uppsi},\uppsi)$ as follows:
\begin{align*}
\frac{1}{T} \int_0^T\Big(r(\hat X^{n,\uppsi}_t, U^{n,\uppsi}_t)-\mathfrak{k}^n(\uppsi_t)\Big)dt= \int_{\RR^d\times \bU\times \calY} \big(r(x,u)-\mathfrak{k}^n(y)\big)d\mu_{U,\uppsi}^{n,T}(x,u,y)\,. 
\end{align*}
 In the proof of the lower bound (Lemma~\ref{lem-AO-lower}), we work with a general admissible control $U^n_t$ (whose corresponding extended process is $U^{n,\uppsi})$ and a specific control $\uppsi_t=f_n(\hat X^{n,\uppsi})$, for some appropriate function $f_n$. Because of this, we make use of the above mean empirical measures of $(\hat X^{n,\uppsi}, U^{n,\uppsi})$. However, in the proof of the upper bound (Lemma~\ref{lem-AO-upper}), we work with a Markov control $U^n_t=g_n(\hat X^{n}_t)$ for some appropriate function $g_n$ and a general control $\uppsi \in \cE^n_M$ (for some $M>0$).
  In this case,  the pair $(\hat X^{n,\uppsi}, U^{n,\uppsi})$ will be  of the form  $(\hat X^{n,\uppsi}, g_n(\hat X^{n,\uppsi}))$. Using this property, we can then write 
$$ \frac{1}{T} \int_0^Tr(\hat X^{n,\uppsi}_t, g_n(\hat X^{n,\uppsi}_t))dt= \int_{\RR^d} r(x,g_n(x))d\mu_{U,\uppsi}^{n,T,3}(x).$$

\subsection{A suitable topology and corresponding lemmas} 
\label{sec-top} 
In this section, we  describe a suitable topology and  prove  important results in connection with that topology (see Lemmas~\ref{lem-comp} and~\ref{lem-compact}). To that end, we fix a $\delta>0$ and choose $\delta$--optimal $\uppsi^n(T)$ for every $T$ (we will suppress the $\delta$-dependence throughout). From the choice of $\uppsi^n(T)$ and Proposition~\ref{prop-var-rep-poisson}, we can assume that  $\uppsi^n(T)\in \cE^n_{M}$ and we have
$$J^n_L(\hat X^n_0,U^n)\leq \limsup_{T\to\infty} \E\Big[ \Big(\frac{1}{T}\int_0^T r\big(\hat X^{n,\uppsi^n(T)}_t, U^{n,\uppsi^n(T)}_t\big)dt\Big)\wedge L -\frK^n(\uppsi^n(T),T)\Big]+\delta.$$
Since $\uppsi^n(T)\in \cE^n_{M}$,   we have
$$\frK^n(\uppsi^n(T), T)\leq M,\text{ for every $T>0$} .$$

We will use this bound extensively to prove some existence and convergence results. Therefore, it is helpful for us to work with a topology under which the above bound gives compactness. This topology is defined below. First, notice that since $\varkappa (\cdot)>0$, $\sqrt{\varkappa(\phi_{\cdot})}\in L^2_{\text{loc}}(\RR_+,\RR)$, whenever $$\frac{1}{T}\int_0^T\varkappa(\phi_s)ds<\infty, \text{ for every $T>0$}.$$  From this observation, we can see that $\cE^n$ can be regarded as the subset of $\cZ\doteq L^2_{\text{loc}}(\RR_+,\RR^d\times \RR^d\times \RR^d)$. Now we equip $\cZ$ (and thereby $\cE^n$) with a topology that is inherited from the weak$^*$ topology of $\cZ_T\doteq L^2([0,T],\RR^d\times \RR^d\times \RR^d)$ (which we denote by $L^{2,*}_{T}$) which will be useful for us.  $\cZ$ is equipped with the coarsest topology for which the mapping $\cZ\ni u\mapsto u_{\vert[0,T]}\in \cZ_T$ is continuous in $L^{2,*}_{T}$ for every $T>0$,  \emph{i.e.,} we say $u^n\to u$ for $u^n,u\in  \cZ$ under this  topology if and only if $u^n_{\vert[0,T]}$ converges to $u_{\vert[0,T]}$ in $L^{2,*}_{T}$, for every $T>0$. We denote the space $\cZ$ by $L^{2,*}_{\infty}$ when equipped with the aforementioned topology. 

\begin{lemma}\label{lem-comp}For every $M>0$, the set of all $w\in \cZ$ such that 
$$ \sup_{T>0}\frac{1}{T}\int_0^T \|w_t\|^2dt\leq M $$  
	is compact in $L^{2,*}_{\infty}$.
\end{lemma}
\begin{proof}
For a fixed $M>0$, let the set in the hypothesis be denoted by $\mathscr{K}$.
	Now assume the contrary, that is, that there is a sequence $\{w^n\}_{n\in\NN}\subset\mathscr{K}$ which is not convergent. In other words, there is a $T_0>0$ for which $\{w^n_{\vert[0,T_0]}\}_{n\in \NN}\subset L^{2,*}_{T_0}$ is not convergent in $L^{2,*}_{T_0}$. Since $w^n\in \mathscr{K}$, we have $\sup_{T>0}\frac{1}{T}\int_0^T\|w^n_t\|^2 dt \leq  M$. In particular,
	$$ \int_0^{T_0}\|w^n_t\|^2 dt= \int_0^{T_0}\|w^n_{\vert[0,T_0],t}\|^2dt < MT_0$$ 
	which  implies that $\{w^n_{\vert[0,T_0]}\}_{n\in \NN}$ is convergent in $L^{2,*}_{T_0}$. This is a contradiction and also proves the result.
\end{proof}

Even though we have defined the topology above for $\cZ$, the crucial compactness result below is only proved for a simpler case of $L^2_{\text{loc}}(\RR^+,\RR)$ to keep the expressions in the proof concise. To this end, we equip $L^2_{\text{loc}}(\RR^+,\RR)$ with   the coarsest topology for which the mapping $L^2_{\text{loc}}(\RR^+,\RR)\ni u\mapsto u_{\vert[0,T]}\in L^2([0,T],\RR)$ is continuous in weak topology of $L^2([0,T],\RR)$ for every $T>0$. From now on, the space $L^2_{\text{loc}}(\RR^+,\RR)$ is always assumed to be equipped with this topology.

\begin{lemma}\label{lem-compact}
	Suppose $\phi^n>0$ for every $n$, and 
	\begin{align}\label{eq-bound}
	\sup_{n\in \NN}\sup_{T>0}\frac{n}{T}\int_0^T \varkappa(\phi^n_t)dt\leq M, \text{ for some $M>0$}.
	\end{align}
	Then, 
	$$ \limsup_{n\to\infty}\sup_{T>0} \frac{1}{T}\int_0^T |\sqrt{n}(1-\phi^n_t)|^2dt\leq 2M.$$
In particular,	$ \{\sqrt n (1-\phi^n)\}_{n\in \NN}$ is compact in $L^2_{\text{loc}}(\RR^+,\RR)$.  
	
\end{lemma}
The proof of this result is given in the Appendix. The following is an easy corollary (a multi-dimensional version) to the above lemma. For $\uppsi\in \cE^n$, define $h^n(\uppsi)=h^n(\phi,\psi,\varphi)\doteq \big( \sqrt{n} (e-\phi),\sqrt{n} (e-\psi),\sqrt{n} (e-\varphi)\big)$.  
Recall $e= (1,1,\ldots,1)\transp\in \RR^d$.
\begin{corollary}\label{cor-control-compact}
Suppose $\{\uppsi^n\}_{n\in \NN}\subset \cE^n_M$ and 
$$\limsup_{n\to\infty}\sup_{T>0} \sum_{i=1}^d\frac{1}{T}\int_0^T\Big( \lambda^n_i|\sqrt{n}(1-\phi^n_{i,s})|^2+{\mu^n_i} |\sqrt{n}(1-\psi^n_{i,s})|^2+{\gamma^n_i} |\sqrt{n}(1-\varphi^n_{i,s})|^2\Big)ds \leq M\,.
 $$
 Then, $\{h^n(\uppsi^n)\}_{ n\in \NN}$  is a tight family of $L^{2,*}_{\infty}$--valued random variables.
 \end{corollary}
An immediate consequence of the above corollary is that for a fixed $n$, $\{\uppsi^n(T)\}_{T>0}$ is compact in $L^{2,*}_{\infty}$ and therefore, has a weak limit point $\uppsi^n\in L^{2,*}_{\infty}$ along a subsequence (say $T_k$). Moreover, \begin{align}\label{eq-weak-lower-semicontinuity} \E\Big[\limsup_{T\to\infty}\frK^n(\uppsi^n(T),T)\Big]\geq\E\Big[ \limsup_{T\to\infty}\frK^n(\uppsi^n, T)\Big]\end{align}
which follows from the weak lower-semicontinuity of the norm.

\subsection{An auxiliary diffusion limit arising from the variational formulation} \label{sec-aux-diffusion}
Let $\frD^\bU_T$ be the set of $\bU$--valued c\`adl\`ag functions on $[0,T]$ equipped with the Skorohod topology.
\begin{theorem}\label{thm-diff}
	Suppose $\uppsi^n=(\phi^n,\psi^n,\varphi^n)\in \cE^n$ is such that 
 $$\sup_{n\in \NN}\sup_{T>0}\frK^n(\uppsi^n, T)<\infty.$$ 
 Then the family of processes $\{\big(\hat X^{n,\uppsi^n},U^{n,\uppsi^n},h^n(\uppsi^n)\big)\}_{n\in \NN}$ (defined by~\eqref{eq-prelimit-controlled} with $\uppsi=\uppsi^n$) is tight in $\frD^d_T\times \frD^\bU_T\times L^{2,*}_{\infty}$, for every $T>0$. Moreover, every limit point $(X,u,w)\doteq (X^{*,u,w},u,w)$ with $w=(w^1,w^2,w^3)$ satisfies 
 \begin{align}\label{eq-diff-limit-control} dX^{*,u,w}_t= b(X^{*,u,w}_t,u_t))dt + \Sigma \widetilde w _tdt + \Sigma dW_t, \end{align}
 for some $d$--dimensional Brownian motion $W$ and $\widetilde w$ defined as 
 \[\widetilde w_{i,t}\doteq  \frac{\lambda_iw^1_{i,t}+\mu_i \rho_i w^2_{i,t}}{\sqrt{2}}= \frac{\lambda_i}{\sqrt{2}} (w^1_{i,t}+w^2_{i,t}).
 \]
\end{theorem}
We only give the sketch of the proof below. To begin with, fix $T>0$ and re-write~\eqref{eq-prelimit-controlled} for $\uppsi=\uppsi^n$ as follows:
\begin{align}\nonumber
\hat{X}^{n,\uppsi^n}_{i,t} &= \hat{X}^{n,\uppsi^n}_{i,0} + \ell^n_i t - \mu^n_i \int_0^t\hat{Z}^{n,\uppsi^n}_{i,s} ds - \gamma^n_i \int_0^t\hat{Q}^{n,\uppsi^n}_{i,s} ds\\\nonumber
&\qquad + \frac{1}{\sqrt{n}}\Big( \widecheck A^n_i\big(\int_0^t\phi^n_{i,s}ds\big) - \lambda^n_i \int_0^t \phi^n_{i,s}ds\Big)\\\nonumber
&\qquad - \frac{1}{\sqrt{n}}\bigg(\widecheck S^n_i\left( \int_0^t\frac{\psi^n_{i,s}Z^{n,\uppsi^n}_{i,s}}{n} ds\right)  - n\mu^n_i \int_0^t  \frac{\psi^n_{i,s}Z^{n,\uppsi^n}_{i,s}}{n} ds  \bigg) \nonumber\\
&\qquad - \frac{1}{\sqrt{n}}\bigg(\widecheck R^n_i\left(\int_0^t \frac{\varphi^n_{i,s}Q^{n,\uppsi^n}_{i,s}}{n} ds\right)  - n\gamma^n_i \int_0^t \frac{\varphi_{i,s}^nQ^{n,\uppsi^n}_{i,s}}{n}ds  \bigg)\\\nonumber
&\qquad -\frac{\lambda^n_i}{\sqrt{n}}\int_0^t (1- \phi^n_{i,s})ds - \mu^n_i \sqrt{n} \int_0^t (1-\psi^n_{i,s}) \frac{Z^{n,\uppsi^n}_{i,s}}{n}ds \\\nonumber
&\qquad - \gamma^n_i \sqrt{n} \int_0^t (1-\varphi^n_{i,s}) \frac{Q^{n,\uppsi^n}_{i,s}}{n}ds. 
\end{align} 
Here, \[\hat X^{n,\uppsi^n}_0=\hat X^n_0,\quad \hat X^{n,\uppsi^n}_t= \hat Q^{n,\uppsi^n}_t+\hat Z^{n,\uppsi^n}_t, \quad \hat Z^{n,\uppsi^n}_t= \hat X^{n,\uppsi^n}_t-(e\cdot \hat X^{n,\uppsi^n}_t)^+U^{n,\uppsi^n}_t .\]
 We now observe that 
\begin{align*}
\hat M^{n, A,\uppsi^n}_{i,t}&\doteq \frac{1}{\sqrt{n}}\Big( \widecheck A^n_i\big(\int_0^t\phi^n_{i,s}ds\big) - \lambda^n_i \int_0^t \phi^n_{i,s}ds\Big),\\
\hat M^{n,S,\uppsi^n}_{i,t}&\doteq \frac{1}{\sqrt{n}}\bigg(\widecheck S^n_i\left( \int_0^t\frac{\psi^n_{i,s}Z^{n,\uppsi^n}_{i,s}}{n} ds\right)  - n\mu^n_i \int_0^t  \frac{\psi^n_{i,s}Z^{n,\uppsi^n}_{i,s}}{n} ds  \bigg),\\
\hat M^{n,R,\uppsi^n}_{i,t} &\doteq \frac{1}{\sqrt{n}}\bigg(\widecheck R^n_i\left(\int_0^t \frac{\varphi^n_{i,s}Q^{n,\uppsi^n}_{i,s}}{n} ds\right)  - n\gamma^n_i \int_0^t \frac{\varphi_{i,s}^nQ^{n,\uppsi^n}_{i,s}}{n}ds  \bigg)
\end{align*}
are square integrable $\bar\cG_t^n$--martingales. Let 
$$ \hat M^{n,\uppsi^n}\doteq \hat M^{n, A,\uppsi^n} +\hat M^{n, S,\uppsi^n}+\hat M^{n, R,\uppsi^n}$$
and
$$ \hat \Xi^n\doteq  \frac{\lambda^n_i}{\sqrt{n}}\int_0^t (1- \phi^n_{i,s})ds+  \mu^n_i \sqrt{n} \int_0^t (1-\psi^n_{i,s}) \frac{Z^{n,\uppsi^n}_{i,s}}{n}ds  +\gamma^n_i \sqrt{n} \int_0^t (1-\varphi^n_{i,s}) \frac{Q^{n,\uppsi^n}_{i,s}}{n}ds. $$
Below we state the tightness of a certain family of random variables (which are the key aspects of the proof) and briefly explain why such a tightness holds.
\begin{enumerate}
\item[(i)] $\{U^{n,\uppsi^n}\}_{n\in\NN}$: Since $U^{n,\uppsi^n}$ is a $\bU$--valued process which is compact, the family is trivially tight in $\frD^\bU_T$ with a limit point $u$.
\item[(ii)]\label{item-1} $\{h^n(\uppsi^n)\}_{n\in \NN}$: Using Corollary~\ref{cor-control-compact} and the hypothesis of the theorem, we can conclude that this family of random variables is tight in $L^{2,*}_{\infty}.$
\item[(iii)] $\{\uppsi^n\}_{n\in \NN}$: From the definition of $h^n$, it is clear that $\phi_i^n,\psi^n_i,\varphi^n_i\Rightarrow e$, for every $1\leq i\leq d$, as $n\to\infty$. This implies that  $\uppsi^n\Rightarrow (e,e,e)$, as $n\to\infty$. Recall that $e=(1,\ldots,1)\transp\in \RR^d$.

\item[(iv)] \label{item-3}$\big\{((n^{-1}{Z^{n,\uppsi^n}_i},n^{-1}{Q^{n,\uppsi^n}_{i}})\big\}_{n\in \NN}$: Following the arguments of the proof of \cite[Lemma 4(ii)]{AMR04}, we can conclude that 
$$ (n^{-1}Z^{n,\uppsi^n}_i,n^{-1}Q^{n,\uppsi^n}_{i})\Rightarrow (\rho_i,0) \,\,\, \text{ in $\frD_T$ \, as $n\to\infty$.}$$
\item[(v)] $\{\hat M^{n,\uppsi^n}\}_{n\in \NN}$: From the martingale central limit theorem, the random time change lemma, the fact that $\uppsi^n\Rightarrow (e,e,e)$ as $n\to\infty$ and the above display, we can conclude that 
$$ \hat M^{n,\uppsi^n}\Rightarrow \Sigma W \,\, \text{ in $\frD_T^d$ \, as $n\to\infty$. }$$Here, $W$ is a $d$--dimensional Brownian motion and $\Sigma=\sqrt 2 \diag (\sqrt \lambda_1,\sqrt\lambda_2,\ldots,\sqrt \lambda_d)$. 
\item[(vi)] $\{\hat \Xi^n\}_{n\in \NN} $: From the above tightness in (ii) and the following arguments used in showing tightness of $\{m^k\}_{k\in \NN}$ (in the proof of Theorem~\ref{thm-fclt-poisson}), we can conclude the tightness of $\{\int_0^\cdot h^n_tdt\}_{n\in \NN}$ in $\frC_T^d\times \frC_T^d\times \frC^d_T$. Combining this with tightness in (iv), gives us the following: If $h^n\Rightarrow w= (w^1,w^2,w^3)\in L^{2,*}_{\infty}$, then 
$$ \int_0^\cdot h^n_tdt \Rightarrow \Big(\int_0^\cdot w^1_tdt,\int_0^\cdot w^2_tdt,\int_0^\cdot w^3_tdt\Big)\,\, \text{ in $\frC_T^d\times \frC_T^d\times \frC^d_T$\,  as $n\to\infty$} $$
and 
$$ \hat \Xi^n_i\Rightarrow \lambda_i w^1+\mu_i\rho_iw^2 .$$
\end{enumerate}

Following the arguments of the proof of \cite[Lemma 4(iii)]{AMR04}, we can show that $\{\hat X^{n,\uppsi^n}\}_{n\in \NN}$ is tight in $\frD^d_T$. Therefore, along a subsequence (again denoted by $n$), 
$$ \big(\hat X^{n,\uppsi^n}, U^{n,\uppsi^n}, h^n(\uppsi^n)\big)\Rightarrow (X^{*,u,w},u,w) \text{ in $\frD^d_T\times L^{2,*}_{\infty}$,  as $n\to\infty$.}$$
Here, $X^{*,u,w}$ and $w$ are related according to~\eqref{eq-diff-limit-control}.

\begin{remark} 
If in the above theorem, $U^{n,\uppsi^n}= v(\hat X^{n,\uppsi^n})$, for some $v:\RR^d\rightarrow\bU$ that is continuous, then 
the limit point $\hat X^{n,\uppsi^n}$  is given as the solution to
$$ dX^{*,v,w}_t= b(X^{*,v,w}_t,v(X^{*,v,w}_t))dt + \Sigma \widetilde w_tdt + \Sigma dW_t.$$
This diffusion resembles very much a well studied process in relation  to the ERSC cost known as the {\it ground diffusion process} which is defined as the solution to the following equation
	\begin{align}\label{eq-x*}
dX^{*,v}_t= b(X^{*,v}_t,v(X^{*,v}_t))dt + \Sigma\Sigma\transp \grad \Phi^v (X^{*,v}_t)dt + \Sigma dW_t.
\end{align}
It turns out that $w^*= \Sigma\transp\grad \Phi^v(\cdot)$ is the optimal stationary Markov control for~\eqref{eqn-erg-cont-aug}. We do not go into further details and interested reader can refer \cite{AB18,ari2018strict} and the reference therein. The control in~\eqref{eq-diff-limit-control} corresponds to a sub-optimal control for~\eqref{eqn-erg-cont-aug}. 
\end{remark}

\section{Proof of the Lower bound for Theorem~\ref{thm-AO}}\label{sec-low-bound}

In this section we prove the following lower bound result. 

\begin{lemma} \label{lem-AO-lower}  The following holds:
\begin{align*}
\liminf_{n\to\infty} \hat{\Lambda}^n(\hat X^n_0) \ge  \Lambda. 
\end{align*}
\end{lemma} 

\begin{proof} 
	Let  $n_k$ be a sequence along which $\liminf_{k\to\infty} \hat \Lambda ^{n_k}(\hat X^{n_k}_0)$ is attained.  We simply denote this sequence as $n$. For a fixed $\delta>0$, choose an admissible control $U^n=U^n(t,\widecheck A^n_{[0,t]},\widecheck S^n_{[0,t]},\widecheck R^n_{[0,t]}) $ (with the associated SCP $\hat Z^n=\hat Z^n (t,\widecheck A^n_{[0,t]},\widecheck S^n_{[0,t]},\widecheck R^n_{[0,t]}) $ defined according to~\eqref{eqn-QZ-hat})
	 	such that
	\begin{align}\label{eq-lb-eq1}
	\hat \Lambda^n(\hat X^n_0)+\delta &\geq J^n(\hat X^n_0, U^n).
	\end{align}
		In the following, we will show that 
	\begin{align}\label{eq-interest} \liminf_{n\to\infty} J(\hat X^n_0, U^n)\geq \Lambda -\delta.\end{align}
We will use the notation from Section~\ref{sec-sub-var-PP}.	Applying Corollary~\ref{cor-var-rep-unbounded-poisson}, we have 
	\begin{align*}
	J^n(\hat X^n_0, U^n) &= \limsup_{T\to\infty}\sup_{{\uppsi \in \cE^n}}\E\Bigg[ \frac{1}{T} \int_0^T  r(\hat X^{n,\uppsi}_t,  U^{n,\uppsi}_t)dt  - \frK^n(\uppsi, T) \Bigg] \\
	&\geq \limsup_{T\to\infty}  \E\Bigg[ \frac{1}{T} \int_0^T  r(\hat X^{n,\uppsi}_t,  U^{n,\uppsi}_t)dt - \frK^n(\uppsi, T)\Bigg],
	\end{align*}
	where the inequality holds for any $\uppsi=(\phi,\psi,\varphi)\in \cE^n$ and $U^{n,\uppsi}_t\doteq  U^n(t, \widecheck A^{n,\phi}_{[0,t]}, \widecheck  S^{n,\psi}_{[0,t]}, \widecheck R^{n,\varphi}_{[0,t]})$ which is a $\bU$--valued process, and $\frK^n(\uppsi, T)$ is defined as in \eqref{eqn-frKn}.  Also, $\widecheck A^{n,\phi}$, $\widecheck S^{n,\psi}$, $\widecheck R^{n,\varphi}$ are defined according to~\eqref{eq-scaled-cont-poisson}.
	
	For every $n$, we now make a particular choice for $\uppsi^n= (\phi^n,\psi^n,\varphi^n)$.
	 Fix a $\delta>0$. Then using Theorem~\ref{thm-sup-inf} (in particular,~\eqref{eq-lim-sup-inf}) and Theorem~\ref{thm-2p-game}(v), there exists a $l>0$ large enough and $w^*\in \Wsm(l)$ such that 
	\begin{align}\label{eq-erg-occu} \Lambda\leq \int_{\RR^d\times \bU} \big(r^v(x)-\frac{1}{2}\|w^*(x)\|^2 \big) \widetilde \mu^{*}_{v,w^*}(dx) +\delta,\end{align}
	for every $v\in \Usm$  and $w^*(\cdot)$ is continuous.
	Now we define	
		\begin{align}\label{eq-choice-control}
	\widetilde \phi^n_{i}(x) \doteq 1- \frac{\big(w^* (x)\big)_i}{\sqrt{n}}, \text{ and } \widetilde \psi^n_{i}(x)\doteq 1- \frac{\big(w^* (x)\big)_i}{\sqrt{n}}, \text{ for $x\in \RR^d$}.
	\end{align}
	Finally, we set
	$$\phi^n_{i,t}=\widetilde \phi^n_i(\hat X^{n,\uppsi^n}_t),\;\;\psi^n_{i,t}= \widetilde \psi_i^n(\hat X^{n,\uppsi^n}_t) \text{ and } \varphi^n_{i,t}= 1, \text{ for $t\geq 0$.}$$
	This particular choice is motivated by the arguments used in Theorem~\ref{thm-fclt-poisson}. From the above choice, we can ensure that $\phi^n_i>0$ and $\psi^n_{i}>0$, for large enough $n$.  From now on, we simply write $\hat X^{n,\uppsi^n}$ as $\hat X^{n,\uppsi}$. With the above choice, we have 
	\begin{align*}
	J^n(\hat X^n_0,U^n)
	\geq \limsup_{T\to\infty}\frac{1}{T} \E\Bigg[ \int_0^T\Big( r(\hat X^{n,\uppsi}_t, U^{n,\uppsi}_t) - \sum_{i=1}^d \lambda^n_i \varkappa(\widetilde \phi^n_{i}(\hat X^{n,\uppsi}_t))-\sum_{i=1}^d n\mu^n_i \varkappa(\widetilde \psi^n_{i}(\hat X^{n,\uppsi}_t))\Big)dt\Bigg].
	\end{align*}
From the definitions of $\hat X^{n,\uppsi}$, $\widetilde \phi^n$, $\widetilde \psi^n$ and the fact that $w^*$ is compactly supported, the infinitesimal generators of $\hat X^n$ and $\hat X^{n,\uppsi}$ coincide outside the compact set $\overline B_l$, $\hat X^{n,\uppsi}$ will inherit the Lyapunov function corresponding to $\hat X^n$ (albeit with different coefficients and compact sets in the Foster-Lyapunov equation). This implies that the family of mean empirical (occupation) measures \\$\{\mu_{U,\uppsi}^{n,T,1}\}_{T>0,n\in\NN}$ defined in~\eqref{def-adm-emp-m} is tight in $\cP(\RR^d\times\bU)$. Therefore, along a subsequence $T_k$, $\{\mu_{U,\uppsi}^{n,T_k,1}\}_{k\in \NN}$ converges weakly to some $\mu_{U,\uppsi}^{n,1} \in \cP(\RR^d\times \bU)$ and we have
\begin{align*}
J^n(\hat X^n_0, U^n)
&\geq \limsup_{k\to\infty}\frac{1}{T_k} \E\Bigg[ \int_0^{T_k} r(\hat X^{n,\uppsi}_t,  U^{n,\uppsi}_t) dt- \sum_{i=1}^d \lambda^n_i\int_0^{T_k} \varkappa(\widetilde \phi^n_{i}(\hat X^{n,\uppsi}_t))dt\\
&\qquad\qquad\qquad\qquad-\sum_{i=1}^d n\mu^n_i\int_0^{T_k} \varkappa(\widetilde \psi^n_{i}(\hat X^{n,\uppsi}_t))dt\Bigg]\,\\
 &\geq  \int_{\RR^d\times\bU}\Big( r(x,  u) - \sum_{i=1}^d \lambda^n_i \varkappa(\widetilde \phi^n_{i}(x))-\sum_{i=1}^d n\mu^n_i \varkappa(\widetilde \psi^n_{i}(x))\Big)d\mu_{U,\uppsi}^{n,1}(x,u).
\end{align*}
From the choice of $\widetilde \phi^n$ and $\widetilde \psi^n$, it is clear that $\widetilde \phi^n(\cdot)$ and $\widetilde \psi^n(\cdot)$ converge to $e=(1,\dots,1)\transp$  uniformly on  $\RR^d$ and using the fact that $n^{-1}\lambda^n_i\to \lambda_i$ and $\mu^n_i\to\mu_i$, we obtain 
\begin{align*}
\Big|\lambda^n_i\varkappa (\widetilde \phi^n_i(x))-\frac{1}{2}\lambda_i|\big(w^*(x)\big)_i|^2 \Big|\leq C\frac{\sup_{y\in B_l} \|w^*(y)\|^3}{\sqrt{n}}\to  0
\end{align*}
and 
\begin{align*}
\Big| n\mu^n_i\varkappa (\widetilde \psi^n_i(x))-\frac{1}{2}\mu_i|\big(w^{*}(x)\big)_i|^2 \Big|\leq C\frac{\sup_{y\in B_l} \|w^*(y)\|^3}{\sqrt{n}}\to  0,
\end{align*}
as $n\to\infty$, uniformly in $x\in\RR^d$, for some $C>0$.

Following the computation similar to the one in \cite[Pg. 3559-3560]{ABP15}, we can conclude that $\mu_{U,\uppsi}^{n,1}$ converges weakly along a subsequence (still denoted by $n$) to an ergodic occupation measure $\widetilde\mu_{v,w^*}^*$, for some $v\in \Usm$. From here and using~\eqref{eq-erg-occu}, we have
\begin{align*}
\liminf_{n\to\infty}\hat \Lambda^n(\hat X^n_0)\geq \liminf_{n\to\infty} J^n(\hat X^n_0, U^n)-\delta&\geq  \int_{\RR^d\times \bU}\Big( r(x,u) - \frac{1}{2}\|w^{*}(x)\|^2\Big)d\widetilde\mu_{v,w^*}^*(x,u)\geq \Lambda-2\delta.
\end{align*} From the arbitrariness of $\delta$, we have the result. 
\end{proof}

\section{Proof of the upper bound for Theorem~\ref{thm-AO}} \label{sec-upp-bound}
In this section we prove the upper bound for asymptotic optimality. This is much more involved than that of the proof of Lemma~\ref{lem-AO-lower} {which involved the already existing techniques (in the context of CEC)}. 
 In contrast, as will be seen below, results similar to Lemma~\ref{lem-limit-truc} and Theorem~\ref{thm-sup-inf} (in the case where the driving noise is a Poisson process) can simplify the proof. Unfortunately, such results cannot be proved in a straightforward way.  The lack of results similar to Lemma~\ref{lem-limit-truc} and Theorem~\ref{thm-sup-inf} manifests into the following difficulties: For an SCP $v^n$ (obtained from an appropriate nearly optimal stationary Markov control with respect to $\Lambda$), even though the associated $\hat X^n$ is stable,  it is not at all clear if 1.) there are  nearly optimal auxiliary controls associated with $J^n(\hat X^n_0,v^n)$ that are Markov and that their mean empirical measures  are tight, and 2.) the extended diffusion-scaled process is stable. Therefore, we develop a novel approach to overcome these difficulties.  This approach involves using the `exponential' analog of uniform integrability (Lemma~\ref{lem-tail-est}) to conclude that $J^n(\hat X^n_0,v^n)$  can be approximated arbitrarily well by an ERSC cost associated with a certain truncated running cost (Corollary~\ref{cor-trunc-limit-L}). In Lemma~\ref{lem-fin}, we then represent that ERSC cost (associated with that truncated running cost; we refer to this for now as the truncated cost) using the variational formulation. Due to the truncated nature of the running cost, there are  associated nearly optimal auxiliary controls that are tight (from Proposition~\ref{prop-var-rep-poisson}). Since the truncated cost approximates $J^n(\hat X^n_0,v^n)$ arbitrarily well, the nearly optimal auxiliary controls (associated with the truncated cost) are also the nearly optimal auxiliary controls (associated with $J^n(\hat X^n_0,v^n)$). Moreover, we can choose these controls such that their mean empirical measures are tight (again, from Proposition~\ref{prop-var-rep-poisson}). In the final part of the approach, we use the tightness of these controls and the explicit form of a Lyapunov function (from \cite{AHP21}) and show that the extended diffusion-scaled process is stable.

\begin{lemma} \label{lem-AO-upper}
	The following holds:\begin{align*}
\limsup_{n\to\infty} \hat{\Lambda}^n(\hat{X}^n_0) \le  \Lambda. 
\end{align*}
\end{lemma}
\begin{proof}
	Let $v^*\in\Usm $ be an optimal control corresponding to $\Lambda$. This control, as we know from Theorem~\ref{thm-diffusion},  satisfies~\eqref{eqn-optimality1}. This creates an issue in invoking weak convergence of measures later on because of the following: $v^*=v^*(\cdot)$ is in general, merely a Borel measurable function from $\RR^d$ to $\bU$ and the family of integrals of a Borel measurable function with respect to a weakly convergent family of measures does not necessarily converge to the corresponding integral with respect to the limiting measure.
	 To overcome this, we construct a $\delta$--optimal control $v^\delta$ that is a continuous map from $\RR^d\to \bU$ and is $e_d\doteq (0,0,\ldots,1)\transp$ outside a sufficiently large ball (say, $K$ is the radius of the ball).

	The construction of the aforementioned $\delta$--optimal control is carried out in Lemma~\ref{lem-cont-control}. Therefore, we just invoke this lemma and ascertain the existence of such a control which we from now on denote by $v^\delta=v^\delta(\cdot)$.

 From the construction of $v^\delta$ in the proof of Lemma~\ref{lem-cont-control}, we know that $v^\delta:\RR^d\rightarrow \bU$ is a continuous function such that 
\begin{align*}
v^\delta(x)=\begin{cases} v^\delta, \text{ whenever $\|x\|\leq K-\frac{1}{l}$,}\\
 e_d, \text{ whenever $\|x\|>K$,}
\end{cases}
\end{align*}
for sufficiently large $l$ and $K$. We focus our attention on the case when $\|x\|\leq K$ and $\|x\|>K$ as there is a distinct change in behavior of $v^\delta$ at $\|x\|=K$. To that end, we will define the set $R_n$ below which captures this behavior. In the following, using $v^\delta$, we define a SCP $Z^n$ and the corresponding pre-limit process $X^n$, for large enough $n$. Subsequently, we will define diffusion-scaled versions $\hat Z^n$ and $\hat X^n$. We follow the construction in \cite[Pg. 3561]{ABP15} and \cite[Section 2.6]{AMR04}. We first note that $v^\delta_i$, for $i=1,\ldots,d$ plays the role of the fraction of class-$i$ customers in queue, when total queue size is positive. 
Since $\hat X^n$ is the  argument of $v^\delta $,  we write $\hat X^n_t= \hat x^n(X^n_t)$ (recall $\hat x^n$ in \eqref{eq-hatx}). 
It is important to note that we cannot simply define 
$$ Z^n= X^n- (e\cdot X^n-n)^+v^\delta(\hat X^n),$$
because $X^n-Z^n$ will not necessarily  lie in $\ZZ^d_+$ if defined as above. To overcome this, define  a measurable map $\vartheta: \{z\in \RR^d_+: (e\cdot z)\in \ZZ\}\rightarrow \ZZ^d_+$ as 
\begin{align}
\vartheta(z)\doteq \Big( \lfloor z_1\rfloor, \lfloor z_2\rfloor, \ldots, \lfloor z_d\rfloor + \sum_{i=1}^{d}(z_i-\lfloor z_i\rfloor)\Big)\,. 
\end{align}
Observe that $\|\vartheta(z)-z\|\leq 2d$ and $e\cdot \vartheta(z)=1$.  Define the set $$R_n\doteq \Big\{x\in \RR^d_+: \max_{1\leq i\leq d} |x_i-\rho_in|\leq K\sqrt{n}\Big\}.$$
For $1\leq i\leq d$,
\begin{align}\label{def-scp} Z^n_i=Z^n_i[X^n]=\begin{cases}
X^n_i- q^n_i(X^n), \text{ whenever  $X^n\in R_n$},\\
X^n_i\wedge \Big( n-\sum_{j=1}^{i-1} X^n_j\Big)^+, \text{ otherwise}
\end{cases}\end{align}
with  $q^n(x)\doteq \vartheta\big((e\cdot x-n)^+ v^\delta(\hat x^n)\big).$ 

From the discussion in \cite[Pg. 3561]{ABP15}, this is a well-defined work conserving SCP, for large enough $n$. From now on we  restrict ourselves to such large enough $n$. To proceed further, we need to understand the conditions  under which  we have 
$$ Q^n_t= X^n_t-Z^n_t= q^n(X^n_t).$$
It is easy to check that a sufficiency condition is $\sum_{i=1}^{d-1}X^n_{i,t}\leq n$,  which can also be re-written as $$ \sum_{i=1}^{d-1} \hat X^n_{i,t}\leq \rho_d \sqrt{n}.$$
Define $S_n\doteq \{x: \sum_{i=1}^{d-1}\hat x^n_i(x)\leq \rho_d\sqrt{n}\}$ with $\hat x^n(x)$ as defined in~\eqref{eq-hatx}. We will suppress $x$ and just write $\hat x^n$ to keep  the expression concise.

Now that we have defined the  processes $Z^n$ and $X^n$ completely (see \cite[Proposition 1]{AMR04}), we move on to define the diffusion-scaled versions  $\hat X^n$, $\hat Z^n$ and $\hat Q^n= \hat X^n-\hat Z^n$. To do that, we define 
$$ \hat q^n(\hat x^n)\doteq \vartheta\big(\sqrt n (e\cdot \hat x^n)v^{\delta}(\hat x^n) \big).$$
From the above discussion, $$\hat Q^n_t= \hat X^n_t-\hat Z^n_t= \frac{1}{\sqrt{n}}\hat q^n(\hat X^n_t) \quad \text{ whenever $X^n\in S_n$}.$$

Therefore, recalling that $\widetilde r(q)=\kappa\cdot q$, we have
\begin{align*} \int_0^T \widetilde r(\hat Q^n_t)dt&= \int_0^T \widetilde r\big(\frac{1}{\sqrt{n}}\hat q^n(\hat X^n_t)\big)\Ind_{\{ \hat X^n_t\in S_n\}} dt + \int_0^T \widetilde r\big(\hat X^n_t-\hat Z^n_t\big)\Ind_{\{ \hat X^n_t\notin S_n\}} dt \\
&=\int_0^T \widetilde r\big(\frac{1}{\sqrt{n}}\hat q^n(\hat X^n_t)\big)\Ind_{\{\hat X^n_t\in S_n\}} dt + \int_0^T \widetilde r\big(\hat X^n_t-\hat Z^n_t\big)\Ind_{\{\hat  X^n_t\notin S_n\}} dt\\
&\qquad\qquad - \int_0^T \widetilde r(\frac{1}{\sqrt{n}}\hat q^n(\hat X^n_t))\Ind_{\{ \hat X^n_t\notin S_n\}} dt+ \int_0^T\widetilde  r\big(\frac{1}{\sqrt{n}}\hat q^n(\hat X^n_t)\big)\Ind_{\{\hat X^n_t\notin S_n\}} dt \\
&= \int_0^T \widetilde r\big(\frac{1}{\sqrt{n}}\hat q^n(\hat X^n_t)\big) dt + \int_0^T \widetilde r\big(\hat X^n_t-\hat Z^n_t\big)\Ind_{\{ \hat X^n_t\notin S_n\}} dt\\
&\qquad\qquad - \int_0^T \widetilde r\big(\frac{1}{\sqrt{n}}\hat q^n(\hat X^n_t)\big)\Ind_{\{\hat X^n_t\notin S_n\}} dt\\
&= \int_0^T \widetilde r\big(\frac{1}{\sqrt{n}}\hat q^n(\hat X^n_t)\big) dt + \Delta^n(T).
\end{align*}
Here,
$$ \Delta^n(T)=\Delta^n_1(T)- \Delta^n_2(T) $$
with 
\begin{align} 
\Delta^n_1(T) &\doteq \int_0^T \widetilde r\big(\hat X^n_t-\hat Z^n_t\big)\Ind_{\{ \hat X^n_t\notin S_n\}} dt,\label{eqn-Delta1-def} \\
 \Delta^n_2(T) &\doteq \int_0^T \widetilde r\big(\frac{1}{\sqrt{n}}\hat q^n(\hat X^n_t)\big)\Ind_{\{\hat X^n_t\notin S_n\}} dt.\label{eqn-Delta2-def}
 \end{align}
The rest of the proof is a consequence of several lemmas that follow.
\end{proof} 

To keep the expressions concise, we define $\rQ^n_t\doteq  \frac{1}{\sqrt{n}}\hat q^n(\hat X^n_t)$. Observe that it suffices for us to show that 
\begin{align*}    \limsup_{n\to\infty}\limsup_{T\to\infty} \frac{1}{T} \log \E\Big[e^{\int_0^T \widetilde r (\rQ^n_t)dt +\Delta^n(T)}\Big]\leq \Lambda +\delta
 \end{align*} 
 as this will then imply 
 \begin{align*}
 \limsup_{n\to\infty}\hat \Lambda^n&\leq  \limsup_{n\to\infty}\limsup_{T\to\infty} \frac{1}{T} \log \E\Big[e^{\int_0^T \widetilde r (\hat Q^n_t)dt}\Big]\\
&= \limsup_{n\to\infty}\limsup_{T\to\infty} \frac{1}{T} \log \E\Big[e^{\int_0^T \widetilde r (\rQ^n_t)dt +\Delta^n(T)}\Big] \leq \Lambda.
 \end{align*}
 For $i=1,2$, we first note that since $\hat X^n_t-\hat Z^n_t=(e\cdot \hat X^n_t)^+\tilde U^n_t $, for some process $0\leq \tilde U^n\leq 1$, we have
\begin{align} \label{eq-delta-bound}	\Delta ^n_i(T) &\leq \sqrt d\max_{1\leq j\leq d} \kappa_j \int_0^T  \|\hat X^n_t\| \Ind_{\{\hat X^n_t\notin S_n\}} dt \nonumber \\
&\leq \sqrt d\max_{1\leq j\leq d} \kappa_j \int_0^T  \|\hat X^n_t\| \Ind_{\{\|\hat X^n_t\|>\bar \rho _d \sqrt n\}} dt,\end{align}
where $\bar \rho_d\doteq \frac{\rho_d}{\sqrt{d}}.$
\begin{lemma}\label{lem-tail-est}
	The following holds:
	\begin{align}
	\limsup_{L\to\infty}\limsup_{n\to \infty}\limsup_{T\to\infty}\frac{1}{T}\log\E\Big[e^{\int_0^T \widetilde r(\rQ^n_t)dt} \Ind_{\{\int_0^T \widetilde r(\rQ^n_t)dt \geq LT\}}\Big]=-\infty.
	\end{align}
\end{lemma}

\begin{proof}
Define a random variable $$Z^n_T\doteq e^{ \int_0^T \widetilde r(\rQ^n_t)dt - LT}.$$ 
Using this, we clearly have 
\begin{align*}
e^{-LT} \E\Big[ e^{\int_0^T \widetilde r(\rQ^n_t)dt}\Ind_{\{\int_0^T \widetilde r(\rQ^n_t)dt \geq LT\}}\Big]&= \E\Big[Z^n_T\Ind_{\{Z^n_T\geq 1\}}\Big] \\
&\leq  \E\Big[ (Z^n_T)^{\rho}\Big]=e^{-\rho LT} \E\Big[ e^{\rho \int_0^T \widetilde r(\rQ^n_t)dt}\Big],
\end{align*}
for $\rho>1$. This gives us
\begin{align*}
\limsup_{n\to\infty}\limsup_{T\to\infty}\frac{1}{T}\log\E\Big[e^{\int_0^T \widetilde r(\rQ^n_t)dt}&\Ind_{\{\int_0^T \widetilde r(\rQ^n_t)dt \geq LT\}}	\Big]\\
&\leq -(\rho-1)L + \limsup_{n\to\infty}\limsup_{T\to\infty}\frac{1}{T}\log\E\Big[	e^{\rho\int_0^T{\hat C_1}\|\hat X^n_t\|dt }\Big].
\end{align*}
 Using Corollary~\ref{cor-finite-cost-prelimit}, it is clear that for $\rho=1+\eps$ (for  $\eps>0)$, we can  
 ensure that the second term on the right hand side is finite. Now we take $L\to\infty$ to get the desired result.
	\end{proof}

	\begin{lemma} The following holds: for $i=1,2$,
	\begin{align} 
	\limsup_{n\to\infty} \limsup_{T\to\infty} \frac{1}{T} \log \E\Big[ e^{\Delta^n_i(T)}\Big] =-\infty, 
	\end{align}
		where $\Delta^n_1(T)$ and $\Delta^n_2(T)$ are defined in \eqref{eqn-Delta1-def} and \eqref{eqn-Delta2-def}, respectively.
	\end{lemma}

	\begin{proof} 
	The proof follows exactly along the same lines as the proof of Lemma~\ref{lem-tail-est}. 
	Here, we define random variable 
		 $$\bar Z^n_T\doteq \exp\left( \sqrt d\max_{1\leq j\leq d} \kappa_j\int_0^T \Big(\|\hat X^n_t\| \Ind_{\{\|\hat X^n_t\|>\bar \rho_d\sqrt n\}}- \bar \rho_d\sqrt n\Big)dt \right).$$ 
	Using the fact that $\bar Z^n_T\geq 1$, we have
\begin{align*}
\E\Big[\bar Z^n_T\Big]
&\leq  \E\Big[ (\bar Z^n_T)^{\eta}\Big]=\E\left[ \exp\left( \eta \sqrt d\max_{1\leq j\leq d} \kappa_j\int_0^T \Big(\|\hat X^n_t\| \Ind_{\{\|\hat X^n_t\|>\bar \rho_d\sqrt n\}}- \bar \rho_d\sqrt n\Big)dt \right)\right],
\end{align*}
for $\eta>1$.
This gives us
\begin{align*}
&\limsup_{T\to\infty}\frac{1}{T}\log\E\left[\exp\left( \sqrt d\max_{1\leq j\leq d} \kappa_j\int_0^T \Big(\|\hat X^n_t\| \Ind_{\{\|\hat X^n_t\|>\bar \rho_d\sqrt n\}}\Big)dt \right)	\right]\\
&\leq -(\eta-1)\bar \rho_d\sqrt n + \limsup_{T\to\infty}\frac{1}{T}\log\E\left[	\exp\left(\eta \sqrt d\max_{1\leq j\leq d}\kappa_j\int_0^T\|\hat X^n_t\|dt \right)\right].
\end{align*}
 Using Corollary~\ref{cor-finite-cost-prelimit}, it is clear that for $\eta=1+\eps$ (for $\eps>0)$, we can  again ensure that the second term on the right hand side is uniformly bounded in $n$. Now taking $n\to\infty$ and using~\eqref{eq-delta-bound}, we get the desired result.
	\end{proof}

A couple of immediate consequences of the previous two lemmas are the following: 
$$\limsup_{n\to\infty}\limsup_{T\to\infty} \frac{1}{T} \log \E\Big[e^{\int_0^T \widetilde r (\rQ^n_t)dt +\Delta^n(T)}\Big] \leq  \limsup_{n\to\infty}\limsup_{T\to\infty} \frac{1}{T} \log \E\Big[e^{\int_0^T \widetilde r (\rQ^n_t)dt}\Big] \doteq \ring \Lambda$$
and the corollary below.
\begin{corollary}\label{cor-trunc-limit-L} The following holds: 
	\begin{align*}
	\ring \Lambda= \lim_{L\to\infty}\limsup_{n\to\infty} \limsup_{T\to\infty}\frac{1}{T}\log\E\Big[e^{\int_0^T \widetilde r(\rQ^n_t)dt} \Ind_{\{\int_0^T \widetilde r(\rQ^n_t)dt \leq LT\}}\Big].
	\end{align*}
\end{corollary}
\begin{proof} It is clear that for $L>0$, 
	\begin{align} \label{eq-trunc-low-bound}
\limsup_{n\to\infty}	\limsup_{T\to\infty}\frac{1}{T}\log\E\Big[e^{\int_0^T \widetilde r(\rQ^n_t)dt} \Ind_{\{\int_0^T \widetilde r(\rQ^n_t)dt \leq LT\}}\Big]\leq \ring \Lambda.
	\end{align}
	It is also easy to see that 
	\begin{align}\nonumber 
	\ring \Lambda&\leq \max \Bigg\{ 	\limsup_{n\to\infty}\limsup_{T\to\infty}\frac{1}{T}\log\E\Big[e^{\int_0^T \widetilde r(\rQ^n_t)dt} \Ind_{\{\int_0^T \widetilde r(\rQ^n_t)dt \leq LT\}}\Big],\\\label{eq-trunc-up-bound}
	&\qquad \limsup_{n\to\infty}\limsup_{T\to\infty}\frac{1}{T}\log\E\Big[e^{\int_0^T \widetilde r(\rQ^n_t)dt} \Ind_{\{\int_0^T \widetilde r(\rQ^n_t)dt \geq LT\}}\Big]\Bigg\}.
	\end{align}
	Taking $L\to \infty$ and using Lemma~\ref{lem-tail-est}, we have the desired result by combining~\eqref{eq-trunc-low-bound} and~\eqref{eq-trunc-up-bound}.
	\end{proof} 
\begin{remark}Corollary~\ref{cor-trunc-limit-L} implies that  Lemma~\ref{lem-AO-upper} follows once we prove that
\begin{align}\label{eq-upper-bound-AO}
\limsup_{L\to\infty}\limsup_{n\to\infty} \limsup_{T\to\infty}\frac{1}{T}\log\E\Big[e^{\int_0^T \widetilde r(\rQ^n_t)dt} \Ind_{\{\int_0^T \widetilde r(\rQ^n_t)dt \leq LT\}}\Big]\leq \Lambda +\delta.\end{align} This is what we do next.	
	
\end{remark}

For $ L>0$, define 
$$\widetilde R_{ L}(\rQ^n,T)\doteq \frac{1}{T}\int_0^T \widetilde r(\rQ^n_t)\wedge Ldt\quad \text{ and } \quad \ring J^{n}_{L}\doteq \limsup_{T\to\infty}\frac{1}{T}\log \E\Big[e^{\big(\int_0^T \widetilde r(\rQ^n_t)dt\big)\wedge LT}\Big]\,.$$ 
 We remark that $\ring J^n_L$  is defined for a fixed underlying SCP defined in~\eqref{def-scp}. In what follows, we adopt the notation of Section~\ref{sec-var}.
\begin{lemma} \label{lem-fin}
	$$ \limsup_{L\to\infty}\limsup_{n\to\infty}\ring J^n_{L}\leq \Lambda.$$
\end{lemma}

\begin{proof}  
	Applying Proposition~\ref{prop-var-rep-poisson}, we write 
	\begin{align*}
	\ring J^n_{ L}	&= \limsup_{T\to\infty}\sup_{{\uppsi\in \cE^n}}\E\Big[\Big(\frac{1}{T}\int_0^T \widetilde r(\rQ^n_t)dt\Big)\wedge L -\frK^n(\uppsi,T)\Big]\\
	&\leq  \limsup_{T\to\infty}\sup_{{\uppsi\in \cE^n}}\E\Big[ \widetilde R_L(\rQ^{n,\uppsi},T) -\frK^n(\uppsi,T)\Big]\,.
	\end{align*}	
	where $ \rQ^{n,\Uppsi}_t\doteq \frac{1}{\sqrt{n}}\hat q^n(\hat X^{n,\uppsi}_t)$ and $\uppsi=(\phi,\psi,\varphi)$. Recall that $\frK^n(\uppsi,T)$ is defined in \eqref{eqn-frKn}.
	Now fix $\delta>0$. From Proposition~\ref{prop-var-rep-poisson}, it is also clear that 
	\begin{align}\label{eq-up-sup-1}
	\ring J^n_{ L}	\leq  \limsup_{T\to\infty}\sup_{{\uppsi\in \cE^n_M}}\E\Big[ \widetilde R_L(\rQ^{n,\uppsi},T) -\frK^n(\uppsi,T)\Big]+\delta,
	\end{align}
	for $M>0$ depending only on $L$ and $\delta$.   Let $\widetilde \uppsi^{n,T}(\cdot)= (\widetilde \phi^{n,T}(\cdot),\widetilde \psi^{n,T}(\cdot),\widetilde \varphi^{n,T}(\cdot))$ be the $\delta$--optimal control in the above supremum (it is clear that this depends on $n$ and $T$). From this choice  and the definition of $\cE^n_M$, we have 
	\begin{align}\label{eq-up-b-3}
	\ring J^n_{ L}	&\leq  \limsup_{T\to\infty}\E\Big[ \widetilde R_L(\rQ^{n,\widetilde \uppsi^{n,T}},T) -\frK^n(\widetilde \uppsi^{n,T},T)\Big]+2\delta
\end{align}	and 
	$\frK^n(\widetilde \uppsi^{n,T},T)\leq M$.
	From Corollary~\ref{cor-control-compact}, we know that  $\frK^n(\widetilde \uppsi^{n,T},T)\leq M$ implies 
	\begin{align}\label{eq-h-comp} \limsup_{n\to\infty}\sup_{T>0}\widecheck\frK^n(\widetilde \uppsi^{n,T},T)\doteq \limsup_{n\to\infty}\sup_{T>0}\frac{1}{T}\int_0^T \|h^n_t(\widetilde \uppsi^{n,T})\|^2dt \leq 2M\,. \end{align}
	Recall that $h^n(\widetilde \uppsi^{n,T})= \sqrt{n}\big( (e-\widetilde \phi^{n,T}), (e -\widetilde \psi^{n,T}), (e-\widetilde \varphi^{n,T})\big)$.
	 To proceed further, we require the stability of $\hat X^{n, \widetilde \uppsi^{n,T}}$, which we are unable to prove for $\widetilde \uppsi^{n,T}$. Hence, we  construct $\uppsi^{n,T}(\cdot)= ( \phi^{n,T}(\cdot), \psi^{n,T}(\cdot), \varphi^{n,T}(\cdot))$ that is $2\delta$--optimal to the supremum in~\eqref{eq-up-sup-1}, for which we can prove the stability of $\hat X^{n,\uppsi^{n,T}}$.  We now proceed with the construction of $\uppsi^{n,T}$ and prove that the constructed $\uppsi^{n,T}$ is indeed, $2\delta$--optimal. For $K>0$, define $\tau^n_K\doteq \inf\{ t>0: \|h^n(\widetilde\uppsi^{n,T})\|>K\}\wedge T$.
		 Then, $\uppsi^{n,T}(\cdot)= ( \phi^{n,T}(\cdot), \psi^{n,T}(\cdot), \varphi^{n,T}(\cdot))$  is defined as  
	 \begin{align*}
	 \phi^{n,T}(\cdot) \doteq \Big(e- \big(e-\phi^{n,T}(\cdot)\big) &\Ind_{[0,\tau^n_K]}(\cdot)\Big)\,, \quad 
	  \psi^{n,T}(\cdot)\doteq \Big(e-\big(e -\psi^{n,T}(\cdot)\big)\Ind_{[0,\tau^n_K]}(\cdot)\Big)\,,\\
	  \varphi^{n,T}(\cdot)&\doteq \Big(e- \big(e -\varphi^{n,T}(\cdot)\big) \Ind_{[0,\tau^n_K]}(\cdot)\Big) \,, \end{align*}
	 where $K$ is chosen later.   We simply drop the dependence on $T$ and write $\widetilde \uppsi^{n,T}(\cdot)$ as $\widetilde \uppsi^n=(\widetilde \phi^n,\widetilde \psi^n,\widetilde \varphi^n)$ and $ \uppsi^{n,T}(\cdot)$ as $ \uppsi^n=( \phi^n, \psi^n, \varphi^n)$.
	 From the definition of $\frK^n(\cdot,T)$, it immediately follows that 
	 \begin{align}\label{eq-newcontrol} \frK^n(\uppsi^n,T)\leq \frK^n(\widetilde \uppsi^n,T)\leq M\,.\end{align}

For $t\leq \tau^n_R$,	 since $\uppsi^n_t=\widetilde \uppsi^n_t$, it is trivial to see that $\hat X^{n,\uppsi^n}_t=\hat X^{n,\widetilde \uppsi^n}_t$. Using~\eqref{eq-newcontrol}, we then have
\begin{align*}
\E&\Big[ \widetilde R_L(\rQ^{n, \uppsi^n_T},T) -\frK^n( \uppsi^n,T)\Big]-\E\Big[ \widetilde R_L(\rQ^{n, \widetilde\uppsi^n_T},T) -\frK^n( \widetilde\uppsi^n,T)\Big]\\
&\geq\E\Big[ \widetilde R_L(\rQ^{n, \uppsi^n},T)\Big]-\E\Big[ \widetilde R_L(\rQ^{n, \widetilde\uppsi^n},T) \Big]\\
&\geq \E\Big[\frac{1}{T}\int_0^T \widetilde r(\rQ^{n, \uppsi^n}_t)\wedge Ldt\Big]-\E\Big[\frac{1}{T}\int_0^T \widetilde r(\rQ^{n,\widetilde\uppsi^n}_t)\wedge Ldt\Big]\\
&\geq \E\Big[\frac{1}{T}\int_0^T \Big(\widetilde r(\rQ^{n, \uppsi^n}_t)\wedge L\Big) \Ind_{[\tau^n_K,\infty)}(t)dt\Big]-\E\Big[\frac{1}{T}\int_0^T \Big( \widetilde r(\rQ^{n,\widetilde\uppsi^n}_t)\wedge L\Big)\Ind_{[\tau^n_K,\infty)}(t) dt\Big]\\
&\geq   -\frac{2L}{T}\int_0^T\PP(\tau^n_K<t)d t= -\frac{2L}{T} \E\Big[\int_0^T \Ind_{[\tau^n_K,\infty)}(t)d t\Big]\,.
\end{align*}
From~\eqref{eq-h-comp}, we can now choose $K=K(L,\delta,M)$ large enough such that $2LT^{-1}\E\Big[\int_0^T \Ind_{[\tau^n_K,\infty)}(t)d t\Big]<\delta$, uniformly in $T$. This proves the desired $2\delta$--optimality of $\uppsi^n$.

	 Therefore, from the above,~\eqref{eq-up-b-3} becomes
	 \begin{align*}
	 \ring J^n_L&\leq \limsup_{T\to\infty}\Bigg(\E\Big[ \widetilde R_L(\rQ^{n,\uppsi^n},T)\Big] -\E\Big[ \frK^n(\uppsi^n,T)\Big]\Bigg)+3\delta\,.
	 \end{align*}
	 Our next goal is to analyze 
	 $$\limsup_{T\to\infty}\Bigg(\E\Big[ \widetilde R_L(\rQ^{n,\uppsi^n},T)\Big] -\E\Big[ \frK^n(\uppsi^n,T)\Big]\Bigg). $$
	 From the definition of $\uppsi^n $ constructed above, it is clear that  $\|h^n_t(\uppsi^n)\|\leq K$, for every $t\geq 0$. Since we already know that the mean empirical measures of $\{h^n(\uppsi^n)\}_{n\in \NN}$ are tight (from~\eqref{eq-newcontrol}), we will focus on showing the tightness of the mean empirical measures of $\hat X^{n,\uppsi^n}$. Moreover, we will show that  	
	\begin{align}\label{eq-tight-emp} \limsup_{n\to\infty}\limsup_{T\to\infty} \frac{1}{T}\E\Big[\int_0^T \|\hat X^{n,\uppsi^n}_t\|dt\Big]\leq R,\end{align}
	for some $R>0$.  We assume this for now. It is rigorously stated and proved in Lemma~\ref{lem-tightness-empirical}. 
	Using~\eqref{eq-tight-emp}, we now proceed  to analyze  
 \begin{align*}
 &\limsup_{T\to\infty}\Bigg(\E\Big[ \widetilde R_L(\rQ^{n,\uppsi^n},T)\Big] -\E\Big[ \frK^n(\uppsi^n,T)\Big]\Bigg)\\
 &=\limsup_{T\to\infty}\Bigg(\E\Big[\frac{1}{T}\int_0^T \widetilde r(\rQ^{n,\uppsi^n}_t)\wedge Ldt\Big] -\E\Big[ \frK^n(\uppsi^n,T)\Big]\Bigg).
   \end{align*}
 Since the  first term above on the right hand side can be expressed as an integral over the mean empirical measures of the processes $\hat X^{n,\uppsi^n}$ (denoted by $ \mu^{n,T,3}_{v^\delta,\uppsi}$; see~\eqref{def-markov-emp-m} for its definition),    it is clear that $\{\mu^{n,T,3}_{v^\delta,\uppsi}\}_{n,T}\subset \calP(\RR^d)$ is tight in $n$ and $T$,  from~\eqref{eq-tight-emp}. Therefore, along a subsequence $T_k$, there are measures $ \mu^{n,3}_{v^\delta,\uppsi} \in \calP(\RR^{d})$ which are the weak limits of $\{ \mu^{n,T_k,3}_{v^\delta,\uppsi}\}_{n,k}$. This gives us  
   \begin{align} \label{eq-tight-emp2}
  &\limsup_{k\to\infty}\Bigg(\E\Big[ \frac{1}{T_k}\int_0^{T_k} \widetilde r(\rQ^{n,\uppsi^n}_t)\wedge Ldt\Big] -\E\Big[ \frK^n(\uppsi^n,T_k)\Big]\Bigg) \nonumber\\
  &=  \int_{\RR^{d}}\widetilde r\Big(\vartheta\big(\sqrt n (e\cdot \hat x^n)v^{\delta}(\hat x^n) \big)\Big)\wedge Ld\mu^{n,3}_{v^\delta,\uppsi}(x)- \limsup_{k\to\infty}\E\Big[ \frK^n(\uppsi^n,T_k)\Big]\,. 
   \end{align}

   It now only remains to take the limit as $n \to \infty$. From Theorem~\ref{thm-diff}, we know that for $T>0$, $\hat X^{n,\uppsi^n}$ converges weakly to $X^{*,v^\delta,w^*}$ on $\frD_T^d$, for some $ L^2([0,T],\RR^d)$--valued random variable $w^*$. Here, $X^{*,v^\delta,w^*}$ is given as the solution to~\eqref{eq-diff-limit-control} with $u_t= v^\delta(X_t^{*,v^\delta,w^*})$ and $w_t=w^*_t$.
   
   Again, from~\eqref{eq-tight-emp}, we can ensure that there exist a subsequence $n_k$ of $n$ and a measure $\widetilde \mu^{*}_{v^\delta,w^*}$ of process $X^{*,v^\delta,w^*}$ such that 
 
   \begin{align*}
   \lim_{k\to\infty} &\bigg\{ \int_{\RR^{d}}\widetilde r\Big(\vartheta\big(\sqrt n_k (e\cdot \hat x^{n_k})v^{\delta}(\hat x^{n_k}) \big)\Big)\wedge Ld \mu^{n_k,3}_{v^\delta,\uppsi}(x)-\limsup_{m\to\infty}\E\Big[ \frK^{n_k}(\uppsi^{n_k},T_m)\Big] \bigg\}\\
   &= \int_{\RR^{d}}\widetilde r\big((e\cdot x)v^\delta (x)\big)\wedge Ld\widetilde \mu^{*}_{v^\delta,w^*}(x)-\limsup_{k\to\infty}\limsup_{m\to\infty}\E\Big[ \frK^{n_k}(\uppsi^{n_k},T_m)\Big].
   \end{align*}
   In the above, we have used the fact that 
   $$\Big\| \frac{1}{\sqrt n}\vartheta\big(\sqrt n (e\cdot \hat x^n)v^{\delta}(\hat x^n) \big) - (e\cdot x)^+ v^\delta (x)\Big\|\leq \frac{2d}{\sqrt n}\,.
   $$ 
  Thus, from \eqref{eq-tight-emp2} and the above, we obtain  
 \begin{align}\label{eq-last-1}
 &\limsup_{n\to\infty}\limsup_{T\to\infty}\Bigg(\E\Big[\frac{1}{T}\int_0^{T} \widetilde r(\rQ^{n,\uppsi^n}_t)\wedge Ldt\Big] -\E\Big[ \frK^n(\uppsi^n,T)\Big]\Bigg) \nonumber\\
 &\leq  \int_{\RR^{d}}\widetilde r\big((e\cdot x)v^\delta (x)\big)\wedge Ld\widetilde \mu^{*}_{v^\delta,w^*}(x)-\limsup_{k\to\infty}\limsup_{m\to\infty}\E\Big[ \frK^{n_k}(\uppsi^{n_k},T_m)\Big]. 
 \end{align}
 Now let us compute the second term on the right hand side of the above display. From Lemma~\ref{lem-compact} and Theorem~\ref{thm-diff}, we have
 \begin{align}\label{eq-last-2}
 \limsup_{k\to\infty}\limsup_{m\to\infty}\E\Big[ \frK^{n_k}(\uppsi^{n_k},T_m)\Big]\geq \limsup_{T\to\infty} \frac{1}{2T}\int_0^T \|w^*_t\|^2 dt\,. 
 \end{align}

 For large enough $L$, using Corollary~\ref{cor-trunc-limit-L}, and equations~\eqref{eq-last-1} and~\eqref{eq-last-2}, we can  ensure that 
 \begin{align}\label{eq-weak-conv}
 \limsup_{n\to\infty} \hat \Lambda^n\leq  \ring \Lambda\leq  \int_{\RR^{d}}\widetilde r\big((e\cdot x)v^\delta (x)\big)\wedge Ld\widetilde \mu^{*}_{v^\delta,w^*}(x)-\limsup_{T\to\infty} \frac{1}{2T}\int_0^T \|w^*_t\|^2 dt +\delta.
 \end{align}
Now using Lemma~\ref{lem-trivial-bound}, we have
$$\limsup_{n\to\infty} \hat \Lambda^n\leq \Lambda_{v^\delta}+\delta\leq \Lambda+2\delta. $$
Arbitrariness of $\delta>0$ gives us the result.
\end{proof}

Now all that remains to be shown is~\eqref{eq-tight-emp}. To that end, we recall and set up a few useful definitions. 
	Recall that with SCP $Z^n$ as defined in~\eqref{def-scp},  $\hat X^{n,\uppsi}$ is given as the solution to 
	\begin{align}\nonumber
	\hat{X}^{n,\uppsi}_{i,t} &= \hat{X}^{n,\uppsi}_{i,0} + \ell^n_i t - \mu^n_i \int_0^t\hat{Z}^{n,\uppsi}_{i,s} ds - \gamma^n_i \int_0^t\hat{Q}^{n,\uppsi}_{i,s} ds\\\nonumber 
	&\qquad + \frac{1}{\sqrt{n}}\Big( \widecheck A^n_i(\int_0^t\phi_{i,s}ds) - \lambda^n_i t \Big) - \frac{1}{\sqrt{n}}\bigg(\widecheck S^n_i\left(\int_0^t \psi_{i,s}\frac{Z^{n,\uppsi}_{i,s}}{n} ds\right)  - \mu^n_i \int_0^t  Z^{n,\uppsi}_{i,s} ds  \bigg)\\\nonumber
	&\qquad - \frac{1}{\sqrt{n}}\bigg(\widecheck R^n_i\left( \int_0^t\varphi_{i,s} \frac{Q^{n,\uppsi}_{i,s}}{n} ds\right)  - \gamma^n_i \int_0^t Q^{n,\uppsi}_{i,s}ds  \bigg)\,,
	\end{align}
	where  $\hat X^{n,\uppsi}_t= \hat Q^{n,\uppsi}_t+\hat Z^{n,\uppsi}_t$.
	
	To understand the tightness of the empirical measures of  $\hat X^{n,\uppsi}_t$, we consider the operator $\frL^{n,w}$ with $$w=\Big(\{w^1_i\}_{i=1}^d,\{w^2_i\}_{i=1}^d,\{w^3_i\}_{i=1}^d\Big)$$ defined as
	\begin{align*}
	\frL^{n,w}f(x)&\doteq \sum_{i=1}^d\Big(\lambda^n_i w^1_i \diff f(x;e_i) + \big(\mu^n_i z_iw^2_i +\gamma^n_i q_i(x,z)w^3_i\big) \diff f(x,-e_i)\Big) \,\\
	&=\Lg^{n,u}f(x) + \sum_{i=1}^d\Big(\lambda^n_i(1- w^1_i) \diff f(x;e_i) + \big(\mu^n_i z_i(1-w^2_i) +\gamma^n_i q_i(x,z)(1-w^3_i)\big) \diff f(x,-e_i)\Big).
	\end{align*} 
	Note that $\Lg^{n,u}$ is the infinitesimal generator of the process $X^n$ (under a constant SCP $u$) given in \eqref{eq-Lg-hatXn}.  We suppressed the dependence of $\frL^{n,w}$ on $u$.	Here, $q_i(x,z)=x_i-z_i$ and $e_i$ is a $\RR^d$ vector with $1$ as $i$th coordinate and $0$, everywhere else. Motivated from \cite[Theorem 3.4]{AHP21},	let $\eta: \RR\rightarrow\RR$ be a  convex function in $\cC^2 (\RR)$ defined as follows: 
	\begin{align}\label{def-eta}
	\eta(t)\doteq \begin{cases}
	-\frac{1}{2}, \text{ for $t\leq -1$},\\
	(t+1)^3-\frac{1}{2}(t+1)^4-\frac{1}{2}, \text{ for $-1\leq t\leq 0$},\\
	t, \text{ for $t\geq 0$}\,,
	\end{cases}
	\end{align}
	and, let $\xi:\RR^d\rightarrow \RR$ be defined as $ \xi(x)\doteq \sum_{i=1}^d \frac{\eta(x_i)}{\mu_i}$ 
	 and
	$$ \srZ(x)\doteq \eps_0\eps_1 \xi(-x)+\eps_0 \xi(x).$$
	 Here, $\eps_0$ and $\eps_1$ are some positive constants whose values  are not relevant for us. 
	
	\begin{lemma}\label{lem-tightness-empirical}
		Suppose that $\uppsi^n= (\phi^n,\psi^n,\varphi^n)$ is as in Lemma~\ref{lem-fin}.
		Then, 	$$ \limsup_{n\to\infty}\limsup_{T\to\infty} \frac{1}{T}\E\Big[\int_0^T \|\hat X^{n,\uppsi^n}_t\|dt\Big]\leq R,$$
		for some $R>0$.	
	\end{lemma}
	\begin{proof}Following the similar calculations as done in the proof of \cite[Theorem 3.4]{AHP21}, we have
		\begin{align}\label{eq-lyap-fost}
		\frL^{n,w} \srZ(\hat x^n(x))\leq &\hat C_0-\hat C_1\|\hat x^n(x)\| + \sum_{i=1}^d\Big(\lambda^n_i(1- w^1_i) \diff \srZ(\hat x^n (x);n^{-\frac{1}{2}}e_i) \nonumber\\
		&\qquad\qquad + \big(\mu^n_i z_i(1-w^2_i) +\gamma^n_i q_i(x,z)(1-w^3_i)\big) \diff \srZ(\hat x^n (x),-n^{-\frac{1}{2}}e_i)\Big)\,.
		\end{align}
		Note that $|\diff \srZ(\hat x^n (x), \pm n^{-\frac{1}{2}} e_i)|\leq \hat C_4n^{-\frac{1}{2}}$, for some $\hat C_4>0$ independent of $n$.
		We now apply It{\^o}'s formula to $\srZ(\hat X^{n,\uppsi^n}_t)$, to get
		\begin{align*}
		\E[\srZ(\hat X^{n,\uppsi^n}_T)]&= \srZ(\hat X^{n,\uppsi^n}_0)+ \E\Big[\int_0^T \frL^{n,w} \srZ(\hat X^{n,\uppsi^n}_t) dt\Big] \\
		&\leq \srZ(\hat X^{n,\uppsi^n}_0) +\hat C_0 T- \hat C_1 \E\Big[\int_0^T \|\hat X^{n,\uppsi^n}_t\|dt\Big]\\
		& + \sum_{i=1}^d\E\Bigg[\int_0^T\Big( \hat C_4\lambda^n_i |1-\phi^n_{i,t}|n^{-\frac{1}{2}} + \hat C_4\hat C_5\mu^n_i |1-\psi^n_{i,t}|\| X^{n,\uppsi^n}_t\|n^{-\frac{1}{2}} \\
		& \qquad \qquad \qquad \qquad \qquad  + \hat C_4\hat C_5\gamma^n_i|1-\varphi^n_{i,t}| \| X^{n,\uppsi^n}_t\|  n^{-\frac{1}{2}}\Big)dt\Bigg].  
		\end{align*}
		In the above, to get the second term in the integral, we have used the fact that $\max\{\| Z^{n,\uppsi^n}_t\|,\|Q^{n,\uppsi^n}_t\|\}\leq \hat C_5 \| X^{n,\uppsi^n}_t\|$ for some $\hat C_5>0$ which is independent of $t$ and $n$.
		
		From the above equation, it is clear that 
		\begin{align}\label{eq-lyap-fost-1}
		&\frac{\hat C_1}{T}\E\Big[\int_0^T \|\hat X^{n,\uppsi^n}_t\|dt\Big]\nonumber \\
		& \leq \hat C_0 + \sum_{i=1}^d\frac{1}{T}\E\Bigg[\int_0^T\Big( \hat C_4\lambda^n_i |1-\phi^n_{i,t}|n^{-\frac{1}{2}} + \hat C_4\hat C_5\mu^n_i |1-\psi^n_{i,t}|\| X^{n,\uppsi^n}_t\|n^{-\frac{1}{2}} \nonumber \\
		& \qquad \qquad \qquad \qquad \qquad \qquad  + \hat C_4\hat C_5\gamma^n_i|1-\varphi^n_{i,t}| \| X^{n,\uppsi^n}_t\|  n^{-\frac{1}{2}}\Big)dt\Bigg].
		\end{align}
		
		We will show that the following hold: 
		\begin{align}\nonumber
		\max\Big\{& \frac{1}{T}\E\Big[\int_0^T \mu^n_i |1-\psi^n_{i,t}|\| X^{n,\uppsi^n}_t\|n^{-\frac{1}{2}}dt\Big],\\\label{eq-aug-cont-est-1}
		&\qquad \frac{1}{T}\E\Big[\int_0^T\gamma^n_i|1-\varphi^n_{i,t}| \|X^{n,\uppsi^n}_t\| n^{-\frac{1}{2}}dt\Big]\Big\}= \hat C_6+ O(n^{-\frac{1}{2}}) \frac{1}{T}\E\Big[\int_0^T \|\hat X^{n,\uppsi^n}_t\|dt\Big],
		\end{align}
for some $\hat C_6>0$		and 
		\begin{align}\label{eq-aug-cont-est-2}
		\limsup_{n\to\infty}\limsup_{T\to\infty} \frac{1}{T}\int_0^T\lambda^n_i |1-\phi^n_{i,t}|n^{-\frac{1}{2}}dt<\infty.
		\end{align}
		We will first show~\eqref{eq-aug-cont-est-2}. Since $n^{-1}\lambda^n_i\to \lambda_i$, we have
		\begin{align*}
		\limsup_{n\to\infty}\limsup_{T\to\infty} \frac{1}{T}\int_0^T\lambda^n_i |1-\phi^n_{i,t}|n^{-\frac{1}{2}}dt\leq \lambda_i \limsup_{n\to\infty}\limsup_{T\to\infty} \frac{1}{T}\int_0^T \sqrt{n}|1-\phi^n_{i,t}|dt. 
		\end{align*}
		From Corollary~\ref{cor-control-compact} and Cauchy-Schwartz inequality, we have
		$$ \limsup_{n\to\infty}\limsup_{T\to\infty} \frac{1}{T}\int_0^T \sqrt{n}|1-\phi^n_{i,t}|dt\leq\limsup_{n\to\infty}\limsup_{T\to\infty} \Big(\frac{1}{T}\int_0^T n|1-\phi^n_{i,t}|^2dt\Big)^{\frac{1}{2}} <\infty.$$
		Combining the previous two displays gives us~\eqref{eq-aug-cont-est-2}.	 We now proceed to show~\eqref{eq-aug-cont-est-1}. We only show for one of the terms of~\eqref{eq-aug-cont-est-1} \emph{viz.,} $$ \frac{1}{T}\E\Big[\int_0^T \mu^n_i |1-\psi^n_{i,t}|\|X^{n,\uppsi^n}_t\|n^{-\frac{1}{2}}dt\Big]= O(n^{-\frac{1}{2}}) \frac{1}{T}\E\Big[\int_0^T \|\hat X^{n,\uppsi^n}_t\|dt\Big]$$
		as the proof for the other term in~\eqref{eq-aug-cont-est-1} follows along similar lines. To that end,
		\begin{align*}
		&\frac{1}{T}\E\Big[\int_0^T \mu^n_i |1-\psi^n_{i,t}|\| X^{n,\uppsi^n}_t\|n^{-\frac{1}{2}}dt\Big]\\
		&=	\frac{1}{T}\E\Big[\int_0^T \mu^n_i |1-\psi^n_{i,t}|\| X^{n,\uppsi^n}_t-\rho n +\rho n\|n^{-\frac{1}{2}}dt\Big] \\ 
		&\leq \frac{1}{T}\E\Big[\int_0^T \mu^n_i |1-\psi^n_{i,t}|\Big(\| X^{n,\uppsi^n}_t-\rho n\| +\rho n\Big)n^{-\frac{1}{2}}dt\Big]\\
		&\leq  \frac{1}{T}\E\Big[\int_0^T \mu^n_i |\sqrt{n}(1-\psi^n_{i,t})|\|n^{-\frac{1}{2}} (X^{n,\uppsi^n}_t-\rho n)\| n^{-\frac{1}{2}}dt\Big]+\frac{1}{T}\E\Big[\int_0^T \mu^n_i |1-\psi^n_{i,t}|\rho n^{\frac{1}{2}}dt\Big]\\
			&\leq  \frac{n^{-\frac{1}{2}}}{T}\E\Big[\int_0^T \mu^n_i |\sqrt{n}(1-\psi^n_{i,t})|\|\hat X^{n,\uppsi^n}_t\| dt\Big]+\frac{1}{T}\E\Big[\int_0^T \mu^n_i |1-\psi^n_{i,t}|\rho n^{\frac{1}{2}}dt\Big].
		\end{align*}
		Using  a similar argument in proving~\eqref{eq-aug-cont-est-2}, we can show that the second term above is bounded uniformly in $T$ and $n$.  Recall that $\|h^n(\uppsi^n)\|\leq K$, from the construction of $\uppsi^n$. This means that 
		\begin{align*}
		\frac{n^{-\frac{1}{2}}}{T}\E\Big[\int_0^T \mu^n_i |\sqrt{n}(1-\psi^n_{i,t})|\|\hat X^{n,\uppsi^n}_t\| dt\Big]\leq \frac{K\mu^n_i n^{-\frac{1}{2}}}{T}\E\Big[\int_0^T  \|\hat X^{n,\uppsi^n}_t\| dt\Big].
		\end{align*}
		This proves~\eqref{eq-aug-cont-est-1} as $\mu^n_i\to\mu_i$. Now substituting~\eqref{eq-aug-cont-est-1} and~\eqref{eq-aug-cont-est-2} in~\eqref{eq-lyap-fost-1} for sufficiently large $n$ gives us
	$$ \limsup_{T\to\infty} \frac{1}{T}\E\Big[\int_0^T \|\hat X^{n,\Uppsi^n}_t\|dt\Big]\leq R,$$
	for some $R>0$.  Now taking $n\to\infty$ gives us the desired result.
	\end{proof}

\bigskip

\appendix
\section*{Appendix: Auxiliary results}

\section{A result on ERSC problem for limiting  diffusion}

\begin{lemma}\label{lem-cont-control}
	For every $\delta>0$ and $u_0\in \bU$, there exists a control $v\in \Usm$ such that $\Lambda_v\leq \Lambda+2\delta$
	such that $v(\cdot)=u_0\in \bU$ outside a compact set $K=K(\delta)$ and $v(\cdot)$ is continuous.
\end{lemma}
\begin{remark}
	This result is the risk-sensitive analog to \cite[Theorem 4.1]{ABP15} and we adapt the proof of  \cite[Proposition 1.3]{AB18}. The main difference in the proof of \cite[Proposition 1.3]{AB18} and the proof below is that, instead of being a nearly optimal control in \cite[Theorem 1.3]{AB18}, $u_0$ (in~\eqref{def-trunc-drift}) is replaced by any control in $ \bU$. 	
	 This helps us in constructing nearly optimal Markov controls that are continuous. This fact is used in invoking the continuity of the integrand in~\eqref{eq-weak-conv} and weak convergence of measures to conclude the equality in~\eqref{eq-weak-conv}.

\end{remark}
\begin{proof}
For $l>0$, define 
\begin{align}\label{def-trunc-drift}
b_l(x,u)= \begin{cases}
b (x,u), \text{ if $x\in B_l$}\,,\\
b(x,u_0), \text{ if $x\notin B_l$}\,.
\end{cases}
\end{align}
$r_l(x,u)$ is similarly defined. The corresponding generator is denoted by $\Lg^u_l$. Following the similar argument as that of \cite[Proposition 1.3]{AB18}, we can conclude that there exists a pair $(V_l,\Lambda_{l,u_0})\in W^{2,p}_{\text{loc}}(\RR^d)\times \RR$ such that 
\begin{align}\label{eq-trunc-hjb}
\min_{u\in \bU}\Big[ \Lg^u_l V_l +r_l(x,u)V_l\Big]= \Lambda_{l,u_0} V_l,
\end{align}
where $V_l$ restricted to $B_l$ is in $\cC^2(\RR^d)$ and $\Lambda_{u_0}\geq \Lambda_{l,u_0}\geq \Lambda.$ Moreover, 
$$ \inf_{l>0} \inf_{x\in \RR^d}V_l(x)>0.$$
 Let $v^*_l$ be the measurable selector of~\eqref{eq-trunc-hjb}. Since $v^*_l=u_0$ on $B_l^c$, we can infer that $v^*_l\in \Ussm$. 
 
 Again following along the same lines as in \cite[Proposition 1.3]{AB18}, we can conclude that there exists a subsequence $l_n$ along which $(V_{l_n},\Lambda_{l_n,u_0})$ converge to $(\widecheck V,\widecheck \Lambda)$ in $W^{2,p}_{\text{loc}}(\RR^d)\times \RR$ such that
 \begin{align}\label{eq-trunc-limit-hjb}
 \min_{u\in \bU}\Big[ \Lg^u \widecheck V+ r(x,u)\widecheck V\Big]= \widecheck \Lambda \widecheck V
 \end{align} 
 with $\widecheck V\in \cC^2(\RR^d)$ and $\Lambda\leq \widecheck \Lambda \leq \Lambda_{u_0}$. Fix $u\in \Uadm$. From here, the usual application of It{\^o}'s formula to $e^{\int_0^T r(X_t,u_t)dt}V(X_t)$ (then followed by Fatou's lemma) gives us the following:
 \begin{align*}
 \widecheck \Lambda\leq \limsup_{T\to\infty} \frac{1}{T}\log \E\Big[e^{\int_0^T r(X_t,u_t)dt}\Big].
 \end{align*}
 This together with the fact that $\widecheck \Lambda \geq \Lambda$ implies that $\widecheck \Lambda=\Lambda.$
 
 A trivial consequence of this fact is that for every $\delta>0$, we can choose a large enough $l$ such at that $v^*_l$ (which is constant $u_0$ outside $B_l$) is the $\delta$--optimal control for our ERSC problem.
 
 To deduce the continuity of the control $v^*_l$ on $B_l$, we invoke \cite[Proposition 3]{AMR04}. From here, we cannot yet conclude the result because the $v^*_l$ can still be discontinuous on $\partial B_l$. To resolve this, we follow the approach of \cite[Pg. 3566]{ABP15},  \emph{viz.,} consider a family $\{\rho^k\}_{k\in \NN}$ of cut-off functions such that $\rho^k=	0$ on $B^c_{l-1/k}$ and $\rho^k=1$ on $B_{l-2/k}$. Now consider a new control $v^k_l(x)\doteq \rho^k(x) v^*_l(x)+ (1-\rho^k(x))u_0$. Clearly, $v^k_l\to v^*_l$, as $k\to\infty$, uniformly on compliment of any neighborhood of $\partial B_l$. Using \cite[Theorem 4.3]{ari2018strict}, we can conclude that $\Lambda_{v^k_l}\to\Lambda_{v^*_l}$ as $k\to\infty.$ Therefore, we can choose $k$ large enough such that $\Lambda_{v^k_l}\leq \Lambda_{v^*_l}+\delta\leq \Lambda+2\delta$. This finishes the proof.
\end{proof}

\section{Some properties on the function $\varkappa$ in \eqref{eqn-varkappa}}\label{sec-varkappa}
\begin{lemma}\label{lem-varkappa-below-0}
	For $-1<r\leq 0$, we have the following:
	$$ \varkappa(1+r)\geq \frac{1}{2}|r|^2.$$
	
\end{lemma}

\begin{proof} Using Taylor's theorem around $r=0,$ we have
	$$\varkappa(1+ r)= \frac{1}{2}|r|^2-\frac{r^3}{6r_0^2}, \text{ for some $r_0\in (r,0)$}\,.  $$
	In the above, we have used the fact that $\varkappa (1)=0$, $\frac{d}{dr}\varkappa (1+ r)_{|r=0}=0$ and $\frac{d^2}{dr^2}\varkappa (1+r)_{|r=0}=\frac{1}{2}$ and $\frac{d^3}{dr^3}\varkappa (1+ r)_{|r=0}=-1$. It is now clear that the desired result holds.	
\end{proof}
\begin{lemma}\label{lem-varkappa-above-0}
	For $l>0$ and $\alpha(l)\doteq\frac{1}{2(1+l)}$, 
	$$ \varkappa(1+r)\geq \alpha(l)|r|^2, \quad \text{whenever $0\leq r<l.$}$$
\end{lemma}
\begin{proof}  Fix $l>0$.
	From the definition of $\varkappa,$ 
	$$ \varkappa(1+r)= (1+r)\ln (1+r) - (1+r)+1= (1+r)\ln (1+r)-r.$$
	Clearly at $r=0$, $\varkappa(1+r)-\alpha(l)|r|^2=0$. Therefore, to prove the lemma, it suffices to show that $\beta (r)\doteq \varkappa(1+r)-\alpha(l) |r|^2$ is non-decreasing in $(0,l)$. Indeed, if $\beta(r)$ is non-decreasing, then for any $r\in (0,l)$, we will have $\beta(r)\geq 0$.
	
	  We will next show that 
	\begin{align*} \gamma(r)\doteq \frac{d}{dr}\beta(r)= \ln(1+r)-2\alpha(l) r\geq 0, \quad \text{whenever $0<r<l$.}\end{align*}
	 First note that $\gamma(r)=0$ at $r=0$.  Consider $\frac{d}{dr}\gamma(r)=\frac{1}{1+r} -2\alpha(l) \geq 0$ in $(0,l)$. We have 
	\begin{align*}
	\frac{d}{dr}\gamma(r)\geq 0\implies r<\frac{1}{2\alpha(l)}-1.
	\end{align*}
	From the definition of $\alpha(l)$, we have the result.
\end{proof}

\begin{proof}[Proof of Lemmas~\ref{lem-compact}] 
Fix a subsequence of $n$ which we again denote by $n$.
By Lemmas~\ref{lem-varkappa-below-0} and~\ref{lem-varkappa-above-0}, we obtain 
	\begin{align}\nonumber
	\frac{1}{T}\int_0^T \varkappa(\phi^n_t)dt &=\frac{1}{T}\int_0^T \varkappa (\phi^n_t)\Ind_{\{0<\phi^n_t\leq 1\}}dt+   \frac{1}{T}\int_0^T\sum_{L=0}^\infty \varkappa (\phi^n_t)\Ind_{\{1+ \frac{L}{\sqrt{n}}<\phi^n_t\leq 1+\frac{L+1}{\sqrt{n}}\}}dt\\\nonumber
	&\geq \frac{1}{2T}\int_0^T |1-\phi^n_t|^2\Ind_{\{0<\phi^n_t<1\}}dt \\\label{eq-comp-lem-1}
	&\qquad\qquad + \sum_{L=0}^\infty\alpha \big(1+ \frac{L+1}{\sqrt{n}}\big)\frac{1}{T}\int_0^T |1-\phi^n_t|^2\Ind_{\{1+ \frac{L}{\sqrt{n}}<\phi^n_t\leq 1+\frac{L+1}{\sqrt{n}}\}}dt.
	\end{align} 
	Using monotone convergence theorem, we can justify the interchange of the integral and summation in the above the equation. In the last equation, we have used Lemmas~\ref{lem-varkappa-below-0} and~\ref{lem-varkappa-above-0} for the first and second terms, respectively.
Define $\mathscr{N}\doteq \{-1,0,1,2,\ldots\}$ and  a family of functions $\{f^{n,L}\}_{n,L}$ for $n\in \NN$ and $L\in \mathscr{N}$ as 
$$ f^{n,L}_t=\begin{cases} &(1-\phi^n_t)\Ind_{\{0<\phi^n_t\leq 1\}}, \text{ for $L=-1$,}\\
& (1-\phi^n_t)\Ind_{\{1+ \frac{L}{\sqrt{n}}<\phi^n_t\leq 1+\frac{L+1}{\sqrt{n}}\}}, \text{ for $L\neq -1$}.\end{cases}
	$$
	It is clear that from~\eqref{eq-bound}, we have
	\begin{align*}
	\frac{1}{2T}\int_0^T |\sqrt{n}f^{n,-1}_t|^2dt + \sum_{L=0}^\infty\frac{\alpha \big(1+ \frac{L+1}{\sqrt{n}}\big)}{T}\int_0^T |\sqrt{n}f^{n,L}_t|^2dt\leq  \frac{n}{T}\int_0^T \varkappa(\phi^n_t)dt\leq  M<\infty.
	\end{align*}
	 From here, it is clear that $$ \big\{g^{n,L}\doteq \sqrt{n}f^{n,L}\big\}_{n\in \NN} \text{ is compact in $L^2_{\text{loc}}(\RR^+,\RR)$, for every $L\in \mathscr{N}$.} $$
	Therefore, there exists $v^L\in L^2_{\text{loc}}(\RR^+,\RR)$, for every $L\in \mathscr{N}$ such that 
	$$ g^{n^L_{k},L}\to v^L\, \, \text{ in $L^2_{\text{loc}}(\RR^+,\RR)$ as $k\to\infty$.}$$
	Clearly, the subsequence $n^L_{k}$ depends on $L$. Since $L$ varies over a countable set, we can choose a single subsequence along which above convergence occurs for every $L\in \mathscr{N}$ which from now on is again denoted by $n$ for simplicity. 
	
	We will now study $\sqrt{n}(1-\phi^n_t)$. First note that for any  $L\in\mathscr{N}$, 
	\begin{align*}
	\big\{h^{n, L}_t	\doteq g^{n,-1}_t+\sum_{l=0}^{ L} g^{n,l}_t\big\}_{n\in \NN} \text{ is compact in $L^2_{\text{loc}}(\RR^+,\RR)$.}
	\end{align*}
	Indeed, to see this observe that \begin{align*}\frac{1}{T}\int_0^T |h^{n, L}_t|^2dt &=  \frac{1}{T}\int_0^T |g^{n,-1}_t|^2 dt+\sum_{l=0}^{ L} \frac{1}{T}\int_0^T |g^{n,l}_t|^2 dt\leq \frac{M}{c_n}.
	\end{align*}
	Here, $c_n\doteq \min\{\frac{1}{2}, \alpha(1+\frac{ L}{\sqrt n}\}$. To arrive at the inequality above, we used~\eqref{eq-comp-lem-1} and the fact that $\alpha(\cdot)$ is a strictly decreasing function on $\RR^+$. This concludes the compactness of $\{h^{n, L}\}_{n\in \NN}$ for every $ L\in \mathscr{N}$ as $\inf_{n} c_n=\frac{1}{2}$. Moreover, we have
	$$ \sup_{L\in \mathscr{N}}\limsup_{n\to\infty} \sup_{T>0} \frac{1}{T}\int_0^T |h_t^{n,L}|^2 dt\leq 2M$$ and 
	$h^{n,L}$ converges in $L^2_{\text{loc}}(\RR^+,\RR)$ as $n\to\infty$, to $v^{-1}+\sum_{l=0}^L v^l$, for every $L\in \mathscr{N}$.
	 From the weak lower semi-continuity of the norm, we have
	$$ \sup_{L\in \mathscr{N}} \sup_{T>0}\left\{\frac{1}{T}\int_0^T|v^{-1}_t|^2 dt + \sum_{l=0}^L\frac{1}{T}\int_0^T |v^l_t|^2dt\right\}\leq \sup_{L\in \mathscr{N}}\liminf_{n\to\infty} \sup_{T>0} \frac{1}{T}\int_0^T |h^{n,L}_t|^2dt \leq 2M . $$
	From here, it is clear that 
	\begin{align*}  \sup_{T>0}\frac{1}{T}\int_0^T|v^{*}_t|^2 dt \leq \liminf_{n\to\infty} \sup_{T>0}\left\{\frac{1}{T}\int_0^T|g^{n,-1}_t|^2 dt + \sum_{l=0}^\infty\frac{1}{T}\int_0^T |g^{n,l}_t|^2dt\right\} \leq 2M . \end{align*}
		Here, $v^*\doteq v^{-1}+\sum_{L=0}^\infty v^L.$  From the above, observe that 
		$$\left\{ \sqrt n (1-\phi^n_t)=g^{n,-1}_t+ \sum_{l=0}^\infty g^{n,l}_t\right \}_{n\in \NN} \text{ lies in $L^2_{\text{loc}}(\RR^+,\RR)$, for large $n$}.  $$ 
		To be concise, we have shown that 
		\begin{align} \label{eq-low-semi-compact} \sup_{T>0}\frac{1}{T}\int_0^T|v^{*}_t|^2 dt \leq \liminf_{n\to\infty} \sup_{T>0}\frac{1}{T}\int_0^T|\sqrt{n}(1-\phi^n_t)|^2 dt  \leq 2M . \end{align}
	
	Recalling that we had started with an arbitrary subsequence of $n$, we can replace  $\liminf$	with $\limsup$ in the above display to get $$\limsup_{n\to\infty} \sup_{T>0}\frac{1}{T}\int_0^T|\sqrt{n}(1-\phi^n_t)|^2 dt  \leq 2M. $$
	This finishes the proof.
	\end{proof}

	 \section{ Proof of~Theorem \ref{thm-fclt-poisson}}
	 
	 In the following, we show that 
	 \begin{align}\label{eq-fclt-low} 
	 \limsup_{n\to\infty}\frac{1}{T} \log\E[e^{TG(\widetilde N^n)}]
	 \leq  \sup_{w\in \cA}\E\Big[ G\Big(W+\int_0^\cdot w_t dt\Big) -\frac{\lambda}{2T}\int_0^T|w_t|^2 dt\Big]= \frac{1}{T}\log \E[e^{TG(W)}]
	 \end{align}
	and 
	\begin{align}\label{eq-fclt-up} 
	\liminf_{n\to\infty}\frac{1}{T} \log\E[e^{TG(\widetilde N^n)}]\geq  \sup_{w\in \cA}\E\Big[ G\Big(W+\int_0^\cdot w_t dt\Big) -\frac{\lambda}{2T}\int_0^T|w_t|^2 dt\Big]= \frac{1}{T}\log \E[e^{TG(W)}].
	\end{align}
Recall that $W$ is a one-dimensional Brownian motion and the associated set $\cA$ is as defined in Theorem~\ref{thm-var-rep-BM-gen}, but for $1$--dimensional case. In the rest of the proof, whenever a Brownian motion is involved, we denote it by $W$.

	\noindent {\bf Proof of~\eqref{eq-fclt-low}}: 
	Fix $\delta>0$. Then from~\eqref{eq-var-rep-poisson-gen-bound}, we have
	\begin{align}\label{eq-app-up-fclt}
	\frac{1}{T}\log \E\big[ e^{TG(\widetilde N^n)}\big]\leq \sup_{\phi \in \widetilde \cE_{M}} \E\Big[ G\Big( M^{n,\phi}-\lambda \int_0^\cdot\sqrt{n} (1-\phi_t)dt\Big)- \frac{\lambda n}{T}\int_0^T\varkappa(\phi_t)dt\Big]+\delta\,. 
	\end{align} 
We remark that $\phi$ in the above equation is a $\widetilde \cE_M$--valued random variable.	Since $\phi\in \widetilde \cE_M$, it satisfies 
	$$\sup_{n\in \NN} \frac{\lambda n}{T}\int_0^T \varkappa(\phi_{t})dt \leq M.$$
	Now choose a $\delta$--optimal (corresponding to~\eqref{eq-app-up-fclt}) $\phi^n\in \widetilde \cE_M$, for every $n$. Then using Lemma~\ref{lem-compact}, we can conclude that $$ \limsup_{n\to\infty} \frac{1}{T}\int_0^T|\sqrt{n} (1-\phi^n_t)|^2 dt \leq 2M.$$ 
	This clearly implies that  there exists a subsequence (denoted by $n_k$) such that  the family of random variables $\{\sqrt{n_k} (1-\phi^{n_k})\}_{k\in \NN}$ is tight in $L^2([0,T],\RR)$ when equipped with weak$^*$ topology (denoted by $L^{2,*}_T$ from now on).
	From the tightness of $\{\sqrt{n_k} (1-\phi^{n_k})\}_{k\in \NN}$ in $L^{2,*}_T$, we know that there exists  a $L^{2,*}_T$--valued random variable $w$ 	such that $\{\sqrt{n_k} (1-\phi^{n_k})\}_{k\in \NN}$ converges weakly to $w$, along that subsequence. 
		
		Define the following family of $\frC_T$--valued random variables: $m_t^{n_k}\doteq \int_0^t \sqrt{n_k} (1-\phi^{n_k}_s)ds$. We show that  $\{m^{n_k}\}_{k\in \NN}$ is a pre-compact set of  $\frC_T$. To that end, first observe that \begin{align*}\sup_{k\in \NN}\sup_{t\in [0,T]}|m^{n_k}_t|&\leq \sup_{k\in \NN}\sup_{t\in [0,T]} t \sqrt{\frac{T}{tT}\int_0^t|\sqrt{n_k}(1-\phi^{n_k}_s)|^2ds}\\
		&\leq \sup_{k\in \NN}\sup_{t\in [0,T]} T\sqrt{\frac{1}{T}\int_0^t|\sqrt{n_k}(1-\phi^{n_k}_s)|^2ds}\leq T\sqrt M.\end{align*}
		To show equicontinuity, for $0\leq s<t\leq T,$ we consider
		\begin{align*} m^{n_k}_t-m^{n_k}_s= \int_s^t \sqrt{n_k} (1-\phi^{n_k}_u)du.\end{align*} 
		Using Cauchy-Schwartz inequality, we have
		\begin{align*}
		|m^{n_k}_t-m^{n_k}_s|\leq \sqrt T \sqrt{t-s}  \sqrt{\frac{1}{T}\int_0^T |\sqrt{n_k}(1-\phi^{n_k}_u)|^2du}\leq \sqrt{MT} (t-s)^{\frac{1}{2}}.
		\end{align*}
		This proves that  $\{m^{n_k}\}_{k\in \NN}$ is a pre-compact set of $\frC_T$ and hence, is tight in $\frC_T$. 	
		
		We now analyze $\{M^{n,\phi^n}\}_{n\in \NN}$.  To do this, we first show that $\int_0^\cdot \phi^n_t dt \to \fre(\cdot)$ in $\frC_T$. We begin by observing that $\phi^n\in L^{2,*}_T$ and $\{\phi^n\}_{n\in \NN}$ is convergent in $L^{2,*}_T$ with $\phi^*\equiv 1$ being the limit. Therefore arguing similarly as above will imply boundedness (uniformly in $n$) and equicontinuity of $\{j_n(\cdot)\doteq \int_0^\cdot \phi^n_tdt\}_{n\in \NN}$ (hence, tightness of $\{j_n\}_{n\in \NN}$ in $\frC_T$).  It is clear that the limit point of $\{j_n\}_{n\in \NN}$ is $\fre(\cdot)$.	Therefore, using martingale central limit theorem and random change of time lemma (\cite[Pg. 151]{billingsley1999}), we can conclude that $M^{n,\phi^n}$ converges weakly to a Brownian motion $W$ in $\frD_T$, along a subsequence. 
				\begin{align*}
		\limsup_{k\to\infty}	\frac{1}{T}\log \E\big[ e^{TG(\widetilde N^{n_k})}\big]&\leq \limsup_{k\to\infty} \E\Big[ G\Big( M^{n_k,\phi^{n_k}}+ \lambda\int_0^\cdot \sqrt{n_k} (\phi^{n_k}_t-1)dt \Big)- \frac{\lambda n_k}{T}\int_0^T\varkappa(\phi^{n_k}_t)dt\Big]+2\delta\\
	&	\leq \E\Big[G\Big(W+\int_0^\cdot wdt\Big)\Big]-\liminf_{k\to\infty}\E\Big[\frac{\lambda n_k}{T}\int_0^T\varkappa(\phi^{n_k}_t)dt\Big] + 2\delta\,.
		\end{align*}
		In the above, we have used the fact that $\{\big(M^{n_k,\phi^{n_k}}, \int_0^\cdot \sqrt n_k (1-\phi^{n_k}_t)dt\big)\}_{k\in \NN}$ converges weakly in $\frD_T\times  \frC_T$ to $(W, \int_0^\cdot w_tdt)$ and the continuous mapping theorem. To simplify $$\liminf_{k\to\infty}\E\Big[\frac{\lambda n_k}{T}\int_0^T\varkappa(\phi^{n_k}_t)dt\Big],$$ we use~\eqref{eq-low-semi-compact} to obtain 
		\begin{align*}
		\liminf_{k\to\infty}\E\Big[\frac{\lambda n_k}{T}\int_0^T\varkappa(\phi^{n_k}_t)dt\Big]&\geq \liminf_{k\to\infty} \frac{1}{2T}\int_0^T |\sqrt{n_k}(1-\phi^{n_k}_t)|^2dt,\\
		&\geq \frac{1}{2T}\int_0^T |w_t|^2dt.
		\end{align*}
		Lower semicontinuity of the functional $f(w)\doteq \frac{1}{T} \int_0^T |w_t|^2dt$ in $L^{2,*}_T$ is used to get the last inequality above. Therefore, from arbitrariness of $\delta>0$, we have shown that 
			\begin{align*}
	\limsup_{k\to\infty}	\frac{1}{T}\log \E\big[ e^{TG(\widetilde N^{n_k})}\big]\leq \E\Big[G\Big(W+\int_0^\cdot w_tdt\Big)-\frac{1}{2T}\int_0^T |w_t|^2dt\Big]\leq \frac{1}{T}\E\big[e^{TG(W)}\big]\,. 
		\end{align*}

		\noindent {\bf Proof of~\eqref{eq-fclt-up}}:
		Fix $\delta>0$ and a one-dimensional Brownian motion $W$. Then choose a $w^*\in \cA$ such that 
		$$ \sup_{w\in \cA}\E\Big[ G\Big(W+\int_0^\cdot w_t dt\Big) -\frac{\lambda}{2T}\int_0^T|w_t|^2 dt\Big]\leq\E\Big[ G\Big(W+\int_0^\cdot w^*_t dt\Big) -\frac{\lambda}{2T}\int_0^T|w^*_t|^2 dt\Big]+\delta$$
		and define $\phi^n_t\doteq 1-\frac{w^*_t}{\sqrt{n}}$. Clearly, $\{\sqrt{n}(1-\phi^n)\}_{n\in\NN}$ is weakly convergent family of random variables in $L^{2,*}_T$ with limit being $w^*$. Moreover, it is also clear that $M^{n,\phi^n}$ converges weakly to $W$   in $\frD_T$. Now we obtain
		\begin{align*}
			\liminf_{n\to\infty}\frac{1}{T} \log\E\big[e^{TG(\widetilde N^n)}\big]&\geq 	\liminf_{n\to\infty} \E\Bigg[ G\Big( M^{n,\phi^n}+\frac{ \lambda n\int_0^\cdot \phi^n_tdt-\lambda n\fre(\cdot) }{\sqrt{n}}\Big)- \frac{\lambda n}{T}\int_0^T\varkappa(\phi^n_t)dt\Bigg],\\
			&\geq \E\Big[ G\Big( W+\int_0^\cdot w^*_tdt\Big)- \frac{\lambda }{2T}\int_0^T|w^*_t|^2dt\Big],\\
			&\geq \frac{1}{T} \log\E[e^{TG(W)}] -\delta\,. 
		\end{align*}
		From arbitrariness of $\delta>0$, we have the result.

\section*{Acknowledgement}
We thank the referees for their helpful comments and suggestions which have helped us improve the quality and presentation of the paper. This work is funded by  the NSF Grant DMS 2216765. 

\bibliographystyle{abbrv}
\bibliography{MulticlassV-RS-control}

\end{document}